\documentclass[english,american]{amsart}
\usepackage[T1]{fontenc}
\usepackage[latin9]{inputenc}
\usepackage{textcomp}
\usepackage{mathrsfs}
\usepackage{amsbsy}
\usepackage{amstext}
\usepackage{amsthm}
\usepackage{amssymb}
\usepackage{xfrac}

\makeatletter

\providecommand{\tabularnewline}{\\}

\numberwithin{equation}{subsection}
\numberwithin{figure}{subsection}
  \theoremstyle{plain}
  \newtheorem*{thm*}{\protect\theoremname}
\theoremstyle{plain}
\newtheorem{thm}{\protect\theoremname}[subsection]
  \theoremstyle{plain}
  \newtheorem{lem}[thm]{\protect\lemmaname}
  \theoremstyle{plain}
  \newtheorem{cor}[thm]{\protect\corollaryname}
  \theoremstyle{definition}
  \newtheorem{defn}[thm]{\protect\definitionname}
  \theoremstyle{remark}
  \newtheorem*{rem*}{\protect\remarkname}
  \theoremstyle{definition}
  \newtheorem*{example*}{\protect\examplename}
  \theoremstyle{plain}
  \newtheorem{prop}[thm]{\protect\propositionname}

\usepackage[all]{xy}

\makeatother

\usepackage{babel}
  \addto\captionsamerican{\renewcommand{\corollaryname}{Corollary}}
  \addto\captionsamerican{\renewcommand{\definitionname}{Definition}}
  \addto\captionsamerican{\renewcommand{\examplename}{Example}}
  \addto\captionsamerican{\renewcommand{\lemmaname}{Lemma}}
  \addto\captionsamerican{\renewcommand{\propositionname}{Proposition}}
  \addto\captionsamerican{\renewcommand{\remarkname}{Remark}}
  \addto\captionsamerican{\renewcommand{\theoremname}{Theorem}}
  \addto\captionsenglish{\renewcommand{\corollaryname}{Corollary}}
  \addto\captionsenglish{\renewcommand{\definitionname}{Definition}}
  \addto\captionsenglish{\renewcommand{\examplename}{Example}}
  \addto\captionsenglish{\renewcommand{\lemmaname}{Lemma}}
  \addto\captionsenglish{\renewcommand{\propositionname}{Proposition}}
  \addto\captionsenglish{\renewcommand{\remarkname}{Remark}}
  \addto\captionsenglish{\renewcommand{\theoremname}{Theorem}}
  \providecommand{\corollaryname}{Corollary}
  \providecommand{\definitionname}{Definition}
  \providecommand{\examplename}{Example}
  \providecommand{\lemmaname}{Lemma}
  \providecommand{\propositionname}{Proposition}
  \providecommand{\remarkname}{Remark}
  \providecommand{\theoremname}{Theorem}
\providecommand{\theoremname}{Theorem}

\begin{document}
\selectlanguage{english}%
\global\long\def\End{\mathbf{\mathrm{End}}}

\global\long\def\Hom{\mathrm{Hom}}

\global\long\def\div{\mathrm{div}}

\global\long\def\Lie{\mathrm{Lie}}

\global\long\def\ord{\mathrm{ord}}

\global\long\def\no{\mathrm{no}}

\global\long\def\fol{\mathrm{fol}}

\global\long\def\Fr{\mathrm{Fr}}

\global\long\def\Ver{\mathrm{Ver}}

\global\long\def\Spec{\mathrm{Spec}}

\global\long\def\rk{\mathrm{rk}}
\selectlanguage{american}%

\title{Theta operators on unitary Shimura varieties}
\selectlanguage{english}%

\author{Ehud de Shalit and Eyal Z. Goren}

\date{December 11, 2017}

\keywords{Shimura variety, theta operator, modular form}

\subjclass[2000]{11G18, 14G35}

\address{Ehud de Shalit, Hebrew University of Jerusalem, Israel}

\address{ehud.deshalit@mail.huji.ac.il}

\address{Eyal Z. Goren, McGill University, Montr\'eal, Qu\'ebec, Canada}

\address{eyal.goren@mcgill.ca}
\begin{abstract}
We define a theta operator on $p$-adic vector-valued modular forms
on unitary groups of arbitrary signature, over a quadratic imaginary
field in which $p$ is inert. We study its effect on Fourier-Jacobi
expansions and prove that it extends holomorphically beyond the $\mu$-ordinary
locus, when applied to scalar-valued forms.
\end{abstract}

\maketitle

\tableofcontents

\section*{Introduction}

\selectlanguage{american}%
Let $E$ be a quadratic imaginary field and $p$ a prime which is inert in $E$.
The purpose of this article is to define a theta operator $\Theta$
for $p$-adic vector-valued modular forms on unitary Shimura varieties
of arbitrary signature associated with the extension $E/\mathbb{Q}$, and prove some
fundamental results concerning it. Specifically, we prove a formula for the action
of $\Theta$ in terms of Fourier-Jacobi expansions (Theorem \ref{thm:theta on q expansions}).
We also prove that $\Theta$ extends to a holomorphic operator outside
the $\mu$-ordinary locus, when acting on scalar-valued modular forms
in characteristic $p$ (Theorems \ref{thm:holomorphic extension of theta m less n}
and \ref{thm:m=00003Dn_holmorphicity}). 

When the prime $p$ is \emph{split} in $E$, general points on the
special fiber of the Shimura variety parametrize ordinary abelian
varieties. A theta operator, and a whole array of differential operators
derived from it, were defined in this context in Eischen's thesis
\cite{Ei}. Her construction was generalized in \cite{E-F-M-V} to
unitary Shimura varieties associated with a general CM field, but
still under the ordinariness assumption. In their work, these authors
circumvent the study of $\Theta$ on Fourier-Jacobi expansions by
expressing it in Serre-Tate coordinates at CM points. 

``Ordinariness'' is a strong assumption. Over the ordinary locus,
it provides a unit-root splitting of the Hodge filtration in the cohomology
of the universal abelian variety. This allows one to extend Katz's
approach to $\Theta$ \cite{Ka2}. The unit-root splitting serves
as a $p$-adic replacement for the Hodge decomposition over the complex
numbers, which underlies the construction of similar $C^{\infty}$-differential
operators of Ramanujan and Maass-Shimura \cite{Sh}, \S III.

In \cite{dS-G1}, we defined a $\Theta$-operator on unitary modular
forms of signature $(2,1)$ and determined its effect on $q$-expansions,
for $p$ \emph{inert} in the quadratic imaginary field $E$. The main
obstacle in this case was that the abelian variety parametrized by
a general point of the special fiber of the Shimura varietyis is not ordinary 
anymore, but so-called $\mu$-ordinary,
and its cohomology does not admit a unit-root splitting. Our approach
there, adopted also in the present paper, is to make systematic use
of Igusa varieties; we first define the theta operator on them, and
show that it descends to the Shimura variety. 

Recently, we have learned of the work of Ellen Eischen and Elena Mantovan
\cite{E-M} in which they construct the same differential operators
in the $\mu$-ordinary ($p$ inert) case. Their method is closer to
the original idea of Katz, but they replace the unit-root splitting
by slope filtration splitting of $F$-crystals. Their construction
is more general than ours, as it applies to unitary Shimura varieties
associated with a general CM field. They apply their differential
operators to the study of $p$-adic families of modular forms in the
spirit of Serre, Katz and Hida. Their work should have applications
to questions of over-convergence, construction of $p$-adic
$L$-functions and Iwasawa theory.
However, the issues addressed in the present paper,
the effect of $\Theta$ on Fourier-Jacobi expansions and its holomorphic
extension beyond the $\mu$-ordinary locus, are not considered there. 

\medskip{}

We now provide some background and motivation for the study undertaken
in this paper. The theta operator for elliptic modular forms is related
to an operator already defined by Ramanujan. On $q$-expansions it is given by
\[
f=\sum_{n}a_{n}q^{n}\mapsto\Theta(f)=\sum_{n}na_{n}q^{n}.
\]
Over the complex numbers,
this operator does not preserve the space of holomorphic modular forms.
However, viewed at the level of $q$-expansions for $p$-adic, or
mod $p$, modular forms, it does, at least when one has reasonable
demands: in characteristic $p$ one has to multiply $\Theta(f)$ by
$h$, the Hasse invariant, which is a modular form of weight $p-1$
vanishing outside the ordinary locus; $p$-adically one has to be
content with working merely over the ordinary locus. 

These aspects were present from the very start in the work of Swinnerton-Dyer
and Serre \cite{Se1,Se2,Sw-D}. In fact, already in \cite{Se1}, motivated
by relation to Galois representations, Serre investigates the notion
of \emph{filtration}. The filtration of a $q$-expansion
of a mod $p$ modular form is the minimal weight in which one may
find a modular form with that $q$-expansion; one is interested in
its variation under applications of $\Theta$, which at the level
of Galois representations corresponds to a cyclotomic twist. Following
closely on the heels of these developments, Katz gave a geometric
construction of $\Theta$ on (essentially) all modular curves with
good reduction at $p$ in \cite{Ka2}. 

Not much later, Jochnowitz \cite{Joc} studied $\Theta$-\emph{cycles}. The
basic idea is simple. If $g=\Theta(f)$ has filtration $w_{0}$, 
the series of filtrations $w_{i}$ of $\Theta^{i}(g)$, $i=0,1,\dots,p-1$,
is a collection of weights that is generally increasing, but not always,
because $w_{p-1}=w_{0}$. The question of the variation of the filtration
along the cycles is interesting and has important applications. See
\cite{Gr,Joc}. Further deep uses of the $\Theta$-operator to over-convergence
and classicality of $p$-adic modular forms were given in \cite{Col,C-G-J}.

Shortly after \cite{Ka2}, Katz has studied in \cite{Ka3}, \S II, such an
operator for Hilbert modular forms associated to a totally real field
$L$, and in fact enriched the theory by introducing $g=[L:\mathbb{Q}]$
basic theta operators. These operators were instrumental in his construction of $p$-adic $L$-functions for CM fields via the Eisenstein measure. In that work, as in the case of modular curves,
strong use is made of the behaviour of de Rham cohomology and the unit
root splitting over the ordinary locus. The study of these operators
was further developed by Andreatta and the second author \cite{A-G},
who constructed mod~$p$ versions of them by means of the Igusa variety,
and provided some results on filtrations, $\Theta$-cycles and relations
to cyclotomic twists. 

\medskip{}

It seemed a natural idea at that point to extend the theory of the
theta operator to other Shimura varieties of PEL type. However, two obstacles
arise: (i) The abelian variety classified by a general point of the Shimura variety in positive characteristic
may not be ordinary anymore. In particular, its de Rham cohomology
may not admit a unit root splitting. (ii) The natural definition takes
modular forms, even if scalar-valued, to vector-valued modular forms. 

Bearing in mind the Kodaira-Spencer isomorphism, which is involved
in the definition of $\Theta$, the second problem could be anticipated.
In the Hilbert modular case, it is the abundance of endomorphisms
that allows one to return to scalar-valued modular forms. In spite
of these difficulties, progress has been made on other Shimura varieties:
As Eischen had already remarked in her thesis, her construction generalizes
almost immediately to the symplectic case. Panchishkin and Courtieu
discussed similar operators for Siegel modular forms in \cite{Co-Pa} \S\S 2-3,\cite{Pan}.
For different aspects in the symplectic case see the papers by B\"ocherer-Nagaoka
\cite{B-N} and Ghitza-McAndrew \cite{G-M}, and additional references
therein. For other cases, see the work of Johansson \cite{Joh}.

Our construction of the theta operator via the Igusa tower was motivated by
Gross' construction in \cite{Gr}. For an application of the Igusa tower to the study of
vector-valued $p$-adic Siegel modular forms see \cite{Ich}.

\bigskip{}
The contents of this paper are as follows. Let $E$ be a quadratic
imaginary field, $p$ a rational prime that is inert in $E$ and $\kappa=\mathcal{O}_{E}/(p)$
its residue field. Let $n\geq m$ be positive integers. Fixing additional
data, one obtains a scheme $\mathcal{S}$ over $\mathcal{O}_{E,(p)}$ that parametrizes
abelian schemes with $\mathcal{O}_{E}$-action of signature $(n,m)$,
endowed with a principal polarization and level structure. Its complex points
are a union of Shimura varieties associated to
the unitary group $GU(n,m)$. Let $S\rightarrow{\rm Spec(\kappa)}$ denote
its special fiber, and let $S_{s}$ be the base change of $\mathcal{S}$
to $W_{s}=W_{s}(\kappa)$. 

In $\S$\ref{sec:Background} we collect background material and definitions,
and in particular define the type of vector-valued $p$-adic modular
forms that will be considered in this paper. Automorphic vector bundles
over $\mathcal{S}$ correspond to representations of the group $GL_{m}\times GL_{n}$,
and there are two ``basic'' vector bundles, $\mathcal{Q}$ and $\mathcal{P}$,
corresponding to the standard representations of the two blocks, from
which all others are derived.\footnote{As the Levi factor of the appropriate parabolic in $GU(n,m)_\mathbb{C}$ is $\mathbb{G}_m \times GL_{m}\times GL_{n}$ we could, in principle, take also representations that are non-trivial on the first factor. However, we will have no need for this greater generality in this paper and so, here and in the sequel, we will  consider automorphic vector bundles associated to representations of $GL_{m}\times GL_{n}$ only.} Characteristic $p$ holds its own idiosyncrasies
and there are $3$ vector bundles, denoted $\mathcal{Q},\mathcal{P}_{0}$
and $\mathcal{P}_{\mu}$, from which all $p$\emph{-adic} automorphic
vector bundles $\mathcal{E}_{\rho}$ are derived by representation-theoretic
constructions; in particular, $\rho$ refers here to a representation
of $GL_{m}\times GL_{m}\times GL_{n-m}$. We briefly explain the origin
of these vector bundles. The relative cotangent bundle of the universal
abelian variety $\mathcal{A}\rightarrow\mathcal{S}$ decomposes according
to signatures, providing us with vector bundles $\mathcal{P},\mathcal{Q}$
of ranks $n,m$, respectively. Over the $(\mu$-)ordinary locus $S_{s}^{\mathrm{ord}}$
of $S_{s}$, $\mathcal{P}$ admits a filtration $0\rightarrow\mathcal{P}_{0}\rightarrow\mathcal{P}\rightarrow\mathcal{P}_{\mu}\rightarrow0$.
The vector bundle $\mathcal{E}_{\rho}$ lives over $S_{s}^{\mathrm{ord}}$ and
is obtained by ``twisting'' $\rho$ by the triple $(\mathcal{Q},\mathcal{P}_{\mu},\mathcal{P}_0)$ 
(see $\S$\ref{subsec:Twisting} for details). A mod-$p^{s}$
modular form of weight $\rho$ is defined to be a section of $\mathcal{E}_{\rho}$
over $S_{s}^{\mathrm{ord}}$. 

In $\S$\ref{sec:Differential-operators-on} we define the Igusa tower
over $S_{s}^{\mathrm{ord}}$ and study its properties. The key fact
about the Igusa tower is that the vector bundles $\mathcal{P}_{0},\mathcal{P}_{\mu}$
and $\mathcal{Q}$ (unlike $\mathcal{P}$ !) are all \emph{canonically trivialized} over it.
To be precise, much as in Katz \cite{Ka1}, the Igusa tower is a double
limit of schemes $\{T_{t,s}|t,s\geq1\}$, where $T_{t,s}$ is a scheme
over the truncated Witt vectors $W_{s}$ of length $s$, and whenever
$t\geq s$ a trivialization as above is obtained. Consequently,
we are able to propagate, by linear algebra constructions alone, the
trivial connection $d:\mathcal{O}_{T}\rightarrow\Omega_{T/W_{s}}$
for $T=T_{t,s},t\geq s$, to a connection 
\[
\tilde{\Theta}:\mathcal{E}_{\rho}\rightarrow\mathcal{E}_{\rho}\otimes\Omega_{T/W_{s}}\cong\mathcal{E}_{\rho}\otimes\mathcal{P}\otimes\mathcal{Q},
\]
the last isomorphism stemming from the Kodaira-Spencer map. When we
follow this map by the projection $\mathcal{E}_{\rho}\otimes\mathcal{P}\otimes\mathcal{Q}\rightarrow\mathcal{E}_{\rho}\otimes\mathcal{P}_{\mu}\otimes\mathcal{Q}$,
and combine it with pull back of modular forms under $T\rightarrow S_{s}^{\mathrm{ord}}$,
we obtain an operator 
\[
\Theta:H^{0}(S_{s}^{\mathrm{ord}},\mathcal{E}_{\rho})\rightarrow H^{0}(S_{s}^{\mathrm{ord}},\mathcal{E}_{\rho}\otimes\mathcal{P}_{\mu}\otimes\mathcal{Q}).
\]
This operator can be iterated and combined with representation-theoretic
operations as discussed in the end of $\S$\ref{sec:Differential-operators-on},
to produce an array of differential operators $D_{\kappa}^{\kappa^{\prime}}$
as in \cite{E-F-M-V,E-M}. 

The initial parts of $\S$\ref{sec:Toroidal-compactifications-and} are
a review of the theory of toroidal compactifications for the case
at hand. We follow Faltings-Chai \cite{F-C}, that relies on the seminal
work of Mumford and his school, Skinner-Urban \cite{S-U}, and the
definitive volume by Lan \cite{Lan}. In particular, the reader will
find a precise explanation of the meaning of the Fourier-Jacobi expansion
of a vector-valued modular form 
\[
f=\sum_{\check{h}\in\check{H}^{+}}a(\check{h})\,q^{\check{h}}.
\]
 See $\S$\ref{subsec:Fourier-Jacobi-expansions}. In this notation our
first main theorem states the following.
\begin{thm*}
(Theorem \ref{thm:theta on q expansions}) Let $\xi$ be a rank-$m$
cusp. Let $f$ be a global section of $\mathcal{E}_{\rho}$ and $\sum_{h\in\check{H}^{+}}a(\check{h})\,q^{\check{h}}$
its Fourier-Jacobi expansion at $\xi$. Then the section $\Theta(f)$
of $\mathcal{E}_{\rho}\otimes\mathcal{P}_{\mu}\otimes\mathcal{Q}$
has the Fourier-Jacobi expansion 
\[
\Theta(f)=\sum_{\check{h}\in\check{H}^{+}}a(\check{h})\otimes\check{h}\cdot q^{\check{h}}.
\]
\end{thm*}

The analogous result for the Fourier-Jacobi expansion at a non maximally degenrate
cusp (of rank $<m$) should involve also theta operators on lower-rank Shimura varieties
acting on the coefficients. For most practical purposes, however, e.g. for a $q$-expansion
principle, rank $m$ cusps suffice.

In $\S$\ref{sec:Analytic-continuation-of} we consider the extension
of the operator $\Theta$ to the complement of the $\mu$-ordinary
locus. This we are able to do, so far, only for scalar-valued modular
forms. The proof requires a partial compactification of a particular
Igusa variety as in \cite{dS-G1}, and delicate computations with
Dieudonn{\'e} modules in the spirit of our recent work \cite{dS-G2}.
Let $\mathcal{L}=\det\mathcal{Q}$ and $k\ge0.$ 
\begin{thm*}
(Theorems \ref{thm:holomorphic extension of theta m less n}, \ref{thm:m=00003Dn_holmorphicity})
Consider the operator 
\[
\Theta:H^{0}(S^{\mathrm{ord}},\mathcal{L}^{k})\rightarrow H^{0}(S^{\mathrm{ord}},\mathcal{L}^{k}\otimes\mathcal{Q}^{(p)}\otimes\mathcal{Q}).
\]
 Then $\Theta$ extends holomorphically to an operator 
\[
\Theta:H^{0}(S,\mathcal{L}^{k})\rightarrow H^{0}(S,\mathcal{L}^{k}\otimes\mathcal{Q}^{(p)}\otimes\mathcal{Q}).
\]
 
\end{thm*}
Finally, in $\S$\ref{sec:Theta-cycles} we introduce the notion of $\Theta$-cycles
and recall interesting phenomena observed in \cite{dS-G1}. 

\medskip{}

Our paper and the work of Eischen-Mantovan suggest several directions
in which the theory can be further developed. In addition to those
mentioned in \cite{E-M} we suggest the following problems. (i) Provide
a formula for the Fourier-Jacobi expansion and the theta operator
$\Theta$ at general cusps. (ii) Study the extension of $\Theta$ to a holomorphic operator for general vector-valued
unitary modular forms. (iii) Develop a theory of mod $p$ operators,
such as $U$ and $V$ and characterize the kernel of $\Theta$ in
terms of $V$, \emph{cf. }\cite{Ka2}. (iv) Study $\Theta$-cycles
in relation to mod $p$ Galois representations. 

\bigskip{}

\textbf{Acknowledgments.} We would like to thank Ellen Eischen and
Elena Mantovan for discussions relating to the contents of both our
papers. It is a pleasure to thank G. Rosso for bringing their work
to our attention, and P. Kassaei for valuable comments. We thank the referees for very useful comments. 

Much of this
paper was written during visits to the Hebrew University and to McGill
University and it is our pleasant duty to thank these institutes for
their hospitality. This research was supported by NSERC grant 223148
and ISF grant 276/17.

\section{\label{sec:Background}Background}

\subsection{The Shimura variety}

\subsubsection{\label{subsec:Linear-algebra}Linear algebra}

We review some background and set up standard notation. Let $E$ be
a quadratic imaginary field, embedded in $\mathbb{C},$ $0\le m\le n$
and $\Lambda=\mathcal{O}_{E}^{n+m}.$ Let
\begin{equation}
I_{n,m}=\left(\begin{array}{ccc}
 &  & I_{m}\\
 & I_{n-m}\\
I_{m}
\end{array}\right)\label{eq:I_(n,m)}
\end{equation}
where $I_{l}$ is the unit matrix of size $l$, and introduce the
perfect hermitian pairing
\begin{equation}
\left(u,v\right)=\,^{t}\overline{u}I_{n,m}v\label{eq:hermitian_pairing}
\end{equation}
on $\Lambda.$ Let
\[
\boldsymbol{G}=GU(\Lambda,(,))
\]
be the group of unitary similitudes of $\Lambda$, regarded as a group
scheme over $\mathbb{Z},$ and denote by $\nu:\boldsymbol{G}\to\mathbb{G}_{m}$
the similitude character. For any commutative ring $R$
\[
\boldsymbol{G}(R)=\left\{ g\in GL_{n+m}(\mathcal{O}_{E}\otimes R)|\,\forall u,v\in\Lambda\otimes R\,\,\,\,\left(gu,gv\right)=\nu(g)(u,v)\right\} .
\]
Then $\boldsymbol{G}(\mathbb{R})=GU(n,m)$ is the general unitary
group of signature $(n,m),$ and $\boldsymbol{G}(\mathbb{C})\simeq GL_{n+m}(\mathbb{C})\times\mathbb{C}^{\times}.$

Let $\delta_{E}$ be the unique generator of the different $\mathfrak{d}_{E}$
of $E$ with $\mathrm{Im}(\delta_{E})>0.$ The \emph{polarization
pairing}
\begin{equation}
\left\langle u,v\right\rangle =Tr_{E/\mathbb{Q}}(\delta_{E}^{-1}(u,v))\label{eq:polarization}
\end{equation}
is then a perfect alternating pairing $\Lambda\times\Lambda\to\mathbb{Z}$
satisfying $\left\langle au,v\right\rangle =\left\langle u,\overline{a}v\right\rangle $
($a\in E$).

Let $p$ be an odd prime which is inert in $E$, and fix once and
for all an embedding $\overline{\mathbb{Q}}\subset\overline{\mathbb{Q}}_{p}$.
Let $E_{p}$ be the completion of $E$ and $\mathcal{O}_{p}$ its
ring of integers. As $-1$ is a norm from $E_{p}$ to $\mathbb{Q}_{p}$,
one easily checks that $\boldsymbol{G}_{/\mathcal{O}_{p}}$ is quasi-split.
In fact, over $\mathcal{O}_{p}$ the lattice $\Lambda_{p}=\mathbb{Z}_{p}\otimes\Lambda=\mathcal{O}_{p}^{n+m}$,
equipped with the hermitian form (\ref{eq:hermitian_pairing}), is
isomorphic to the same lattice equipped with the pairing $^{t}\overline{u}J_{n+m}v$,
where by $J_{l}$ we denote the matrix with $1$'s on the anti-diagonal
and $0$'s elsewhere. This will be useful later.

\medskip{}

If $R$ is an $\mathcal{O}_{E,(p)}$-algebra then any $R$-module
$M$ endowed with a commuting $\mathcal{O}_{E}$- action decomposes
according to types,
\[
M=M(\Sigma)\oplus M(\overline{\Sigma}),
\]
where $M(\Sigma)$ is the $R$-submodule on which $\mathcal{O}_{E}$
acts via the canonical homomorphism $\Sigma:\mathcal{O}_{E}\hookrightarrow\mathcal{O}_{E,(p)}\to R,$
while $M(\overline{\Sigma})$ is the part on which it acts via the
conjugate homomorphism $\overline{\Sigma}$. Indeed, it is enough
to decompose $\mathcal{O}_{E}\otimes R=R(\Sigma)\times R(\overline{\Sigma})$
as an $\mathcal{O}_{E}$-algebra. The same notation will be applied
to coherent sheaves with $\mathcal{O}_{E}$-action on schemes defined
over $\mathcal{O}_{E,(p)}.$

We denote by $\kappa$ the field $\mathcal{O}_{E}/p\mathcal{O}_{E}$
of $p^{2}$ elements. 

\subsubsection{The Shimura variety and the moduli problem}

Fix an integer $N\ge3$ relatively prime to $p.$ Let $\mathbb{A}=\mathbb{R}\times\mathbb{A}_{f}$
be the ad{\'e}le ring of $\mathbb{\mathbb{Q}}$, where $\mathbb{A}_{f}=\mathbb{Q}\cdot\mathbb{\widehat{Z}}$
are the finite ad{\'e}les. Let $K_{f}\subset\boldsymbol{G}(\mathbb{\widehat{Z}})$
be an open subgroup of the form $K_{f}=K^{p}K_{p},$ where $K^{p}\subset\boldsymbol{G}(\mathbb{A}^{p})$
is the principal congruence subgroup of level $N$, and
\[
K_{p}=\boldsymbol{G}(\mathbb{Z}_{p})\subset\boldsymbol{G}(\mathbb{Q}_{p}),
\]
which is a hyperspecial maximal compact subgroup at $p.$ Let $K_{\infty}\subset\boldsymbol{G}(\mathbb{R})$
be the stabilizer of the negative definite subspace spanned by $\{-e_{i}+e_{n+i};\,1\le i\le m\}$
in $\Lambda_{\mathbb{R}}=\mathbb{C}^{n+m}$, where $\left\{ e_{i}\right\} $
stands for the standard basis. This $K_{\infty}$ is a maximal compact-modulo-center
subgroup, isomorphic to $G(U(m)\times U(n)).$ By $G(U(m)\times U(n))$
we mean the pairs of matrices $(g_{1},g_{2})\in GU(m)\times GU(n)$
having the same similitude factor. Let $K=K_{\infty}K_{f}\subset\boldsymbol{G}(\mathbb{A})$
and $\mathfrak{X}=\boldsymbol{G}(\mathbb{R})/K_{\infty}.$

To the Shimura datum $(\boldsymbol{G},\mathfrak{X})$ and the level subgroup $K$ there is associated
a Shimura variety $Sh_{K}$. It is a quasi-projective non-singular
variety of dimension $nm$ defined over $E$. If $m=n$ the Shimura
variety may even be defined over $\mathbb{Q}$, but we still denote
by $Sh_{K}$ its base-change to $E$. The complex points of $Sh_{K}$
are identified, as a complex manifold, with
\[
Sh_{K}(\mathbb{C})=\boldsymbol{G}(\mathbb{Q})\backslash\boldsymbol{G}(\mathbb{A})/K.
\]

Following Kottwitz \cite{Ko} we define a scheme $\mathcal{S}$ over
$\mathcal{O}_{E,(p)}$. This $\mathcal{S}$ is a fine moduli space
whose $R$-points, for every $\mathcal{O}_{E,(p)}$-algebra $R$,
classify isomorphism types of tuples $\underline{A}=(A,\iota,\phi,\eta)$
where
\begin{itemize}
\item $A$ is an abelian scheme of dimension $n+m$ over $R$.
\item $\iota:\mathcal{O}_{E}\hookrightarrow\End(A)$ has signature $(n,m)$
on the Lie algebra of $A$.
\item $\phi:A\overset{\sim}{\to}A^{t}$ is a principal polarization whose
Rosati involution induces $\iota(a)\mapsto\iota(\overline{a})$ on
the image of $\iota$.
\item $\eta$ is an $\mathcal{O}_{E}$-linear full level-$N$ structure
on $A$ compatible with $(\Lambda,\left\langle .,.\right\rangle )$
and $\phi$ (\cite{Lan} 1.3.6).
\end{itemize}
See \cite{Lan} \S1.4  for the comparison of the various languages used to
define the moduli problem. 

The generic fiber $\mathcal{S}_{E}$ of $\mathcal{S}$ is, in general,
a union of \emph{several} Shimura varieties, one of which is $Sh_{K}.$
This is due to the failure of the Hasse principle for $\boldsymbol{G}$,
which can happen when $m+n$ is odd (\cite{Ko} $\S$7). We also remark
that the assumption $N\ge3$ could be avoided if we were willing to
use the language of stacks. As this is not essential to the present
paper, we keep the scope slightly limited for the sake of clarity.

As shown by Kottwitz, $\mathcal{S}$ is \emph{smooth} of relative
dimension $nm$ over $\mathcal{O}_{E,(p)}.$

\subsubsection{The universal abelian variety and its $p$-divisible group\label{subsec:The-universal-abelian}}

By virtue of the moduli problem which it represents, $\mathcal{S}$
carries a universal abelian scheme $\mathcal{A}_{/\mathcal{S}}$ equipped
with a PEL structure as above. Let
\[
S=\mathcal{S}\times_{\Spec(\mathcal{O}_{E,(p)})}\Spec(\kappa)
\]
be the special fiber of $\mathcal{S}.$ Recall that for any geometric
point $x:\Spec(k)\to S$ the $p$-divisible group of $A=\mathcal{A}_{x}$
carries a canonical filtration by $p$-divisible groups
\begin{equation}
\mathrm{Fil^{0}=}A[p^{\infty}]\supset\mathrm{Fil^{1}}=A[p^{\infty}]^{0}\supset\mathrm{Fil}^{2}=A[p^{\infty}]^{\mu}\supset0,\label{eq:ordinary_filtration}
\end{equation}
where $gr^{2}=A[p^{\infty}]^{\mu}$ is multiplicative, $gr^{1}=A[p^{\infty}]^{0}/A[p^{\infty}]^{\mu}$
is local-local and $gr^{0}=A[p^{\infty}]/A[p^{\infty}]^{0}$ is {\'e}tale.
Over $\Spec(k)$ this filtration is even split, i.e. $A[p^{\infty}]$
is uniquely expressible as a product of multiplicative, local-local
and {\'e}tale $p$-divisible groups, but this fact is special for algebraically
closed (or perfect) fields, while a filtration like (\ref{eq:ordinary_filtration})
often exists over more general bases.

The special fiber $S$ contains an open dense subset
called the $\mu$-ordinary locus, \emph{cf. \cite{We1} }and\emph{
}\cite{Mo1}, Theorem 3.2.7, which we denote $S^{\mathrm{ord}}$. It is characterized by the fact that
for any geometric point $x$ of $S,$ $x$ lies in $S^{\mathrm{ord}}$
if and only if the height of $A[p^{\infty}]^{\mu}$ is $2m,$ which
is as large as it can get. Equivalently, the Newton polygon of $A[p^{\infty}]$
has slopes $0,1/2$ and $1$ with horizontal lengths $2m,$ $2(n-m)$
and $2m$ respectively, which is as low as it can get. In fact, Wedhorn
and Moonen show that the isomorphism type of $A[p^{\infty}]$, as
a polarized $\mathcal{O}_{E}$-group, is the same for all $x\in S^{\mathrm{ord}}(k)$:
\[
A[p^{\infty}]\simeq(\mathfrak{d}_{E}^{-1}\otimes\mu_{p^{\infty}})^{m}\times\mathfrak{G}_{k}^{n-m}\times(\mathcal{O}_{E}\otimes(\mathbb{Q}_{p}/\mathbb{Z}_{p}))^{m}.
\]
Here $\mathfrak{G}_{k}$ is the $p$-divisible group denoted by $G_{\text{\textonehalf,\textonehalf}}$ 
in the Dieudonn{\'e}-Manin classification. It is the unique height-2 one-dimensional
connected $p$-divisible group over~$k$. It is well-known that the
ring $\mathcal{O}_{p}$ acts as endomorphisms of $\mathfrak{G}_{k}$.
We normalize this action so that the induced action of $\mathcal{O}_{p}$
on the Lie algebra of $\mathfrak{G}_{k}$ is via $\Sigma:\mathcal{O}_{p}\twoheadrightarrow\kappa\subset k$,
and this pins down $\mathfrak{G}_{k}$ as an $\mathcal{O}_{E}$-group
up to isomorphism. We polarize it fixing an isomorphism of $\mathfrak{G}_{k}$
with its Serre dual. The appearance of the inverse different in the
first factor is a matter of choice, and is meant to allow a more natural
way to write the Weil pairing between the first and last factors,
namely
\[
\left\langle a\otimes x,b\otimes y\right\rangle =Tr_{E/\mathbb{Q}}(\overline{a}b)\left\langle x,y\right\rangle .
\]

Over $S^{\mathrm{ord}}$, a filtration like (\ref{eq:ordinary_filtration})
exists globally, but is far from being split now (\textit{cf.} \cite{dS-G3}, Proposition 2.10). Nevertheless, its
graded pieces are, locally in the pro-{\'e}tale topology, isomorphic to
the constant $p$-divisible groups $(\mathfrak{d}_{E}^{-1}\otimes\mu_{p^{\infty}})^{m},$
$\mathfrak{G}_{k}^{n-m}$ and $(\mathcal{O}_{E}\otimes(\mathbb{Q}_{p}/\mathbb{Z}_{p}))^{m},$
and the isomorphisms can be taken to respect the endomorphisms and
the polarization. This is well-known for $gr^{0}$ and $gr^{2}.$
For $gr^{1}$ it follows from the rigidity of isoclinic Barsotti-Tate
groups with endomorphisms, namely from the fact that the universal
deformation ring of $(\mathfrak{G}_{k},\iota)$ where $\iota:\mathcal{O}_{p}\hookrightarrow\mathrm{End}(\mathfrak{G}_{k})$,
is $W(k)$ (\cite{Mo1}, Corollary 2.1.5, see \emph{loc.cit} $\S$3.3.1
for the polarization). This result implies that for any geometric
point $x\in S^{\mathrm{ord}}(k)$, $gr^{1}\mathcal{A}[p^{\infty}]$
becomes isomorphic over $\mathcal{\hat{O}}_{S,x}$ to $\mathfrak{G}_{k}^{n-m}$,
with its additional structures of endomorphisms and polarization.
By Artin's approximation theorem \cite{Artin} they become isomorphic already over
the strict henselization $\mathcal{O}_{S,x}^{\mathrm{sh}}$, which
means that they are locally isomorphic in the pro-{\'e}tale topology.

\subsubsection{\label{subsec:basic}The basic vector bundles on $S$}

The Hodge bundle $\omega=\omega_{\mathcal{A}/\mathcal{S}}$ is the
pull-back via the zero section $e_{\mathcal{A}}:\mathcal{S}\to\mathcal{A}$
of the relative cotangent sheaf $\Omega_{\mathcal{A}/\mathcal{S}}$
of the universal abelian scheme. It decomposes as
\[
\omega=\omega(\Sigma)\oplus\omega(\overline{\Sigma})=\mathcal{P}\oplus\mathcal{Q}
\]
according to types. Thus, $\mathrm{rk}(\mathcal{P})=n$ and $\mathrm{rk}(\mathcal{Q})=m.$
\begin{lem}
\label{lem:detP=00003DdetQ}The line bundles $\det(\mathcal{P})$
and $\det(\mathcal{Q})$ are isomorphic over $\mathcal{S}$.
\end{lem}

\begin{proof}
The proof is similar to \cite{dS-G3}, Proposition 1.3. Automorphic
vector bundles over the generic fiber $\mathcal{S}_{E}$ correspond
functorially to representations of the group $GL_{m}\times GL_{n},$
as discussed below in $\S$\ref{subsec:auotomorpic _vector_bundles_mod_p}.
The vector bundles $\det(\mathcal{Q})$ and $\det(\mathcal{P})$ correspond
to the determinant of $GL_{m}$ and the \emph{inverse} of the determinant
of $GL_{n}$. Their ratio therefore corresponds to the determinant
of $GL_{m}\times GL_{n}.$ If the level subgroup $K$ is small enough,
as we always assume, then the arithmetic group by which we divide
the symmetric space to get a complex uniformization of every connected
component of $\mathcal{S}_{\mathbb{C}}$ is contained in $SU(n,m).$
This means that over $\mathbb{C},$ the automorphic line bundle corresponding
to $\det$ is trivial, hence $\det(\mathcal{P})\simeq\det(\mathcal{Q}).$
From this it is easy to get the claim even over the base $\mathcal{O}_{E,(p)}.$
We stress that we do not know a direct moduli-theoretic proof of the
claim in the lemma, and we do not know if the particular isomorphism
supplied by the complex analytic uniformization is defined over $\overline{\mathbb{Q}}$.
See however Corollary \ref{cor:detP=00003DdetQ_muduli} below.
\end{proof}
Over the special fiber $S$ we have the Verschiebung homomorphism
$V:\omega\to\omega^{(p)}$ induced by the Verschiebung isogeny $\mathrm{Ver}:\mathcal{A}^{(p)}\to\mathcal{A}.$
As $V$ commutes with the endomorphisms it maps $\mathcal{P}$ to
$\mathcal{Q}^{(p)}$ and $\mathcal{Q}$ to $\mathcal{P}^{(p)}.$ We
denote the restriction of $V$ to $\mathcal{P}$ (resp. $\mathcal{Q})$
by $V_{\mathcal{P}}$ (resp. $V_{\mathcal{Q}}).$ The homomorphism
\[
H=V_{\mathcal{P}}^{(p)}\circ V_{\mathcal{Q}}:\mathcal{Q}\to\mathcal{Q}^{(p^{2})}
\]
is called the \emph{Hasse matrix. }We let $\mathcal{L}=\det(\mathcal{Q})$,
a line bundle. Then
\begin{equation}
h=\det(H):\mathcal{L}\to\mathcal{L}^{(p^{2})}\simeq\mathcal{L}^{p^{2}}\label{eq:hasse_inv}
\end{equation}
is a global section of $\mathcal{L}^{p^{2}-1}$ called the  \emph{$\mu$-ordinary Hasse
invariant} (see \cite{G-N}, Appendix B). Here we used the well-known
fact that for a line bundle $\mathcal{L}$ over a scheme in characteristic~$p$, there is a canonical isomorphism between $\mathcal{L}^{(p)}$
and $\mathcal{L}^{p},$ sending the base-change $s^{(p)}=1\otimes s$
of the section $s$ under the absolute Frobenius of $S$ to $s\otimes\cdots\otimes s.$
It is an important fact that $h\ne0$ precisely on $S^{\mathrm{ord}}.$
If $n>m$ the zero-divisor of $h$ is even \emph{reduced}, so equals
$S^{\mathrm{no}}=S\smallsetminus S^{\mathrm{ord}}$ with its reduced
subscheme structure. A proof of this fact may be found in \cite{Woo},
Proposition 7.2.11, but can also be extracted from the Dieudonn{\'e} module
computations in Theorem \ref{thm: =00005Cpsi(du)_has_a_zero} below. 

If $n=m$ this is not true; $h$ vanishes then on $S^{\mathrm{no}}$
to order $p+1$. There is a variant, though, that will be useful for
us in the study of the holomorphicity of the theta operator.
\begin{lem}
Let $n=m.$ Consider the maps of line bundles
\[
h_{\mathcal{Q}}=\det(V_{\mathcal{Q}}):\det(\mathcal{Q})\to\det(\mathcal{P})^{(p)}=\det(\mathcal{P})^{p},
\]
\[
h_{\mathcal{P}}=\det(V_{\mathcal{P}}):\det(\mathcal{P})\to\det(\mathcal{Q})^{(p)}=\det(\mathcal{Q})^{p}.
\]
Both $h_{\mathcal{P}}$ and $h_{\mathcal{Q}}$ vanish precisely on
$S^{\mathrm{no}}$ with multiplicity 1 and the following relation
holds:
\[
h=h_{\mathcal{P}}^{p}\circ h_{\mathcal{Q}}.
\]
\end{lem}

\begin{proof}
The claim concerning the vanishing of $h_{\mathcal{P}}$ and $h_{\mathcal{Q}}$
follows again from \cite{Woo}, Proposition 7.2.11, or from the computations
in Theorem \ref{thm:m=00003Dn_holmorphicity} below. The relation
$h=h_{\mathcal{P}}^{p}\circ h_{\mathcal{Q}}$ is a direct consequence
of the definition.
\end{proof}
Although the following Corollary is weaker than Lemma \ref{lem:detP=00003DdetQ},
it is of interest because its proof is entirely moduli-theoretic.
\begin{cor}
\label{cor:detP=00003DdetQ_muduli}Let $S$ be of arbitrary signature
$(n,m)$. There is an isomorphism 
\[
\det(\mathcal{P})^{p+1}\simeq\det(\mathcal{Q})^{p+1}.
\]
\end{cor}

\begin{proof}
Consider first the case of equal signatures $(m,m).$ By comparing
divisors of global sections, we obtain from the last Lemma an isomorphism
of line bundles $\det(\mathcal{P})^{p}\otimes\det(\mathcal{Q})^{-1}\simeq\det(\mathcal{Q})^{p}\otimes\det(\mathcal{P})^{-1}$
, implying the Corollary in this case.

For $S$ of signature $(n,m)$, and a geometric point $x$ of $S$,
we can embed $S$ in a suitable Shimura variety $\mathbb{S}$ of signature
$(n+m,n+m)$ by a morphism given on objects by $\underline{A}\mapsto\underline{A}\times\underline{B_{x}}$,
where $\underline{B_{x}}$ is the abelian variety corresponding to
$x$ with the twisted $\mathcal{O}_{E}$ structure. One easily checks
that the pull-back of the relation $\det(\mathcal{P})^{p+1}\simeq\det(\mathcal{Q})^{p+1}$
on $\mathbb{S}$ gives the same relation on $S$.
\end{proof}
Coming back to the case $m=n$ we have the following Lemma.
\begin{lem}
\label{lem:h=00003Dh_Q^(p+1)}Over an algebraic closure of $\kappa$,
we may fix the isomorphism $\det(\mathcal{P})\simeq\det(\mathcal{Q})=\mathcal{L}$
so that $h_{\mathcal{P}}=h_{\mathcal{Q}},$ hence $h=h_{\mathcal{Q}}^{p+1}.$
\end{lem}

\begin{proof}
Fix a smooth toroidal compactification $\overline{S}$ of $S$. As the abelian scheme $\mathcal{A/S}$ extends with the $\mathcal{O}_E$-action to a semi-abelian scheme over the toroidal compactification $\overline{S}$ (\cite{Lan}, Theorem 6.4.1.1) the
vector bundles $\mathcal{P}$ and $\mathcal{Q}$, as well as the homomorphisms
$h_{\mathcal{P}}$ and $h_{\mathcal{Q}}$, extend to $\overline{S}$ as well.
In Corollary \ref{cor:vanishing locus of Hasse} below we show that
$h$ does not vanish on any irreducible component of the boundary
$\overline{S}\smallsetminus S$. The same therefore must be true for
$h_{\mathcal{P}}$ and $h_{\mathcal{Q}}$. It follows that
\[
\div(h_{\mathcal{P}})=\div(h_{\mathcal{Q}})
\]
as divisors on the smooth, complete variety $\overline{S}.$ Fix any
isomorphism as in Lemma \ref{lem:detP=00003DdetQ}. Having the same
divisors, the sections $h_{\mathcal{P}}$ and $h_{\mathcal{Q}}$ of
$\mathcal{L}^{p-1}$ are equal up to a multiplication by a nowhere
vanishing function on $\overline{S}$ , hence equal up to a scalar on each connected component of $\bar S$.
By extracting a $p-1$ root from this scalar, we can normalize the
isomorphism $\det(\mathcal{P})\simeq\det(\mathcal{Q})$ so that $h_{\mathcal{P}}=h_{\mathcal{Q}}.$
\end{proof}
We remark that for more general Shimura varieties of PEL type the
construction of the Hasse invariant requires substantial work and
is due to Goldring and Nicole \cite{G-N}.

\subsubsection{The vector bundles $\mathcal{P}_{0}$ and $\mathcal{P}_{\mu}$}

The geometric fibers of the subsheaf
\[
\mathcal{P}_{0}=\ker(V_{\mathcal{P}})\subset\mathcal{P}
\]
have constant rank $n-m$ over an open subset $S_{\sharp}$ containing
the ordinary stratum
\[
S^{\mathrm{ord}}\subset S_{\sharp}\subset S.
\]
As the base is non-singular, this implies that \emph{over} $S_{\sharp}$
this $\mathcal{P}_{0}$ is a vector-sub-bundle of $\mathcal{P},$
hence so is the quotient 
\[
\mathcal{P}_{\mu}=\mathcal{P}/\mathcal{P}_{0}.
\]
In fact, $V_{\mathcal{P}}$ induces there an isomorphism
\begin{equation}
V_{\mathcal{P}}:\mathcal{P}_{\mu}\simeq\mathcal{Q}^{(p)},\label{eq:P_mu and Q^p}
\end{equation}
because as long as its kernel has rank $n-m,$ $V_{\mathcal{P}}$
must be surjective. The open subscheme $S_{\sharp}$ is of much interest,
and was analyzed in \cite{dS-G2}. It is the union of Ekedahl-Oort
strata (\cite{Oo},\cite{V-W}) that can be determined precisely.
When $m=1$, for example, its complement in $S$ is zero-dimensional
(the superspecial points). When $m<n$ this $S_{\sharp}$ contains
a unique Ekedahl-Oort stratum $S^{\mathrm{ao}}$ of dimension $mn-1$.
This will be used later on in our work.

The vector bundles $\mathcal{P}_{0},$ $\mathcal{P}_{\mu}$ and $\mathcal{Q}$
will turn out to be the building blocks of the mod-$p$ automorphic
vector bundles over $S^{\mathrm{ord}}.$ See $\S$\ref{subsec:auotomorpic _vector_bundles_mod_p}
for a discussion why we need to substitute the two subquotients $\mathcal{P}_{0}$
and $\mathcal{P}_{\mu}$ \emph{in lieu }of the classical automorphic
vector bundle $\mathcal{P}.$ 

It is a remarkable fact that $\mathcal{P}_{0}$ and $\mathcal{P}_{\mu}$
can be defined on the ordinary stratum also modulo $p^{s}$ for any
$s\ge1,$ although the Verschiebung isogeny is defined only in characteristic
$p.$ One way to see it is as follows. Let $R=\mathcal{O}_{E,(p)}$
and
\[
W_{s}=W_{s}(\kappa)=W(\kappa)/p^{s}W(\kappa)=R/p^{s}R
\]
(we identify the Witt vectors $W=W(\kappa)$ with the completion $\mathcal{O}_{p}$
of $R$). Denote by $S_{s}^{\mathrm{ord}}$ the open subscheme of
$S_{s}=\mathcal{S}\times_{\mathrm{Spec}(R)}\mathrm{Spec}(R/p^{s}R)$
whose underlying topological space is $S^{\mathrm{ord}}.$ The filtration
of the $p$-divisible group of $\mathcal{A}$ by its connected and
multiplicative parts extends uniquely from $S^{\mathrm{ord}}$ to
$S_{s}^{\mathrm{ord}}$. This is well-known for the connected part,
and by Cartier duality follows also for the multiplicative part. It
is crucial for us that the filtered pieces in (\ref{eq:ordinary_filtration})
have \emph{constant height} along $S_{s}^{\mathrm{ord}}.$ Moreover,
by the same result of Moonen quoted above (\cite{Mo1}, Corollary
2.1.5) the graded pieces of $\mathcal{A}[p^{\infty}]$ with their
additional structures of endomorphisms and polarization become isomorphic,
locally in the pro-{\'e}tale topology on $S_{s}^{\mathrm{ord}}$, to the
constant $p$-divisible groups $(\mathfrak{d}_{E}^{-1}\otimes\mu_{p^{\infty}})^{m},$
$\mathfrak{G}^{n-m}$ and $(\mathcal{O}_{E}\otimes\mathbb{Q}_{p}/\mathbb{Z}_{p})^{m}.$ 
(See $\S$\ref{subsec:bigIg} below for the $p$-divisible group $\mathfrak{G}$ over an
arbitrary base.)
In other words, not only modulo $p$ but modulo $p^{s}$ as well,
we can trivialize $gr^{i}\mathcal{A}[p^{t}]$ with the additional
structures after passing to a finite {\'e}tale covering. This remark will
be instrumental in the construction of the big Igusa tower below.

Coming back to the definition of $\mathcal{P}_{0}$ and $\mathcal{P}_{\mu}$
over $S_{s}^{\mathrm{ord}},$ if $t\ge s$ the exact sequence
\begin{equation}
0\to\mathcal{A}[p^{t}]\to\mathcal{A}\overset{p^{t}}{\to}\mathcal{A}\to0\label{eq:rs-exact-sequence}
\end{equation}
shows that $\mathrm{Lie}(\mathcal{A}[p^{t}]/S_{s})\to\mathrm{Lie}(\mathcal{A}/S_{s})$
is an isomorphism\footnote{For any group scheme $G$ over $T$ we define $\mathrm{Lie}(G/T)$
to be the kernel of the map ``mod$\,\varepsilon$'' from $G(T[\varepsilon])$
to $G(T),$ where $\varepsilon^{2}=0.$}. The filtration of $\mathcal{A}[p^{t}]$ induces (over $S_{s}^{\mathrm{ord}}$
only) a filtration of its Lie algebra by $\mathcal{O}_{S_{s}}$-sub-bundles,
hence a similar filtration of $\mathrm{Lie}(\mathcal{A}/S_{s})$.
By duality we get (again over $S_{s}^{\mathrm{ord}})$ a filtration
of $\omega$ by sub-bundles, which on its $\Sigma$-part yields the
exact sequence
\begin{equation}
0\to\mathcal{P}_{0}\to\mathcal{P}\to\mathcal{P}_{\mu}\to0.\label{eq:P_filtration}
\end{equation}
For future reference we record the fact that
\[
\mathcal{P}_{0}=\omega_{\mathcal{A}[p^{\infty}]^{0}/\mathcal{A}[p^{\infty}]^{\mu}},\,\,\,\,\mathcal{P}_{\mu}=\omega_{\mathcal{A}[p^{\infty}]^{\mu}}(\Sigma),\,\,\,\,\mathcal{Q}=\omega_{\mathcal{A}[p^{\infty}]^{\mu}}(\overline{\Sigma}).
\]

We do not know how to extend (\ref{eq:P_filtration}) in any intelligible
way to the $s$-th infinitesimal thickening of $S_{\sharp},$ as we
did when $s=1$ using Verschiebung.

\subsection{$p$-adic automorphic vector bundles}

\subsubsection{Representations of $GL_{m}$}

We review some well known facts from the representation theory of
$GL_{m}.$ Let $R$ be any ring, and $\mathrm{\underline{Rep}}_{R}(GL_{m})$
the category of algebraic representations of $GL_{m}$ on projective
$R$-modules of finite rank. If $\rho\in\mathrm{\underline{Rep}}_{R}(GL_{m}),$
we denote by $\rho(R)$ the associated projective $R$-module, endowed
with a left $GL_{m}(R)$ action. Given an $R$-scheme $S,$ the functoriality
in $R$ allows us to regard $\rho(\mathcal{O}_{S})=\mathcal{O}_{S}\otimes_{R}\rho(R)$
as a vector bundle with a left $GL_{m}(\mathcal{O}_{S})$ action on~$S$. The category $\mathrm{\underline{Rep}}_{R}(GL_{m})$ is a rigid
tensor category, and if $R$ is a field, it is also abelian. Some
special objects of the category are the standard representation $\mathrm{st}$,
and the symmetric and exterior powers $\mathrm{Sym}^{r}\mathrm{st}$
and $\wedge^{r}\mathrm{st}$ of $\mathrm{st}$, defined as suitable
\emph{quotients} of $\otimes^{r}\mathrm{st}.$

If $R$ is a field of characteristic $0,$ the category is even semisimple.
It is well known that the simple objects are then classified by dominant
weights. If $\lambda=(\lambda_{1}\ge\cdots\ge\lambda_{m})$ ($\lambda_{i}\in\mathbb{Z})$
is a dominant weight of $GL_{m},$ the corresponding object is
\begin{equation}
\rho_{\lambda}=\mathrm{Sym}^{\lambda_{1}-\lambda_{2}}(\mathrm{st})\otimes\mathrm{Sym}^{\lambda_{2}-\lambda_{3}}(\wedge^{2}\mathrm{st})\otimes\cdots\otimes\mathrm{Sym}^{\lambda_{m}}(\wedge^{m}\mathrm{st}).\label{eq:rho_lambda}
\end{equation}
Note that $\wedge^{m}\mathrm{st}$ is of rank 1, so $\mathrm{Sym}^{\lambda_{m}}(\wedge^{m}\mathrm{st})=\otimes^{\lambda_{m}}(\wedge^{m}\mathrm{st})$
makes sense even if $\lambda_{m}$ is negative. In Herman Weyl's construction
of $\rho_{\lambda}$ we assume first that $\lambda_{m}\ge0,$ view
$\lambda$ as a partition (Young tableau) of size $d=\sum_{i=1}^{m}\lambda_{i}$,
project $\otimes^{d}\mathrm{st}$ onto a sub-representation using
the Young symmetrizer $c_{\lambda}=a_{\lambda}b_{\lambda}\in\mathbb{Z}[\mathfrak{S}_{d}]$,
and then the resulting quotient is a model for $\rho_{\lambda}$,
\emph{cf.} \cite{F-H}, Ch.6. When $\lambda$ is not necessarily positive,
one reduces to the positive case by a twist by a power of the determinant
$\wedge^{m}\mathrm{st}$.

Recall, however, that over a field of characteristic $p$ the $\rho_{\lambda}$,
defined directly by (\ref{eq:rho_lambda}), are in general reducible
(e.g. $m=2$ and $\lambda=(p\ge0)$), and the category $\mathrm{\underline{Rep}}_{R}(GL_{m})$
is not semi-simple. As the Young symmetrizers are only quasi-idempotents
(i.e. $c_{\lambda}^{2}=n_{\lambda}c_{\lambda}$ for some integer $n_{\lambda}$
called the hook length of~$\lambda$, which might be divisible by
$p$) using them to study the representations of $GL_{m}$ becomes
tricky.

A more \emph{geometric} construction of $\rho_{\lambda}$ that works
\emph{over any ground ring $R,$ }hence produces an element of $\mathrm{\underline{Rep}}_{R}(GL_{m})$
functorially in $R,$ is via the Borel-Weil theorem 
-- see \cite{F-H} Claim
23.57 over $\mathbb{C}$. \cite{Jantzen}, II \S 5, gives this
construction of $\rho_\lambda$ over an arbitrary field, of any characteristic, and not necessarily algebraically closed. It
is however clear that the construction is valid over any ring $R$, and is furthermore functorial in $R$. 
Let $\overline{\lambda}=(\lambda_{m},\dots,\lambda_{1})$
be the anti-dominant weight for the standard torus of $GL_{m}$, which
is the opposite of $\lambda.$ Let $G=GL_{m}$ and let $B$ be the
standard upper-triangular Borel subgroup. Let $\overline{\lambda}$
denote also the character of $B$ obtained by first projecting modulo
the unipotent radical $U$ to the torus and then applying $\overline{\lambda}.$
On the flag variety $G/B$ define the line bundle $L_{\lambda}$ by
\[
L_{\lambda}=G\times^{B}\overline{\lambda}.
\]
This is the quotient of $G\times\mathbb{A}^{1}$ under the equivalence
relation $(gb,t)\sim(g,\overline{\lambda}(b)t)$ ($b\in B$). A global
section of $L_{\lambda}$ is identified with a map $\sigma:G\to\mathbb{A}^{1}$
satisfying $\sigma(gb)=\overline{\lambda}(b)^{-1}\sigma(g).$ In particular,
letting $w$ be the element of maximal length in the Weyl group (the
matrix with 1's on the anti-diagonal), we may define such a section
on the (open dense) big cell $UwB\subset G$ by
\begin{equation}
\sigma_{0}(uwb)=\overline{\lambda}(b)^{-1}.\label{eq:maxweight}
\end{equation}

The Borel-Weil theorem says that if $\lambda$ is dominant, then (a)
$L_{\lambda}$ is ample and $V_{\lambda}=H^{0}(G/B,L_{\lambda})\ne0$,
(b) if we let $G$ act on $V_{\lambda}$ by left translation, i.e.
$(g\sigma)(g')=\sigma(g^{-1}g'),$ this becomes a model for $\rho_{\lambda}$,
and finally (c) the $\sigma_{0}$ of (\ref{eq:maxweight}) extends
to a regular section on all of $G/B,$ the group $B\subset G$ acts
on it via the character $\lambda,$ and up to a scalar, $\sigma_{0}$
is the unique highest weight vector in $\rho_{\lambda}.$

This geometric formulation makes it evident that $\rho_{\lambda}$
so defined is functorial in~$R.$ Moreover, the linear functional
\begin{equation}
\Psi_{\lambda}:\sigma\mapsto\sigma(w)\in\mathbb{A}^{1}\label{eq:Psi_lambda}
\end{equation}
is easily seen to be in $\mathrm{Hom}_{\overline{B}}(\rho_{\lambda}|_{\overline{B}},\lambda)$
where $\overline{B}$ is the \emph{lower} triangular Borel. What's
more, since $L_{\lambda+\mu}=L_{\lambda}\otimes L_{\mu}$ there is
a canonical map (multiplication of global sections)
\begin{equation}
m_{\lambda,\mu}:\rho_{\lambda}\otimes\rho_{\mu}\to\rho_{\lambda+\mu},\label{eq:weight_multiplication}
\end{equation}
which is compatible with the functionals $\Psi_{\lambda},$$\Psi_{\mu}$
and $\Psi_{\lambda+\mu}$. From now on, whenever we write $\rho_{\lambda}$
or $\Psi_{\lambda}$ we shall have this specific model in mind.

We finally remark that if $R$ is an $\mathbb{F}_{p}$-algebra, and
$\phi:R\to R$ is the absolute Frobenius $\phi(x)=x^{p},$ then every
representation $\rho\in\mathrm{\underline{Rep}}_{R}(GL_{m})$ admits
a \emph{Frobenius twist} $\rho^{(p)}=\phi^{*}(\rho)$. In concrete
terms, locally on $R$ we may write $\rho$ in matrices, using a basis
of the underlying projective module, and $\rho^{(p)}$ is the representation
obtained by raising all the entries of the matrices to power $p.$

\subsubsection{Twisting a representation by a vector bundle\label{subsec:Twisting}}

Let $S$ be a scheme over $R.$ For every vector bundle $\mathcal{F}$
of rank $m$ over $S$ we let $\mathrm{\mathrm{\underline{Isom}}(}\mathcal{O}_{S}^{m},\mathcal{F})$
be the right $GL_{m}$-torsor of isomorphisms between $\mathcal{O}_{S}^{m}$
and $\mathcal{F}$, the group scheme $GL_{m/S}$ acting on the right
by pre-composition. If $\rho\in\mathrm{\underline{Rep}}_{R}(GL_{m})$
we consider the vector bundle
\[
\mathcal{F}_{\rho}=\mathrm{\mathrm{\underline{Isom}}(}\mathcal{O}_{S}^{m},\mathcal{F})\times^{GL_{m}}\rho(\mathcal{O}_{S})
\]
(contracted product). One should think of $\mathcal{F}_{\rho}$ as
``$\rho$ twisted by $\mathcal{F}$''. For example, for a dominant
weight $\lambda,$
\[
\mathcal{F}_{\rho_{\lambda}}=\mathrm{Sym}^{\lambda_{1}-\lambda_{2}}(\mathcal{F})\otimes\mathrm{Sym}^{\lambda_{2}-\lambda_{3}}(\wedge^{2}\mathcal{F})\otimes\cdots\otimes\mathrm{Sym}^{\lambda_{m}}(\wedge^{m}\mathcal{F}).
\]

What we have constructed is a \emph{tensor functor $\rho\rightsquigarrow\mathcal{F}_{\rho}$
}from $\mathrm{\underline{Rep}}_{R}(GL_{m})$ into the category $\underline{\mathrm{Vec}}_{S}$
of vector bundles over $S.$ These functors are compatible with base-change
of the underlying scheme $S,$ and with isomorphisms $\mathcal{\mathcal{F}}_{1}\simeq\mathcal{\mathcal{F}}_{2}$
between rank $m$ vector bundles. Thus if over $S'\to S$ the pull-backs
of two vector bundles $\mathcal{\mathcal{F}}_{i}$ become isomorphic
via an isomorphism $\varepsilon$, this $\varepsilon$ induces, over
$S'$, functorial isomorphisms $\varepsilon_{\rho}:\mathcal{F}_{1,\rho}\simeq\mathcal{F}_{2,\rho}$
for every $\rho\in\mathrm{\underline{Rep}}_{R}(GL_{m}).$

Note that if $R$ is an $\mathbb{F}_{p}$-algebra, then $\mathcal{F}_{\rho^{(p)}}=\mathcal{F}_{\rho}^{(p)}$,
where for any sheaf $\mathcal{F}$ over $S$ we denote by $\mathcal{F}^{(p)}=\Phi_{S}^{*}\mathcal{F}$
its pull-back by the absolute Frobenius of $S$. By $\mathcal{F}_{\rho}^{(p)}$
we mean either $(\mathcal{F}_{\rho})^{(p)}$ or $(\mathcal{F}^{(p)})_{\rho},$
the two being canonically identified.

The above generalizes to representations of a product of any number
of linear groups, say $M=\prod_{i=1}^{r}GL_{m_{i}}.$ Given $\rho\in\mathrm{\underline{Rep}}_{R}(M)$
and vector bundles $\mathcal{F}_{i}$ of ranks $m_{i}$ we let
\begin{equation}
\mathcal{E}_{\rho}=\mathrm{\mathrm{\prod_{i=1}^{r}\underline{Isom}}(}\mathcal{O}_{S}^{m_{i}},\mathcal{F}_{i})\times^{M}\rho(\mathcal{O}_{S}).\label{eq:E_rho}
\end{equation}
We call it the vector bundle obtained by \emph{twisting} $\rho$ \emph{by
the vector bundles} $\mathcal{F}_{i}.$

\subsubsection{\label{subsec:auotomorpic _vector_bundles_mod_p}$p$-adic automorphic
vector bundles over $S_{s}^{\mathrm{ord}}$}

Classically, automorphic vector bundles on $\mathcal{S}_{\mathbb{C}}$
are defined in the following way. Every connected component $\mathcal{S}_{\mathbb{C}}^{0}$
is of the form $\Gamma\setminus\boldsymbol{G}(\mathbb{R})/K_{\infty}$
where $K_{\infty}$ is a maximal compact-modulo-center subgroup, and
$\Gamma$ an arithmetic subgroup of $\boldsymbol{G}(\mathbb{R}).$
By a standard procedure due to Harish-Chandra one may embed the symmetric
space $\mathfrak{X}=\boldsymbol{G}(\mathbb{R})/K_{\infty}$ as an
open subset of its \emph{compact dual $\check{\mathfrak{X}}$.} In
our case the compact dual happens to be the Grassmannian $GL_{n+m}(\mathbb{C})/P_{\mathbb{C}}$,
where $P_{\mathbb{C}}$ is the standard maximal parabolic of type
$(m,n).$ (The change of variables involved in the Harish-Chandra
embedding for $U(n,m)$ is called the Cayley transform, as it generalizes
the well-known embedding of the upper half plane as the open unit
disk in $\mathbb{P}_{\mathbb{C}}^{1}$ when $n=m=1.$) The Levi quotient
of $P_{\mathbb{C}}$ is $M_{\mathbb{C}}= GL_{m}(\mathbb{C})\times GL_{n}(\mathbb{C}),$
and the automorphic vector bundles we consider are attached to representations
$\rho\in\mathrm{\underline{Rep}}_{\mathbb{C}}(M).$ 

Let such a representation $\rho$ be given. Let $P_{\mathbb{C}}$
act on $\rho(\mathbb{C})$ via its quotient $M_{\mathbb{C}}$, consider
the vector bundle
\[
GL_{n+m}(\mathbb{C})\times^{P_{\mathbb{C}}}\rho(\mathbb{C})
\]
on $\check{\mathfrak{X}}=GL_{n+m}(\mathbb{C})/P_{\mathbb{C}},$ and
denote by $\mathcal{\widetilde{E}}_{\rho}$ its restriction to $\mathfrak{X}$.
Since left multiplication by $\Gamma$ commutes with right multiplication
by $P_{\mathbb{C}}$, this vector bundle descends to a vector bundle
$\mathcal{E}_{\rho}$ on $\mathcal{S}_{\mathbb{C}}^{0}=\Gamma\setminus\mathfrak{X}.$
Using the complex analytic description of the universal abelian variety
over $\Gamma\setminus\mathfrak{X}$ one checks that the standard representations
of the two blocks in $M$ yield the vector bundles $\mathcal{Q}$
and $\mathcal{P}^{\vee}$. Easy group theory shows then that this
complex analytic construction gives, for any $\rho\in\mathrm{\underline{Rep}}_{\mathbb{C}}(M),$
a vector bundle which may be canonically identified with the $\mathcal{E}_{\rho}$
obtained by twisting $\rho$ by the pair of vector bundles $\mathcal{Q}$
and $\mathcal{P}^{\vee}$, as in the preceding paragraph.

This suggests to adopt the construction outlined in $\S$\ref{subsec:Twisting}
as an algebraic construction of automorphic vector bundles that works
equally well over the arithmetic scheme $\mathcal{S}$, hence also
over its special fiber $S$.

For the purpose of studying $p$\emph{-adic vector-valued modular
forms} this is however not always sufficient. In the classical complex
setting, a great advantage of the construction is that $\mathcal{\widetilde{E}}_{\rho}$
becomes trivial on $\mathfrak{X}$, hence may be described by matrix-valued
factors of automorphy. In the mod-$p$ or $p$-adic theory we need
an analogous covering of $S^{\mathrm{ord}}$ (or $S_{s}^{\mathrm{ord}}$),
over which our basic building blocks, hence all the $\mathcal{E}_{\rho}$,
will be trivialized. This is crucial both for Katz's theory of $p$-adic
modular forms, and for the construction of Maass-Shimura-like differential
operators below. This analogue of $\mathfrak{X}$ is th\emph{e (big)
Igusa tower, }to be described\emph{ }in $\S$\ref{subsec:bigIg}.

\emph{At this point the $\mu$-ordinary case becomes fundamentally
different from the ordinary one. }If $p$ is split in $E$, or if
$p$ is inert but $m=n,$ then both $\mathcal{P}$ and $\mathcal{Q}$
are trivialized over the Igusa tower and everything works well with
the usual automorphic vector bundles. However, if $p$ is inert and
$m<n$ then $\mathcal{P}$ can not be trivialized over the Igusa tower,
nor on any other pro-{\'e}tale cover. The best we can do is to trivialize
its subquotients $\mathcal{P}_{0}$ and $\mathcal{P_{\mu}}$ separately.
This explains why we need to start with \emph{three} basic bundles
$\mathcal{Q},\mathcal{P}_{\mu}$ and $\mathcal{P}_{0}$ over $S_{s}^{\mathrm{ord}},$
and why our $\rho$ will be an element of $\mathrm{\underline{Rep}}_{R}(M)$
with
\[
M=GL_{m}\times GL_{m}\times GL_{n-m}
\]
rather than $GL_{m}\times GL_{n}$ as over $\mathbb{C}$. 

After this long discussion, we can finally make the following definition.
\begin{defn}
\label{def:p_adic automorphic vector bundle}Let $\rho\in\mathrm{\underline{Rep}}_{R}(M)$
where $M=GL_{m}\times GL_{m}\times GL_{n-m},$ and define $\mathcal{E}_{\rho}$
on $S_{s}^{\mathrm{ord}}$ by (\ref{eq:E_rho}) with $(\mathcal{Q},\mathcal{P}_{\mu},\mathcal{P}_{0})$
replacing the $\mathcal{F}_{i}$. We call $\mathcal{E}_{\rho}$ the
$p$\emph{-adic automorphic vector bundle of weight} $\rho$ (mod
$p^{s}),$ and $\underset{\leftarrow s}{\lim}\,H^{0}(S_{s}^{\mathrm{ord}},\mathcal{E}_{\rho})$
the space of \emph{$p$-adic (vector-valued) }modular forms\emph{
of weight} $\rho.$
\end{defn}

\begin{rem*}
(i) Note that by our convention the standard representations of the
second and third factors of $M$ correspond to $\mathcal{P}_{\mu}$
and $\mathcal{P}_{0}$, while the complex analytic standard representation
of $GL_{n}$ corresponded to $\mathcal{P}^{\vee}.$ 

(ii) A $p$-adic modular form need not come from a global section
over $\mathcal{S}.$ It is a rigid analytic object, defined over the
affinoid which is the generic fiber of the formal completion of $\mathcal{S}$
along $S^{\mathrm{ord}}.$ In fact, if $\mathcal{P}_{\mu}$ and $\mathcal{P}_{0}$
are ``involved'' in $\mathcal{E_{\rho}}$ (in the precise sense
that $\rho$ does not come from a representation of the first simple
factor of $M$) then it does not even make sense to ask whether the
modular form extends to a global section over $\mathcal{S}$, because
the $p$-adic automorphic vector bundle does not extend there. In
order to compare classical and $p$-adic modular forms we make the
following definition.
\end{rem*}
\begin{defn}
Let $\rho\in\mathrm{\underline{Rep}}_{R}(M).$ We say that the $p$-adic
automorphic vector bundle $\mathcal{E}_{\rho}$ is \emph{of classical
type} if $\rho$ factors through the first factor of $M.$
\end{defn}

A $p$-adic automorphic vector bundle of classical type is the restriction
to $S_{s}^{\mathrm{ord}}$ of a classical automorphic vector bundle.
Note however that $\mathcal{P}$, an honest automorphic vector bundle
on $S_{s}$, is not a $p$-adic automorphic vector bundle on $S_{s}^{\mathrm{ord}}$
(if $m<n$), as it can not be reconstructed from its graded pieces
$\mathcal{P}_{0}$ and $\mathcal{P}_{\mu}$.

\section{\label{sec:Differential-operators-on}Differential operators on $p$-adic
modular forms}

\subsection{\label{subsec:bigIg}The big Igusa tower}

\subsubsection{The $p$-divisible group $\mathfrak{G}$}

Following a long-standing tradition going back to Katz in the \emph{ordinary
}case, we want to describe a certain \emph{tower} of (big) Igusa varieties
$T_{t,s},$ for all $t,s\ge1.$ The variety $T_{t,s}$ will be an
Igusa variety of level $p^{t}$ over $\mathcal{O}_{E,(p)}/p^{s}\mathcal{O}_{E,(p)}$.
By ``tower'' we mean that the reduction of $T_{t,s+1}$ modulo $p^{s}$
will be identified with $T_{t,s}$, and that for a fixed $s$ there
will be compatible morphisms from level $p^{t'}$ to level $p^{t}$
for all $t'\ge t.$ This ``big Igusa tower'' has been defined and
studied, in much greater generality, in E. Mantovan's work \cite{Man}.

To describe it, we shall have to choose a model $\mathfrak{G}$ over
$W=W(\kappa)=\mathcal{O}_{p}$ of the $p$-divisible group that becomes,
over $\overline{\kappa}$, the group $\mathfrak{G}_{\overline{\kappa}}$
introduced in $\S$\ref{subsec:The-universal-abelian}. This choice results
in freedom, which grows with $t$ and $s$, and prevents the $T_{t,s}$
(unlike the \emph{small} Igusa varieties, see below) from being canonically
defined. This problem will nevertheless disappear over $W(\overline{\kappa}),$
so the reader interested in the construction over $W(\overline{\kappa})/p^{s}W(\overline{\kappa})$
only, can happily ignore the issue.

The easiest way to fix our model is to choose an elliptic curve $\mathcal{\mathscr{C}}$
defined over $W$, with complex multiplication by $\mathcal{O}_{E}$
and CM type $\Sigma$. The theory of complex multiplication guarantees
that such an elliptic curve exists, and has supersingular reduction.
We then let $\mathfrak{G=\mathscr{C}}[p^{\infty}]$ be its $p$-divisible
group. Its special fiber $\mathfrak{G}_{\kappa}$ is of local-local
type, height 2 and dimension 1. The canonical polarization of the
elliptic curve supplies an isomorphism of $\mathfrak{G}$ with its
Serre dual, hence a compatible system of perfect alternating Weil
pairings
\[
\left\langle ,\right\rangle :\mathfrak{G}[p^{t}]\times\mathfrak{G}[p^{t}]\to\mu_{p^{t}}
\]
($t\ge1$)\emph{.}

The completion $\mathcal{O}_{p}$ of $\mathcal{O}_{E}$ maps isomorphically
onto $\mathrm{End}(\mathfrak{G}_{/W})\subset\mathrm{End}(\mathfrak{G_{/\kappa}})$.
Furthermore, for any $W$-algebra $R$
\[
\mathrm{End}_{\mathcal{O}_{E}}(\mathfrak{G}[p^{t}]_{/R})=\mathcal{O}_{p}/p^{t}\mathcal{O}_{p}.
\]
We have $\left\langle \iota(a)u,v\right\rangle =\left\langle u,\iota(\overline{a})v\right\rangle $
for every $a\in\mathcal{O}_{E}$.

\subsubsection{The Igusa moduli problem}

If $R$ is a $W_{s}(\kappa)$-algebra and $A_{/R}$ is fiber-by-fiber
$\mu$-ordinary, then its $p$-divisible group admits a filtration
like (\ref{eq:ordinary_filtration}) whose graded pieces we label
$gr^{i}A[p^{\infty}].$ We choose the indices in such a way that locally
in the pro-{\'e}tale topology on $\mathrm{Spec}(R)$ there exist isomorphisms
\begin{equation}
\epsilon^{0}:(\mathcal{O}_{E}\otimes\mathbb{Q}_{p}/\mathbb{Z}_{p})_{R}^{m}\simeq gr^{0},\,\,\,\,\epsilon^{1}:\mathfrak{G}_{R}^{n-m}\simeq gr^{1},\,\,\,\,\,\epsilon^{2}:(\mathfrak{d}_{E}^{-1}\otimes\mu_{p^{\infty}})_{R}^{m}\simeq gr^{2},\label{eq:Igusa structures}
\end{equation}
respecting the action of $\mathcal{O}_{E}$ and the pairings. Note
that $gr^{1}$ is self-dual, while $\epsilon^{0}$ and $\epsilon^{2}$
determine each other. For future reference we want to make the pairings
on these ``model group schemes'' explicit. If 
\[
\alpha=(x_{1},\dots,x_{m},y_{1},\dots,y_{n-m},z_{1},\dots,z_{m})\in(\mathcal{O}_{E}\otimes\mathbb{Q}_{p}/\mathbb{Z}_{p})_{R}^{m}\times\mathfrak{G}_{R}^{n-m}\times(\mathfrak{d}_{E}^{-1}\otimes\mu_{p^{\infty}})_{R}^{m},
\]
and similarly $\alpha'=(x'_{1},\dots,x'_{m},y'_{1},\dots,y'_{n-m},z'_{1},\dots,z'_{m}),$
we define
\begin{equation}
\left\langle \alpha,\alpha'\right\rangle =\prod_{i=1}^{m}\left\langle x_{i},z'_{m+1-i}\right\rangle \prod_{j=1}^{n-m}\left\langle y_{j},y'_{n-m+1-j}\right\rangle \prod_{i=1}^{m}\left\langle z_{i},x'_{m+1-i}\right\rangle .\label{eq:model_pairing}
\end{equation}
In matrix form, writing $\mu_{p^{\infty}}$ additively, we take, $\,^{t}\alpha J_{n+m}\alpha'$
where $J_{l}$ is the anti-diagonal matrix of size $l$, and \emph{not
$\,^{t}\alpha I_{n,m}\alpha'$ }where $I_{n,m}$ is the matrix (\ref{eq:I_(n,m)}).
As remarked in $\S$\ref{subsec:Linear-algebra}, these two pairings produce
\emph{isomorphic }polarized $\mathcal{O}_{E}$-groups. Thus, there
is no real difference which pairing we take at this point, but for
later book-keeping purposes, we prefer the one with $J_{n+m}.$

We call $\epsilon=(\epsilon^{0},\epsilon^{1},\epsilon^{2})$ a \emph{graded
symplectic trivialization }of the $p$-divisible group. A \emph{graded
symplectic trivialization} of $A[p^{t}]$ is a similar system of isomorphisms
of the $p^{t}$-torsion in the $p$-divisible groups, defined over
$R,$ which is locally {\'e}tale liftable to a graded symplectic trivialization
of the whole $p$-divisible group.
\begin{defn}
The \emph{big Igusa moduli problem of level} $p^{t}$ over $W_{s}(\kappa),$
denoted $T_{t,s}$, classifies tuples
\[
(\underline{A},\epsilon)_{/R/W_{s}},
\]
where $\underline{A}\in S_{s}^{\mathrm{ord}}(R)$ and $\epsilon$
is a graded symplectic trivialization of $A[p^{t}]$ as in (\ref{eq:Igusa structures}),
up to isomorphism. 
\end{defn}

The representability of this moduli problem by a scheme, denoted also
$T_{t,s}$, is standard. One only has to check that it is \emph{relatively
representable} over $S_{s}^{\mathrm{ord}}$ (see \cite{Ka-M}, Chapter
4). The maps between the levels are self-evident. The morphism
\[
\tau:T_{t,s}\to S_{s}^{\mathrm{ord}}
\]
is a Galois {\'e}tale covering of $S_{s}^{\mathrm{ord}}$ (\cite{Man},
Proposition 4).

The \emph{small Igusa variety }of the same level classifies tuples
$(\underline{A},\epsilon^{2})$ of the same nature. There is an obvious
morphism from the big tower to the small one: ``forget $\epsilon^{1}$''.
Since $\epsilon^{0}$ is determined by $\epsilon^{2}$ we do not have
to forget anything more.

\subsubsection{The Galois group}

The Galois group $\Delta_{t}$ of the covering $\tau:T_{t,s}\to S_{s}^{\mathrm{ord}}$
is isomorphic to $GL_{m}(\mathcal{O}_{E}/p^{t}\mathcal{O}_{E})\times U_{n-m}(\mathcal{O}_{E}/p^{t}\mathcal{O}_{E})$
under 
\[
\Delta_{t}\ni\gamma\mapsto[\gamma]=(\gamma{}_{2},\gamma{}_{1})\in GL_{m}(\mathcal{O}_{E}/p^{t}\mathcal{O}_{E})\times U_{n-m}(\mathcal{O}_{E}/p^{t}\mathcal{O}_{E}),
\]
 where
\begin{equation}
\gamma(\underline{A},\epsilon)=(\underline{A},\epsilon\circ[\gamma]^{-1}).\label{eq:Galois action on Igusa level}
\end{equation}
Here by $U_{n-m}(\mathcal{O}_{E}/p^{t}\mathcal{O}_{E})$ we mean the
quasi-split unitary group, consisting of matrices $g$ of size $n-m$
satisfying the relation $^{t}\overline{g}J_{n-m}g=J_{n-m}.$ As explained
before, it is isomorphic to the group of matrices satisfying $^{t}\overline{g}g=I.$
By $\epsilon\circ[\gamma]^{-1}$ we mean that we compose $\epsilon^{1}$
with $\gamma{}_{1}^{-1}$ and $\epsilon^{2}$ with $\gamma{}_{2}^{-1}$
(the action on $\epsilon^{0}$ being determined by the one on $\epsilon^{2}$).
As usual, the group $\Delta_{t}$ acts simply transitively on the geometric
fibers of the morphism $\tau$.

\subsubsection{Trivializing the three basic vector bundles over the Igusa tower}

For simplicity write $T=T_{t,s}$, $\Delta=\Delta_{t},$ and assume
that $t\ge s.$ There is enough level structure then to ``see''
the relative Lie algebra of $\mathcal{A}_{/S_{s}}$ on $\mathcal{A}[p^{t}]_{/S_{s}},$
as explained in the paragraph following (\ref{eq:rs-exact-sequence}).

As the cotangent space at the origin of $\mathfrak{d}_{E}^{-1}\otimes\mu_{p^{t}/W_{s}}$
is canonically identified with $\mathcal{O}_{E}\otimes W_{s}=W_{s}(\Sigma)\oplus W_{s}(\overline{\Sigma})$,
the isomorphism $\epsilon^{2}$ induces canonical trivializations
of $\mathcal{O}_{E}$-vector bundles over $T$
\[
\varepsilon^{2}=((\epsilon^{2})^{-1})^{*}:\mathcal{O}_{E}\otimes\mathcal{O}_{T}^{m}\simeq\mathcal{Q}\oplus\mathcal{P}_{\mu}
\]
(we write $\mathcal{Q}$ for $\tau^{*}\mathcal{Q}$ etc. as $\tau^{*}\mathcal{Q}$
is ``the'' $\mathcal{Q}$ of $\mathcal{A}/T$), or
\[
\varepsilon^{2}(\overline{\Sigma}):\mathcal{O}_{T}^{m}\simeq\mathcal{Q},\,\,\,\varepsilon^{2}(\Sigma):\mathcal{O}_{T}^{m}\simeq\mathcal{P}_{\mu}.
\]
Similarly fix, once and for all, an isomorphism of the cotangent space
at the origin of $\mathfrak{G}[p^{t}]_{/W_{s}}$ (as an $\mathcal{O}_{E}$-module)
with $W_{s}(\Sigma).$ The isomorphism $\epsilon^{1}$ induces then
also a canonical trivialization over $T$
\[
\varepsilon^{1}:\mathcal{O}_{T}^{n-m}\simeq\mathcal{P}_{0}.
\]

The action (\ref{eq:Galois action on Igusa level}) of $\gamma\in\Delta$
on $T$ induces the following action on the trivializations

\begin{equation}
\gamma(\varepsilon^{i})=\varepsilon^{i}\circ\,^{t}\gamma{}_{i}.\label{eq:Gamma_action}
\end{equation}
 ($i=1,2$). Let us check the last formula, dropping the index $i$:
\[
\gamma(\varepsilon)=(\gamma(\epsilon)^{-1})^{*}=(\gamma\circ\epsilon^{-1})^{*}=(\epsilon^{-1})^{*}\circ\gamma{}^{*}=\varepsilon\circ\,^{t}\gamma,
\]
because the matrix representing $[\gamma]^{*}$ on the cotangent space
is the transpose of the matrix representing $[\gamma]_{*}$ on the
Lie algebra, which is simply $[\gamma].$

\subsection{The theta operator}

\subsubsection{Pre-theta}

Let $\rho$ be a representation of $GL_{m}\times GL_{m}\times GL_{n-m}$
over $W_{s},$ and let $\mathcal{E}_{\rho}$ be the automorphic vector
bundle on $S_{s}^{\mathrm{ord}}$ defined above. We define a connection
\[
\widetilde{\Theta}:\mathcal{E}_{\rho}\to\mathcal{E_{\rho}}\otimes\Omega_{S_{s}/W_{s}}
\]
over $S_{s}^{\mathrm{ord}}.$

Let $t\ge s.$ Denote by $\mathcal{O}_{\rho}=\rho(\mathcal{O}_{T})$
the vector bundle over $T=T_{t,s}$ obtained by twisting the representation
$\rho$ by the trivial vector bundles $\mathcal{O}_{T}^{m},\,\mathcal{O}_{T}^{m}$
and $\mathcal{O}_{T}^{n-m}$ as in Definition \ref{def:p_adic automorphic vector bundle}.
The trivial connection on the structure sheaf $\mathcal{O}_{T}$ induces,
by the usual rules, a connection
\[
d_{\rho}:\mathcal{O}_{\rho}\to\mathcal{O}_{\rho}\otimes\Omega_{T/W_{s}}.
\]
For example, if $\rho=\rho_{\lambda}$ where $\lambda=(\lambda_{1},\dots,\lambda_{m})$
is a dominant weight depending only on the first $GL_{m}$ factor,
so that $\mathcal{O}_{\rho}$ is given by (\ref{eq:rho_lambda}),
then $d_{\rho}$ is given by the usual rules of differentiation of
symmetric powers, exterior powers and duals.

On the other hand the trivializations $\varepsilon^{1}$ and $\varepsilon^{2}$
constructed above yield a trivialization
\[
\varepsilon_{\rho}:\mathcal{O}_{\rho}\simeq\tau^{*}\mathcal{E}_{\rho}
\]
over $T.$ To get the action of 
\[
\gamma=(\gamma_{2},\gamma_{1})\in\Delta=GL_{m}(\mathcal{O}_{E}/p^{t}\mathcal{O}_{E})\times U_{n-m}(\mathcal{O}_{E}/p^{t}\mathcal{O}_{E})
\]
on $\varepsilon_{\rho}$ we first map $\gamma$ to $GL_{m}(W_{s})\times GL_{m}(W_{s})\times GL_{n-m}(W_{s})$
via
\[
\gamma\mapsto\iota(\gamma)=(\overline{\gamma_{2}},\gamma_{2},\gamma_{1})
\]
(well defined because $t\ge s$) and let $^{t}[\gamma]_{\rho}=\rho(\,^{t}\iota(\gamma)).$
Then from (\ref{eq:Gamma_action}) we get
\begin{equation}
\gamma(\varepsilon_{\rho})=\varepsilon_{\rho}\circ\,^{t}[\gamma]_{\rho}.\label{eq:gamma_rho_action}
\end{equation}

Let $U\subset S_{s}^{\mathrm{ord}}$ be Zariski open. For $f\in H^{0}(U,\mathcal{E}_{\rho})$
define
\begin{equation}
\widetilde{\Theta}(f)=(\varepsilon_{\rho}\otimes1)\circ d_{\rho}\circ\varepsilon_{\rho}^{-1}(\tau^{*}f)\in H^{0}(\tau^{-1}(U),\tau^{*}\mathcal{E_{\rho}}\otimes_{\mathcal{O}_{T}}\Omega_{T/W_{s}}).\label{eq:pre_theta}
\end{equation}
Since $\tau$ is {\'e}tale, $\Omega_{T/W_{s}}=\mathcal{O}_{T}\otimes_{\mathcal{O}_{S_{s}}}\Omega_{S_{s}/W_{s}},$
so
\[
\widetilde{\Theta}(f)\in H^{0}(\tau^{-1}(U),\tau^{*}\mathcal{E_{\rho}}\otimes_{\mathcal{O}_{S_{s}}}\Omega_{S_{s}/W_{s}}).
\]
We have to show that $\widetilde{\Theta}(f)\in H^{0}(U,\mathcal{E}_{\rho}\otimes_{\mathcal{O}_{S_{s}}}\Omega_{S_{s}/W_{s}}),$
and for that it would suffice to show that it is invariant under $\Delta.$
Let $\gamma\in\Delta.$ Then by (\ref{eq:gamma_rho_action})
\[
\gamma(\widetilde{\Theta}(f))=(\varepsilon_{\rho}\otimes1)\circ\,^{t}[\gamma]_{\rho}\circ d_{\rho}\circ\,^{t}[\gamma]_{\rho}^{-1}\circ\varepsilon_{\rho}^{-1}(\tau^{*}f)=\widetilde{\Theta}(f).
\]
Here we used that (a) $\tau^{*}f$ is Galois invariant, (b) $d_{\rho}$
is Galois invariant since $\tau$ is {\'e}tale, and (c) $d_{\rho}$ commutes
with the scalar matrices $^{t}[\gamma]_{\rho}.$ We summarize our
construction in the following theorem.
\begin{thm}
Let $U\subset S_{s}^{\mathrm{ord}}$ be an open set and $f\in H^{0}(U,\mathcal{E}_{\rho}).$
Then
\[
\widetilde{\Theta}(f)=(\varepsilon_{\rho}\otimes1)\circ d_{\rho}\circ\varepsilon_{\rho}^{-1}(\tau^{*}f)\in H^{0}(U,\mathcal{E}_{\rho}\otimes_{\mathcal{O}_{S_{s}}}\Omega_{S_{s}/W_{s}})
\]
yields a well-defined connection on $\mathcal{E}_{\rho}.$ The connection
defined on $\mathcal{E}\otimes\mathcal{F},$ $\mathcal{E}^{\vee}$
etc. is the tensor product, dual etc. of the connections defined on
the individual sheaves. If $s=1$ (i.e. we are in characteristic $p$),
then the connection defined on $\mathcal{E}^{(p)}$ is trivial. Hence,
if $f$ and $g$ are sections of $\mathcal{E}$ and $\mathcal{F},$
respectively, then on $\mathcal{E}^{(p)}\otimes\mathcal{F}$ we have
$\widetilde{\Theta}(f^{(p)}\otimes g)=f^{(p)}\otimes\widetilde{\Theta}(g).$
\end{thm}

\begin{proof}
The functoriality with respect to linear-algebra operations (including
Frobenius twist in characteristic $p$) is clear. The last remark
is a general fact about modules with connection. For any vector bundle
$\mathcal{E}$ over a base $S$ in characteristic~$p$ there is a
canonical connection $\nabla^{\mathrm{can}}$ on $\mathcal{E}^{(p)},$
characterized by $\nabla^{\mathrm{can}}(f^{(p)})=0$ for any section
$f$ of $\mathcal{E},$ and if $\nabla$ is \emph{any }connection
on $\mathcal{E},$ then its pull-back $\nabla^{(p)}$ to $\mathcal{E}^{(p)}$
is canonically identified with $\nabla^{\mathrm{can}}.$
\end{proof}

\subsubsection{Theta}

Using the inverse of the Kodaira-Spencer isomorphism
\[
\mathrm{KS}:\mathcal{P}\otimes\mathcal{Q}\simeq\Omega_{S_{s}/W_{s}}
\]
we may view $\widetilde{\Theta}$ as a map from $\mathcal{E}_{\rho}$
to $\mathcal{E_{\rho}}\otimes\mathcal{P}\otimes\mathcal{Q}.$ We emphasize
that this map is not a sheaf homomorphism, as it is only $\kappa$-linear
and not $\mathcal{O}_{S_{s}}$-linear. It is better, however, to consider
the operator
\begin{equation}
\Theta=(1\otimes pr_{\mu}\otimes1)\circ(1\otimes\mathrm{KS^{-1})}\circ\widetilde{\Theta}:\mathcal{E}_{\rho}\to\mathcal{E}_{\rho}\otimes\mathcal{P_{\mu}}\otimes\mathcal{Q}.\label{eq:Theta}
\end{equation}
Here $pr_{\mu}:\mathcal{P}\to\mathcal{P}/\mathcal{P}_{0}=\mathcal{P}_{\mu}$
is the canonical projection.

If $s=1$, in characteristic $p$ over $S$, we may replace $\mathcal{P}_{\mu}$
by $\mathcal{Q}^{(p)}$ and $pr_{\mu}$ by $V$. From the point of
view of connections, dividing $\Omega_{S/\kappa}$ by $\ker(V\otimes1)=\mathcal{P}_{0}\otimes\mathcal{Q}$
means that we restrict the connection to the foliation $TS^{+}\subset TS$
which has been introduced and studied in \cite{dS-G2}, i.e. use it
to differentiate sections of $\mathcal{\mathcal{E}_{\rho}}$ only
in the direction of $TS^{+}.$ Although this voluntarily gives up
information encoded in~$\widetilde{\Theta}$, when restricted to characteristic
$p$, the operator $\Theta$ has \emph{four} advantages over its predecessor:
\begin{enumerate}
\item While $\widetilde{\Theta}$ has poles along the complement of $S^{\mathrm{ord}}$
in $S,$ we shall see that $\Theta$ may be analytically continued
everywhere, at least when applied to scalar modular forms.
\item The effect of $\Theta$ on Fourier-Jacobi expansions is particularly
nice, while the formulae for $\widetilde{\Theta}$ contain unpleasant
terms. 
\item Restricting the connection to the foliation $TS^{+}$ should also
result in a nice expansion of $\Theta$ at a $\mu$-ordinary point
in terms of Moonen's generalized Serre-Tate coordinates \cite{Mo1}.
This is the approach taken in \cite{E-M}. For the relation between
$TS^{+}$ and Moonen's generalized Serre-Tate coordinates, see \cite{dS-G2},
$\S$3.3, Theorem 13.
\item Unlike $\widetilde{\Theta}$, the operator $\Theta$ lands back in
a sheaf which is obtained ``by linear algebra operations'' from
$\mathcal{Q},\mathcal{P}_{\mu}$ and $\mathcal{P}_{0}$. This will
allow us to \emph{iterate} $\Theta$, something which we were prohibited
from doing with $\widetilde{\Theta}$ due to the presence of $\mathcal{P}.$
\end{enumerate}

\subsection{Higher order differential operators $D_{\kappa}^{\kappa'}$}

For the sake of completeness we indicate how one gets, by iterating
$\Theta,$ a whole array of differential operators $D_{\kappa}^{\kappa'}.$
We follow, with minor modifications, Eischen's thesis \cite{Ei}.
If $\kappa=(a,b,c)$ is a dominant weight of $M=GL_{m}\times GL_{m}\times GL_{n-m}$
we denote the vector bundle $\mathcal{E_{\rho}}$ associated with
the representation $\rho=\rho_{\kappa}$ by $\mathcal{E}_{\kappa}.$

Let $\mathrm{st}$ be the standard representation of $GL_{m}$ over
$W$, let $a'$ be a positive dominant weight $a'_{1}\ge\cdots\ge a'_{m}\ge0$
and $e=\Sigma_{i=1}^{m}a'_{i}$. Then in $\underline{\mathrm{Rep}}_{W}(GL_{m})$
there exists a distinguished homomorphism, unique up to a $W^{\times}$-multiple,
\[
\pi_{a'}:\mathrm{st}^{\otimes e}\to\rho_{a'}.
\]
One simply has to normalize the homomorphism resulting from the Young
symmetrizer $c_{a'}$ so that it is integral, but not divisible by
$p.$ Whether $\pi_{a'}$ can be further normalized to eliminate the
$W^{\times}$-ambiguity depends on which model we take for $\rho_{a'},$
as two such models are canonically isomorphic only up to multiplication
by a scalar. Since we agreed to take the models given by the Borel-Weil
theorem over $W$, we do not know how to normalize $\pi_{a'}$ any further or whether it is surjective before inverting $p$.

Let $\kappa'=(a',b',0)$ be a dominant weight with $a'$ and $b'$
positive, such that 
\[
e=\sum_{i=1}^{m}a'_{i}=\sum_{i=1}^{m}b'_{i}.
\]
In \cite{E-F-M-V} such a $\kappa'$ is called sum-symmetric.

We twist $\rho_{\kappa'}=\rho_{a'}\otimes\rho_{b'}\otimes1$ by the
vector bundles $\mathcal{Q}$ and $\mathcal{P}_{\mu}$. Recall that
$\mathcal{Q}$ is used to twist $\rho_{a'}$ and $\mathcal{P}_{\mu}$
is used for $\rho_{b'}$, while twisting by $\mathcal{P}_{0}$ is
not needed, as the representation associated with $GL_{n-m}$ is the
trivial one. We get
\[
\pi_{\kappa'}=\pi_{b'}\otimes\pi_{a'}:(\mathcal{P}_{\mu}\otimes\mathcal{Q})^{\otimes e}\to\mathcal{E}_{\kappa'}.
\]

Let $\kappa=(a,b,c)$ be a dominant weight of $M.$ Consider the $e$-th
iteration of the derivation $\Theta$. It maps the sheaf $\mathcal{E}_{\kappa}$
to $\mathcal{E}_{\kappa}\otimes(\mathcal{P}_{\mu}\otimes\mathcal{Q})^{\otimes e}.$
We may now use $\pi_{\kappa'}$ to map $(\mathcal{P}_{\mu}\otimes\mathcal{Q})^{\otimes e}$
to $\mathcal{E}_{\kappa'}$ and finally apply the homomorphism $m_{\kappa,\kappa'}:\mathcal{E}_{\kappa}\otimes\mathcal{E}_{\kappa'}\to\mathcal{E}_{\kappa+\kappa'}$
of (\ref{eq:weight_multiplication}) to get the differential operator
\begin{equation}
D_{\kappa}^{\kappa'}=m_{\kappa,\kappa'}\circ(1\otimes\pi_{\kappa'})\circ\Theta^{e}:\mathcal{E}_{\kappa}\to\mathcal{E}_{\kappa+\kappa'}.\label{eq:D_=00005Ckappa^kappa'}
\end{equation}
As $m_{\kappa,\kappa'}\circ(1\otimes\pi_{\kappa'})$ is a sheaf homomorphism
this $D_{\kappa}^{\kappa'}$ is a differential operator of order $e$.
It is well-defined only up to a scalar from $W^{\times}.$ The operators
$D_{\kappa}^{\kappa'}$ allow us to increase the weight by any $\kappa'$
as long as
\[
\kappa'=(a',b',0),\,\,\,\,a'_{1}\ge\cdots\ge a'_{m}\ge0,\,\,\,b'_{1}\ge\cdots\ge b'_{m}\ge0,\,\,\,\,\sum_{i=1}^{m}a'_{i}=\sum_{i=1}^{m}b'_{i}.
\]
\begin{example*}
\textit{Scalar-valued modular forms. }If $\kappa=(k,\dots,k;0,\dots,0;0,\dots,0)$
then 
\[
\mathcal{E}_{\kappa}=\det(\mathcal{Q})^{k}=\mathcal{L}^{k}.
\]
In this case, global sections of $\mathcal{E}_{\kappa}$ are scalar-valued
modular forms on $\boldsymbol{G}$ of weight~$k$. If we take $\kappa'=(k',\dots,k';k',\dots,k';0,\dots,0)$
then $D_{\kappa}^{\kappa'}$ maps $\mathcal{L}^{k}$ to $\mathcal{L}^{k+k'}\otimes\det(\mathcal{P}_{\mu})^{k'}.$
If $s=1,$ in characteristic $p$, we may identify $\det(\mathcal{P}_{\mu})$
with $\mathcal{L}^{p}$ (\ref{eq:P_mu and Q^p}), so $D_{\kappa}^{\kappa'}$
maps $\mathcal{L}^{k}$ to $\mathcal{L}^{k+(p+1)k'}.$ In these cases
$D_{\kappa}^{\kappa'}$ is obtained by applying $\Theta$ iteratively
$mk$ times and projecting. If $m=1$ then $D_{\kappa}^{\kappa'}$
is simply $\Theta^{k'}.$
\end{example*}

\section{\label{sec:Toroidal-compactifications-and}Toroidal compactifications
and Fourier-Jacobi expansions}

\subsection{Toroidal compactifications and logarithmic differentials\label{subsec:Toroidal-compactifications-and}}

\subsubsection{Generalities}

Our goal in this section is to show that the operator $\Theta$, defined
so far on $S_{s}^{\mathrm{ord}}$, extends to a partial compactification
$\overline{S}_{s}^{\mathrm{ord}}$, obtained by fixing a smooth toroidal
compactification $\overline{S}_{s}$ of $S_{s}$, and removing from
it the closure of $S_{s}^{\mathrm{no}}=S_{s}\smallsetminus S_{s}^{\mathrm{ord}}.$
Thus
\[
\overline{S}_{s}^{\mathrm{ord}}=\overline{S}_{s}\smallsetminus\{\text{Zariski closure of }S_{s}^{\mathrm{no}}\}
\]
is an open subset of $\overline{S}_{s}$. Note that in general the
closure of $S_{s}^{\mathrm{no}}$ may meet the boundary of $\overline{S}_{s}$,
although in some special cases, e.g. whenever $m=1$, $S_{s}^{\mathrm{no}}$
is proper and does not reach the cusps. For a characterization of
$\overline{S}_{s}^{\mathrm{ord}}$ as the non-vanishing locus of the
Hasse invariant see $\S$\ref{subsec:FJ hasse}. Once we extend $\Theta$,
we shall calculate its effect on Fourier-Jacobi expansions and show
that, as in the classical case of $GL_{2}$, it is \emph{morally}
given by ``$q\cdot d/dq$''.

The toroidal compactifications $\overline{\mathcal{S}}$ of $\mathcal{S}$
considered below are smooth over $\mathcal{O}_{E,(p)}$ and their
boundary $\partial\mathcal{S}=\overline{\mathcal{S}}\smallsetminus\mathcal{S}$
is a divisor with normal crossing. However, they depend on auxiliary
combinatorial data, and are not unique. As such, one can not expect
$\overline{\mathcal{S}}$ to solve a moduli problem anymore. The universal
abelian scheme $\mathcal{A}$ nevertheless \emph{extends} \emph{canonically}
to a semi-abelian scheme $\mathcal{G}$ with $\mathcal{O}_E$-action over $\overline{\mathcal{S}}$.
We say that a geometric point $x$ of $\partial\mathcal{S}$ is of
rank $1\le r\le m$ if the toric part of $\mathcal{G}_{x}$ has dimension
$2r,$ i.e. $\mathcal{O}_{E}$-rank $r$. Skinner and Urban \cite{S-U}
call such a point ``a point of genus $n+m-2r$'', referring to the
dimension of the abelian part of $\mathcal{G}_{x}$ instead.

Constructing the toroidal compactifications, even if all proofs are
omitted, requires several pages of definitions and notation. Lan's
book \cite{Lan} is an exhaustive, extremely careful and precise reference.
Unfortunately, some notation introduced there is too long to fit in
a single line. Following \cite{F-C}, Skinner and Urban gave a very
readable account of the compactification in $\S$5.4 of \cite{S-U}, which
we will follow closely. It is set for signature $(n,n)$, but the
modifications needed to treat an arbitrary signature $(n,m)$ are
minor. Yet, this forces us to review everything from scratch, rather
than use \cite{S-U} blindly.

We shall content ourselves with the arithmetical compactification
of $Sh_{K/W}$ (several copies of which comprise $\mathcal{S}_{/W})$.
In $\S$\ref{subsec:Toroidal-compactifications-and} only we will write
$S$ for $Sh_{K/W}$ or for its base-change to $W_{s}$ (rather than
to $\kappa=W_{1}$ as before). As smaller Shimura varieties will show
up in the process, we shall write 
\[
S=S_{G}=S_{G,K}
\]
whenever we need to emphasize the dependence on $G$ or $K$.

\medskip{}

Let $\{e_{i}\}$ denote the standard basis of $V=E^{n+m}$ and consider,
for $0\le r\le m,$
\[
0\subset V_{r}=\mathrm{Span}_{E}\{e_{1},\dots,e_{r}\}\subset V_{r}^{\perp}=\mathrm{Span}_{E}\{e_{1},\dots,e_{n},e_{n+r+1},\dots,e_{n+m}\}\subset V.
\]
If we regard $V=Res_{\mathbb{Q}}^{E}\mathbb{A}^{n+m}$ as a $\mathbb{Q}$-vector
group, whose $\mathbb{Q}$-rational points are $E^{n+m},$ this is
a $\mathbb{Q}$-rational filtration. The quotient $V(r)=V_{r}^{\perp}/V_{r}$
becomes a hermitian space of signature $(n-r,m-r)$ at infinity, and
$\Lambda\cap V_{r}^{\perp}$ projects to a self-dual lattice $\Lambda(r)\subset V(r),$
defining a smaller general unitary group $\boldsymbol{G}_{r}.$ If
$n=m=r$ we understand by $\boldsymbol{G}_{r}$ the group $\mathbb{G}_{m}$
(accounting for the similitude factor, which is present even if $V(r)=0$).

The subgroup
\[
P_{r}=\mathrm{Stab}_{\boldsymbol{G}}(V_{r})
\]
stabilizes also $V_{r}^{\perp}$, and is a maximal $\mathbb{Q}$-rational
parabolic subgroup of $\boldsymbol{G}$. Its unipotent radical is
\[
U_{r}=\{g\in P_{r}|\,g\text{ acts trivially on \ensuremath{V_{r},\,V(r),\text{and }V/V_{r}^{\perp}\}.}}
\]
Its Levi quotient, $L_{r}=P_{r}/U_{r}$, is identified with $Res_{\mathbb{Q}}^{E}GL_{r}\times\boldsymbol{G}_{r}$
under the map $g\mapsto(g|_{V_{r}},g|_{V(r)}).$ The center $Z_{r}=Z(U_{r})$
of $U_{r}$ turns out to be
\[
Z_{r}=\{g\in U_{r}|\,(g-1)(V_{r}^{\perp})=0,\,\,(g-1)(V)\subset V_{r}\}.
\]
In matrix block form
\begin{equation}
P_{r}=\left\{ g=\left(\begin{array}{ccc}
A & C & B\\
 & D & C'\\
 &  & \nu\,^{t}\overline{A}^{-1}
\end{array}\right)\in\boldsymbol{G}\right\} ,\label{eq:Parabolic}
\end{equation}
where $A$ is a square matrix of size $r,$ and $D$ is a square matrix
of size $(n+m-2r)$. The group $U_{r}$ is characterized by $\nu=1,\,A=1,\,D=1,$
and $Z_{r}$ by the additional properties $C=0,\,C'=0.$ When this
is the case, $B=-\,^{t}\overline{B}$. We regard $L_{r}$ also as
a subgroup of $P_{r}$, mapping $(g,h)$ to the matrix which in a diagonal block
form is $(g,h,\nu(h)\,^{t}\overline{g}^{-1})$. Thus $P_{r}=L_{r}U_{r}.$

Every maximal $\mathbb{Q}$-rational parabolic subgroup of $\boldsymbol{G}$
is conjugate to $P_{r}$ for some~$r$.

\subsubsection{Cusp labels, the minimal compactification and the toroidal compactifications}

Let $1\le r\le m.$ The set of cusp labels of level $K$ and rank
$r$ (\cite{S-U} $\S$5.4.2) is the finite set
\[
\mathscr{C}_{r}=[GL_{r}(E)\cdot\boldsymbol{G}_{r}(\mathbb{A}_{f})]\cdot U_{r}(\mathbb{A}_{f})\backslash\boldsymbol{G}(\mathbb{A}_{f})/K.
\]
As before, the rank $r$ will be the $\mathcal{O}_{E}$-rank of the
toric part of the universal semi-abelian variety over the corresponding
cuspidal component. If $g\in\boldsymbol{G}(\mathbb{A}_{f})$ we denote
by $[g]=[g]_{r}=[g]_{r,K}\in\mathscr{C}_{r}$ the corresponding double
coset. The minimal (Baily-Borel) compactification $S^{*}$ of $S$
is discussed in \cite{Lan} $\S$7.2.4 and, when $n=m$, in \cite{S-U}
$\S$5.4.4. It is a singular compactification admitting a stratification
by finitely many locally closed strata
\[
S^{*}=\bigsqcup_{r=0}^{m}\bigsqcup_{[g]_{r}\in\mathscr{C}_{r}}S_{\boldsymbol{G}_{r},K_{r,g}},
\]
where $K_{r,g}=\boldsymbol{G}_{r}(\mathbb{A}_{f})\cap gKg^{-1}.$
Each $S_{\boldsymbol{G}_{r},K_{r,g}}$ is an $(n-r)(m-r)$-dimensional
Shimura variety, so when $r$ attains its maximal value $m$, it is
$0$-dimensional. When $r=0$ we get one stratum, which is the open
dense $S$. The closure of $S_{\boldsymbol{G}_{r},K_{r,g}}$ is the
union of $S_{\boldsymbol{G}_{r'},K_{r',g'}}$for $r\le r'$ and $g'$
such that the cusp label $[g']_{r'}$ is a specialization of $[g]_{r}$
in an appropriate sense (\cite{Lan}, Definition 5.4.2.13). We call
each $S_{\boldsymbol{G}_{r},K_{r,g}}$ a \emph{rank $r$ cuspidal
component of $S^{*}.$}

Any toroidal compactification that we consider will be a smooth scheme
$\overline{S}_{/W}$ endowed with a proper morphism
\[
\pi:\overline{S}\to S^{*}.
\]
Moreover, it will come equipped with a stratification
\[
\overline{S}=\bigsqcup_{r=0}^{m}\bigsqcup_{[g]_{r}\in\mathscr{C}_{r}}\bigsqcup_{\sigma\in\Sigma_{H_{g,\mathbb{R}}^{++}}/\Gamma_g}Z([g]_{r},\sigma)
\]
by finitely many smooth, locally closed $W$-subschemes $Z([g]_{r},\sigma).$
The indexing set $\Sigma_{H_{g,\mathbb{R}}^{++}}/\Gamma_g$ will become
clear shortly. The morphism $\pi$ will respect the stratifications.

Every $Z([g]_{r},$$\sigma)$ is constructed in three steps, related
to the structure of the semi-abelian scheme $\mathcal{G}$ over it,
as follows.
\begin{itemize}
\item First, $S_{\boldsymbol{G}_{r},K_{r,g}}$ is the moduli space of the
abelian part of $\mathcal{G}$ (with the associated PEL structure),
which is of signature $(n-r,m-r)$, hence is a smooth Shimura variety
of dimension $(n-r)(m-r)$ over $W$. Let $\mathcal{A}_{r}$ denote the universal
abelian scheme over it. In contrast to the abelian part, the toric
part of $\mathcal{G}$ is \emph{fixed} by the cusp label $[g]_r$, and
is given by
\[
T_{X}=\mathrm{Hom}_{\mathcal{O}_{E}}(X,\mathfrak{d}_{E}^{-1}\otimes\mathbb{G}_{m})
\]
where $X=X_{g}$ is a rank-$r$ projective $\mathcal{O}_{E}$-module
determined by $g.$ Thus $\dim(T_{X})=2r$ and $\dim(\mathcal{A}_r)=n+m-2r$.  For example, if $g=1$
(the ``standard cusp of rank $r$'') then $X=\mathrm{Hom}(\Lambda\cap V_{r},\mathbb{Z}).$
\item The second step in the construction of $Z([g]_{r},\sigma)$ is the
construction of an abelian scheme $C$ which classifies the extensions
of $\mathcal{A}_{r}$ by $T_{X}$. Let $X^{*}=\Hom_{\mathcal{O}_{E}}(X,\mathcal{O}_{E})$
and
\[
C=C([g]_{r}):=\mathrm{Ext}_{\mathcal{O}_{E}}^{1}(\mathcal{A}_{r},T_{X}).
\]
This can be written also as
\[
C=X^{*}\otimes_{\mathcal{O}_{E}}\mathrm{Ext}_{\mathcal{O}_{E}}^{1}(\mathcal{A}_{r},\mathfrak{d}_{E}^{-1}\otimes\mathbb{G}_{m})=X^{*}\otimes_{\mathcal{O}_{E}}\mathcal{A}_{r}^{t}=\Hom_{\mathcal{O}_{E}}(X,\mathcal{A}_{r}^{t}),
\]
using the fact that $Tr_{E/\mathbb{Q}}\otimes1:\mathfrak{d}_{E}^{-1}\otimes\mathbb{G}_{m}\to\mathbb{G}_{m}$
induces an isomorphism
\[
\mathrm{Ext}_{\mathcal{O}_{E}}^{1}(\mathcal{A}_{r},\mathfrak{d}_{E}^{-1}\otimes\mathbb{G}_{m})\simeq\mathrm{Ext}^{1}(\mathcal{A}_{r},\mathbb{G}_{m})=\mathcal{A}_{r}^{t}.
\]
The relative dimension of $C$ over $S_{\boldsymbol{G}_{r},K_{r,g}}$
is $r(n+m-2r)$, so its total dimension is
\[
(n-r)(m-r)+r(n+m-2r)=nm-r^{2}.
\]
\item In the last and final step one uses auxiliary combinatorial data and
the theory of toroidal embeddings \cite{Fu} to construct the $Z([g]_{r},\sigma).$
Each of them is a torus torsor over $C([g]_{r})$. For details, see
the next subsection.
\end{itemize}

\vspace{0.2cm}

\noindent The stratification by disjoint locally closed strata does not shed
any light on the way these strata are glued together, even if the
closure relations between them are given. However, each stratum $Z=Z([g]_r,\sigma)$
is actually the underlying reduced scheme (the ``support'') of a
formal scheme $\mathfrak{\mathfrak{Z}=Z}([g]_r,\sigma)$ whose over-all
dimension (counting the ``formal parameters'' too) is $mn.$ The
semi-abelian scheme together with the PEL structure extend from $Z$
to $\mathfrak{Z}$ ``in the infinitesimal directions'' to give a
structure called \emph{degeneration data}. As described originally
in \cite{Mu} in the totally degenerate setting, and later on in \cite{A-M-R-T},
\cite{F-C} and \cite{Lan}, this allows one to use\emph{ Mumford's
construction} to glue all the pieces together. We do not reproduce
this construction, but remark that the key to it
is the presence of a polarization, which allows, at a crucial step,
to use Grothendieck's algebraization theorem.

\subsubsection{The torsor $\Xi$}

As our purpose is to establish just enough notation to be able to
study $\Theta$ at the cusps, and as this will be done only at the
\emph{standard} cusps, we shall explain now the third and final step
in the construction of $\mathfrak{Z}([g]_{r},\sigma)$ under the assumption
that $g=1.$ The general case can be treated in a similar manner,
transporting all structures by $g.$ While necessary for applications,
it does not add much conceptually.

Assume therefore that the cusp label is $[g]_{r}=[1]_{r}$ and drop
the $g$ from the notation. Let 
\[
X=\mathrm{Hom}(\Lambda\cap V_{r},\mathbb{Z}),\,\,\,\,Y=\Lambda/(\Lambda\cap V_{r}^{\perp}).
\]
 Let $\phi_{X}:Y\simeq X$ be the isomorphism given by $\phi_{X}(u)(v)=\left\langle u,v\right\rangle $.
It satisfies $\phi_{X}(au)=\overline{a}\phi_{X}(u).$ If $c\in C=\mathrm{Hom}_{\mathcal{O}_{E}}(X,\mathcal{A}_{r}^{t})$
we denote by $c^{t}\in\mathrm{Hom}_{\mathcal{O}_{E}}(Y,\mathcal{A}_{r})$
the unique homomorphism satisfying $\phi_{r}\circ c^{t}=c\circ\phi_{X},$
where $\phi_{r}:\mathcal{A}_{r}\simeq\mathcal{A}_{r}^{t}$ is the
tautological principal polarization of the abelian scheme $\mathcal{A}_{r}$
over $S_{\boldsymbol{G}_{r},K_{r,g}}.$

\[ \xymatrix{X \ar[r]^c & \mathcal{A}_r^t\\ Y \ar[r]^{c^t}\ar[u]^{\phi_X} & \mathcal{A}_r\ar[u]_{\phi_r}}\]

\medskip{}

We construct a torus $T_{H}$ and use it to define a $T_{H}$-torsor
$\Xi$ over $C$ which will be basic for the construction of the local
charts below. Let
\[
H=Z_{r}(\mathbb{Q})\cap K
\]
where $Z_{r}$, as before, is the center of the unipotent radical
of $P_{r}$, and $K$ the level subgroup. Let $\check{H}=\mathrm{Hom}_{\mathbb{Z}}(H,\mathbb{Z})$
and
\[
T_{H}=H\otimes\mathbb{G}_{m/W}=\mathrm{Spec}(W[\check{H}]),
\]
the split torus over the Witt vectors with character group\footnote{{[}S-U{]} denote $\check{H}$ by $S$.}
$\check{H}$ and cocharacter group~$H$. There is another useful way
to think of $H,$ as a rank-$r^{2}$ lattice of hermitian bilinear
forms on $Y$ (the lattice shrinking as the level increases), \emph{cf.}
\cite{S-U} $\S$5.4.1. Simply attach to $h\in H$ the hermitian form
$b_{h}:Y\times Y\to\mathfrak{d}_{E}^{-1}$ defined by
\begin{equation}
b_{h}(y,y')=\delta_{E}^{-1}((h-1)y,y').\label{eq:b_h}
\end{equation}
Here $(,)$ is the pairing on $V_{r}\times(V/V_{r}^{\perp})$ induced
from (\ref{eq:hermitian_pairing}). Using the description of $Z_{r}$
in (\ref{eq:Parabolic}) we may regard $h\mapsto b_{h}$ as assigning
to $h\in H$ the matrix $\delta_{E}^{-1}B.$

\medskip{}

We denote by $\Xi$ the $T_{H}$-torsor over $C=\mathrm{Hom}_{\mathcal{O}_{E}}(X,\mathcal{A}_{r}^{t})$
constructed in \cite{S-U}, smooth of total dimension $mn.$ Recall
that given such a torsor, every character $\chi\in\check{H}$ of $T_{H}$
determines, by push-out, a $\mathbb{G}_{m}$-torsor $\Xi_{\chi}$
over $C,$ and the resulting map
\[
\chi\mapsto[\Xi_{\chi}]
\]
from $\check{H}$ to the group of $\mathbb{G}_{m}$-torsors over $C$
is a homomorphism. Conversely, $\Xi$ is uniquely determined by giving
such a homomorphism. We proceed to describe $\Xi$ in this way. 

If $y,y'\in Y$ let $\chi=[y\otimes y']$ denote the element of $\check{H}$
which sends
\[
H\ni h\mapsto Tr_{E/\mathbb{Q}}b_{h}(y,y')=\left\langle (h-1)y,y'\right\rangle \in\mathbb{Z}.
\]
Then we require $\Xi_{\chi}|_{c},$ the fiber at $c\in C$ of $\Xi_{\chi},$
to be
\[
\mathcal{P}|_{c(\phi_{X}(y))\times c^{t}(y')}^{\times},
\]
where $\mathcal{P}$ is the Poincar{\'e} bundle over $\mathcal{A}_{r}^{t}\times\mathcal{A}_{r}.$
The superscript $\times$ means ``the associated $\mathbb{G}_{m}$-bundle'',
obtained by removing the zero section. It can be checked that this
extends to a homomorphism from $\check{H}$ to the group of $\mathbb{G}_{m}$-torsors
over $C.$ For any $\chi\in\check{H}$ we let $\mathcal{L}(\chi)$
be the line bundle on $C$ whose associated $\mathbb{G}_{m}$-bundle
is $\Xi_{\chi}$. Over the complex numbers, sections of $\mathcal{L}(\chi)$
are classical theta functions on the abelian scheme $C$. We shall
often denote elements of $\check{H}$ also by $\check{h}$. We have
a canonical identification $\mathcal{L}(\check{h}_{1}+\check{h}_{2})=\mathcal{L}(\check{h}_{1})\otimes\mathcal{L}(\check{h}_{2})$.

\medskip{}

Having constructed $\Xi$ we proceed to study its equivariance properties
under the group
\[
\Gamma=GL(V_{r})(\mathbb{Q})\cap K.
\]
Using (\ref{eq:Parabolic}), this is the group of rational matrices
$A$ that also lie in $K.$ Since the action of $P_{r}$ on $Z_{r}$
by conjugation factors through $P_{r}/U_{r}=L_{r},$ the group $\Gamma\subset L_{r}$
acts on $Z_{r}.$ Using (\ref{eq:Parabolic}) again, $A$ sends $B$
to $AB\,^{t}\overline{A}.$ In particular $\Gamma$ acts on $H$,
hence it acts on $T_{H}$ by automorphisms of the torus. 

We also have an action of $\Gamma$ on $C=\mathrm{Hom}_{\mathcal{O}_{E}}(X,\mathcal{A}_{r}^{t})$
induced from its action on~$X.$ Any $\gamma\in\Gamma$ maps $\mathcal{L}(\check{h})|_{c}$
to $\mathcal{L}(\gamma(\check{h}))|_{\gamma(c)}.$ If $\Gamma(\check{h})$
is the stabilizer of $\check{h}\in\check{H}$ then (\cite{S-U} Lemma
5.1) $\Gamma(\check{h})$ acts trivially on the global sections
of $\mathcal{L}(\check{h})$ over $C$. 

Finally, as the push-out of $\Xi|_{\gamma(c)}$ by $[\gamma(y)\otimes\gamma(y')]$
is identically the same as the push out of $\Xi|_{c}$ by $[y\otimes y']$,
or equivalently
\[
\Xi|_{\gamma(c)}=(\Xi\times^{T_{H},\gamma}T_{H})|_{c}
\]
the isomorphism $1\times\gamma:\Xi=\Xi\times^{T_{H}}T_{H}\to\Xi\times^{T_{H},\gamma}T_{H}$
of torsors \emph{over $C$}, yields an action of $\Gamma$ on $\Xi$
which \emph{covers} its action on $C$, and is compatible with the
$\Gamma$-action on $T_{H}$. In short, all the constructions so far
are equivariant under $\Gamma$.

\subsubsection{The local charts\label{subsec:The-local-charts}}

Now comes the choice of the auxiliary data involved in the toroidal
compactification. Let
\[
H_{\mathbb{R}}^{+}\subset H_{\mathbb{R}}
\]
be the cone of positive semi-definite hermitian bilinear forms on
$Y_{\mathbb{R}}$ whose radical is a subspace defined over $\mathbb{Q}$
(i.e. the $\mathbb{R}$-span of a subspace of $Y_{\mathbb{Q}}).$
Let $\Sigma=\{\sigma\}$ be a $\Gamma$-admissible (infinite) rational
polyhedral cone decomposition of $H_{\mathbb{R}}^{+}$ (\cite{Lan},
Definition 6.1.1.10). Admissibility means that the action of $\Gamma$
on $H_{\mathbb{R}}$ permutes the $\sigma$'s, and that modulo $\Gamma$
there are only finitely many cones in $\Sigma$. By convention, the
cones $\sigma$ do not contain their proper faces, and every face
of a cone in $\Sigma$ also belongs to $\Sigma$. In particular, $\Sigma$
contains the origin as its unique 0-dimensional cone. When we treat
all cusp labels, and not only one at a time, an additional assumption
has to be imposed about the compatibility of the polyhedral cone decompositions
associated with a cusp $\xi$ and with a higher rank cusp to which
$\xi$ specializes. It is a non-trivial fact that such polyhedral
cone decompositions exist, see Chapter 2 of \cite{A-M-R-T}.
Moreover, every two $\Gamma$-admissible rational polyhedral cone
decompositions of $H_{\mathbb{R}}^{+}$ have a common refinement of
the same sort. One can even find such a polyhedral cone decomposition
in which every $\sigma$ is spanned by a part of a basis of $H$.
The $T_{H,\sigma}$ defined below will then be smooth over $W$, and
from now on we assume that this is the case. Lan \cite{Lan} calls
such a $\Sigma$ a $\Gamma$\emph{-admissible} \emph{smooth rational
polyhedral cone decomposition of $H_{\mathbb{R}}^{+}$.} If $K$ is
small enough so that $\Gamma$ is \emph{neat}, refinements exist such
that, in addition, the closures of $\sigma$ and $\gamma(\sigma)$,
for $\sigma\in\Sigma$ and $1\ne\gamma\in\Gamma$, meet only at the
origin.

Each cone $\sigma\in\Sigma$ defines a torus embedding
\[
T_{H}\hookrightarrow T_{H,\sigma}=\mathrm{Spec}(W[\check{H}\cap\sigma^{\vee}])
\]
where $\sigma^{\vee}\subset\check{H}_{\mathbb{R}}$ is the dual cone
and $W=W(\kappa)$ as before. By definition
\[
\sigma^{\vee}=\{v\in\check{H}_{\mathbb{R}}|\,v(u)\ge0,\,\forall u\in\sigma\},
\]
so, unlike $\sigma,$ $\sigma^{\vee}$ contains its faces. Observe
that $T_{H}$ naturally acts on $T_{H,\sigma}$. Since $\sigma$ does
not contain a line, $\sigma^{\vee}$ has a non-empty interior. 

Let
\[
\sigma^{\perp}=\{v\in\check{H}_{\mathbb{R}}|\,v(u)=0,\,\forall u\in\sigma\}.
\]
When $d_{\sigma}=\dim(\sigma)<r^{2}$, $\sigma^{\vee}\supset\sigma^{\perp}\ne0$.
Then $Z_{H,\sigma}=\mathrm{Spec}(W[\check{H}\cap\sigma^{\perp}])$
is a torus, $\dim Z_{H,\sigma}=r^{2}-d_{\sigma}$. In fact, $Z_{H,\sigma}$
is the unique minimal orbit of $T_{H}$ in its action on $T_{H,\sigma},$
an orbit which lies in the closure of any other orbit. There is an
obvious surjection $T_{H,\sigma}\twoheadrightarrow Z_{H,\sigma}.$
This surjection admits a section $Z_{H,\sigma}\hookrightarrow T_{H,\sigma}$,
corresponding to $W[\check{H}\cap\sigma^{\perp}]\simeq W[\check{H}\cap\sigma^{\vee}]/I_{\sigma}$,
where $I_{\sigma}$ is the ideal generated by $\check{H}\cap\sigma^{\vee}\smallsetminus\check{H}\cap\sigma^{\perp}.$
Another way to think of $Z_{H,\sigma}$ is as
\[
Z_{H,\sigma}=T_{H,\sigma}\smallsetminus\bigcup_{\tau<\sigma}T_{H,\tau}
\]
where $\tau$ runs over all the \emph{proper} faces of $\sigma$.

The $T_{H,\sigma}$ glue to form a toric variety (locally of finite
type, but not of finite type in general) $T_{H,\Sigma},$ in which
each $T_{H,\sigma}$ is open and dense:
\[
T_{H,\Sigma}=\bigcup_{\sigma\in\Sigma}T_{H,\sigma}.
\]
This $T_{H,\Sigma}$ is stratified by the disjoint union of the $Z_{H,\sigma}$.
The actions of $\Gamma$ on $H$ and $\Sigma$ induce an action of
$\Gamma$ on $T_{H}$ and a compatible action on $T_{H,\Sigma}.$
By our assumption on $\Sigma,$ $T_{H,\Sigma}$ is smooth over $W$.

We ``spread'' this construction over $C=\mathrm{Hom}_{\mathcal{O}_{E}}(X,\mathcal{A}_{r}^{t})$,
twisting it by the torsor $\Xi$, namely we consider
\begin{equation}
\overline{\Xi}_{\Sigma}=\Xi\times^{T_{H}}T_{H,\Sigma}.\label{eq:Xi_bar}
\end{equation}
The group $\Gamma$ acts on each of the three symbols on the right
in a compatible way, so we get an action of $\Gamma$ on $\overline{\Xi}_{\Sigma}$. 

\bigskip{}

Let us bring back the reference to the cusp label $[g]_{r}$, although
in the above we tacitly assumed $[g]_{r}=1$ and dropped $g$ from
the notation. See \cite{S-U}, $\S$5.4.1 for the precise definition of
$H_{g},\Sigma_{g}$ etc. Denote by $H_{g,\mathbb{R}}^{++}$ the set
of \emph{positive-definite} hermitian bilinear forms in $H_{g,\mathbb{R}}^{+}.$
For $\sigma\in\Sigma_{g}$ such that $\sigma\subset H_{g,\mathbb{R}}^{++}$
we let 
\[
Z([g]_{r},\sigma)=\Xi\times^{T_{H}}Z_{H,\sigma},
\]
 and let 
\[
\mathfrak{Z}([g]_{r},\sigma)
\]
be the formal completion of $\overline{\Xi}_{\Sigma}$ or, what amounts
to be the same, of its open subset $\Xi\times^{T_{H}}T_{H,\sigma}$,
along $Z([g]_{r},\sigma)$. These are the local charts at the cuspidal
component labeled by $[g]_{r}$. There is a smooth morphism 
\[
\mathfrak{Z}([g]_{r},\sigma)\to C([g]_{r})
\]
whose fibers are isomorphic to the completion of $T_{H,\sigma}$ along
$Z_{H.\sigma}$. The $\mathfrak{Z}([g]_{r},\sigma)$ are $nm$-dimensional
and smooth over $W$. Each such local chart has $nm-d_{\sigma}$ ``algebraic
dimensions'' and $d_{\sigma}$ ``formal dimensions''. Specializing
the formal variables to $0,$ one gets the support $Z([g]_{r},\sigma)$
of $\mathfrak{Z}([g]_{r},\sigma)$, whose dimension is $nm-d_{\sigma}.$
The action of $\gamma\in\Gamma$ on $\overline{\Xi}_{\Sigma}$ induces
an isomorphism $\gamma_{*}$ between $\mathfrak{Z}([g]_{r},\sigma)$
and $\mathfrak{Z}([g]_{r},\gamma(\sigma))$. For comparison, we remark
that in \cite{Lan} $\S$6.2.5 the $\mathfrak{Z}([g]_{r},\sigma)$ are
denoted $\mathfrak{X}_{\Phi_{\mathcal{H}},\delta_{\mathcal{H}},\sigma}$
and $Z([g]_{r},\sigma)$ are denoted $\Xi_{\Phi_{\mathcal{H}},\delta_{\mathcal{H}},\sigma}$.
Also, under our assumptions the stabilizers denoted in \cite{Lan}
by $\Gamma_{\Phi_{\mathcal{H}},\sigma}$ are trivial.

Once we have described the local charts, it remains to construct on
each of them the \emph{degeneration data} which allows one to carry
on the \emph{Mumford construction}. This results in gluing the various
charts together, and at the same time constructing $\mathcal{G}$
with the accompanying PEL structure over the glued scheme. Care has
to be taken not only to glue pieces labeled by the same cusp label
$[g]_{r}$, but also to respect the way cusp labels specialize. In
the process of gluing, one has to divide by the action of $\Gamma$
on the formal completion of (\ref{eq:Xi_bar}) along the complement
of $\Xi=\Xi\times^{T_{H}}T_{H}$. Note that it does not make sense
to divide $\overline{\Xi}_{\Sigma}$ by $\Gamma$, just as it did
not make sense to divide $\Xi,$ or the abelian scheme $C$ over which
it lies, by the action of $\Gamma$. For the gluing of the local charts,
that we do not review here, see \cite{Lan} $\S$6.3. The final result
is \cite{Lan} Theorem 6.4.1.1.

\subsubsection{Logarithmic differentials}

We construct certain formal differentials on the local chart $\mathfrak{Z}([g]_{r},\sigma),$
relative to $C([g]_{r})$, with logarithmic poles along $Z([g]_{r},\sigma)$.
We shall denote the module of these differentials
\[
\Omega_{\mathfrak{Z}/C}[d\log\infty].
\]
They will play an important role in our formulae for $\Theta$.

Notation as above, consider a cone $\sigma\subset H_{g,\mathbb{R}}^{++}$
and let $h_{1},\dots,h_{d_{\sigma}}$ be positive semi-definite, part
of a basis of $H=H_{g}$, such that
\[
\sigma=Cone(h_{1},\dots,h_{d_{\sigma}}).
\]
Complete the $h_{i}$ to a basis $h_{1},\dots,h_{r^{2}}$ of $H$,
let $\{\check{h}_{i}\}$ be the dual basis of $\check{H}=\mathrm{Hom}(H,\mathbb{Z})$
and introduce formal variables $q_{i}=q^{\check{h}_{i}}$ (to be able
to write the group structure on $\check{H}$ multiplicatively rather
than additively). Then
\[
T_{H,\sigma}=\mathrm{Spec}(W[q_{1},\dots,q_{d_{\sigma}},q_{d_{\sigma}+1}^{\pm1},\dots,q_{r^{2}}^{\pm1}])
\]
and
\[
Z_{H,\sigma}=\mathrm{Spec}(W[q_{d_{\sigma}+1}^{\pm1},\dots,q_{r^{2}}^{\pm1}]).
\]
Locally on $\mathfrak{Z}([g]_{r},\sigma)$ we use as coordinates the
pull-back of any system of $nm-r^{2}$ local coordinates on the base
$C=\mathrm{Hom}_{\mathcal{O}_{E}}(X,\mathcal{A}_{r}^{t}),$ together
with the ``algebraic'' coordinates $q_{d_{\sigma}+1},\dots,q_{r^{2}}$,
and the ``formal'' coordinates $q_{1},\dots,q_{d_{\sigma}}.$ We
emphasize that because of the twist by the torsor $\Xi$ in the construction
of the local charts, the $q_{i}$ are not global coordinates. The
correct way to think of them is as local sections of the line bundles
$\mathcal{L}(-\check{h}_{i})$ on $C$. If the $h_{i}$ are positive
definite, these line bundles will be anti-ample, and the $q_{i}$
will not globalize. 

If $\check{h}\in\check{H}$ is of the form $\check{h}=\sum n_{i}\check{h}_{i}$
we write $q^{\check{h}}=\prod q_{i}^{n_{i}}$ and define
\[
\omega(\check{h})=\frac{dq^{\check{h}}}{q^{\check{h}}}=\sum_{i=1}^{r^{2}}n_{i}\frac{dq_{i}}{q_{i}}\in\Omega_{\mathfrak{Z}/C}[d\log\infty].
\]
This $\omega(\check{h})$ is invariant under the action of $T_{H}$,
essentially since $d\log(q_{0}q)=d\log q$. Hence, despite the fact
that the $q_{i}$ were only \emph{local} coordinates, $\omega(\check{h})$
defines a relative differential on all of $\Xi\times^{T_{H}}T_{H,\sigma}$,
as well as on its completion $\mathfrak{Z}([g]_{r},\sigma)$ along
$Z([g]_{r},\sigma)=\Xi\times^{T_{H}}Z_{H,\sigma}$, with logarithmic
poles along $Z([g]_{r},\sigma)$. The following proposition is an
immediate by-product of the theory of toroidal compactifications.
\begin{prop}
(i) The differentials $\omega(\check{h})$ are well-defined formal
differentials on $\mathfrak{Z}([g]_{r},\sigma)$, relative to $C([g]_{r}),$
with logarithmic poles along $Z([g]_{r},\sigma)$. They are independent
of the choice of bases and depend only on $\check{h}$.

(ii) $\omega(\check{h}_{1}+\check{h}_{2})=\omega(\check{h}_{1})+\omega(\check{h}_{2}).$

(iii) The differentials $\omega(\check{h})$ are compatible with gluing
of the local charts. If $\gamma\in\Gamma$ then the induced isomorphism
between the local charts $\mathfrak{Z}([g]_{r},\sigma)$ and $\mathfrak{Z}([g]_{r},\gamma(\sigma))$
carries $\omega(\check{h})$ to $\omega(\gamma(\check{h}))$.

(iv) The differentials $\omega(\check{h})$ are compatible with the
maps between toroidal compactifications obtained from refinements
of the admissible smooth rational polyhedral cone decompositions,
as in \cite{Lan} $\S$6.4.2.
\end{prop}

\subsubsection{Fourier-Jacobi expansions\label{subsec:Fourier-Jacobi-expansions}}

Let $\overline{S}$ be a fixed smooth toroidal compactification of
$S$ over $W_{s}$ ($1\le s$) as a base ring. Let $\mathcal{G}$
be the universal semi-abelian scheme over $\overline{S}$ and $e_{\mathcal{G}}:\overline{S}\to\mathcal{G}$
its zero section. Then $\omega=e_{\mathcal{G}}^{*}\Omega_{\mathcal{G}/\overline{S}}^{1}$
defines an extension of the Hodge bundle to a rank $n+m$ vector bundle
with $\mathcal{O}_{E}$-action on $\overline{S}$. We continue to
denote by $\mathcal{P}$ and $\mathcal{Q}$ its sub-bundles of type
$\Sigma$ and $\overline{\Sigma}$, of ranks $n$ and $m$ respectively.

Let $\overline{S}^{\mathrm{ord}}$ denote the complement in $\overline{S}$
of the Zariski closure of $S\smallsetminus S^{\mathrm{ord}}.$ Over
this open subset of $\overline{S}$ the semi-abelian variety $\mathcal{G}$
is $\mu$-ordinary in the sense that the connected part of its $p$-divisible
group at every geometric point $x:\mathrm{Spec}(k)\to\overline{S}^{\mathrm{ord}}$
satisfies
\[
\mathcal{G}_{x}[p^{\infty}]^{0}\simeq(\mathfrak{d}_{E}^{-1}\otimes\mu_{p^{\infty}})^{m}\times\mathfrak{G}_{k}^{n-m}.
\]
To see this, assume that $x$ lies on a rank $r$ cuspidal component,
but that the abelian part $\mathcal{A}_{x}$ of $\mathcal{G}_{x}$
is not $\mu$-ordinary, i.e. the multiplicative part of $\mathcal{A}_{x}[p^{\infty}]$
has height strictly less than $2(m-r).$ Mumford's construction shows
that we may deform $\mathcal{G}$ into an abelian variety $\mathcal{A}_{y}$
($y$ signifying a point on the base of the deformation ``near''
$x$) so that the multiplicative part of $\mathcal{A}_{y}[p^{\infty}]$
has height strictly less than $2m.$ But such a point $y$ being non-$\mu$-ordinary,
we conclude that $x$ lies in the closure of $S\smallsetminus S^{\mathrm{ord}}$,
contrary to our assumption.

It follows that the filtration
\[
0\to\mathcal{P}_{0}\to\mathcal{P}\to\mathcal{P}_{\mu}\to0
\]
extends to a filtration by a sub-vector bundle over $\overline{S}^{\mathrm{ord}}$.
Thus the automorphic vector bundles $\mathcal{E}_{\rho}$ defined
in $\S$\ref{subsec:auotomorpic _vector_bundles_mod_p} extend to $\overline{S}^{\mathrm{ord}}$
too. In the following discussion fix the representation $\rho.$

Let $[g]_{r}\in\mathscr{C}_{r}$ be a cusp label of rank $0<r\le m$
and let $Z=Z([g]_{r})$ be the corresponding cuspidal component of
$\partial S$ obtained by ``gluing'' the $Z([g]_{r},\sigma)$ for
$\sigma\in\Sigma_{g},\,\,\sigma\subset H_{g,\mathbb{R}}^{++}$, and
dividing by $\Gamma$. Let $\mathfrak{Z}([g]_{r})$ be the formal
completion of $\overline{S}$ along $Z([g]_{r}).$ Let $\xi\in S_{\boldsymbol{G}_{r},K_{r,g}}$
be a geometric point, and let $Z_{\xi}$ be the pre-image of $\xi$
in $Z$. Then $Z_{\xi}$ is obtained by ``gluing'' the pre-image
$Z([g]_{r},\sigma)_{\xi}$ of $\xi$ in $Z([g]_{r},\sigma$) for all
$\sigma$ as above, dividing by the action of $\Gamma$. Observe that
the toric part $T_{X}$ and the abelian part $\mathcal{A}_{r,\xi}$
of $\mathcal{G}$ are constant over each $Z([g]_{r},\sigma)_{\xi}$.
Thus $\mathcal{P}_{0},\mathcal{P}_{\mu}$ and $\mathcal{Q}$ are trivialized
over the pre-image $\mathfrak{Z}([g]_{r},\sigma)_{\xi}$ of $\xi$
in the local chart $\mathfrak{Z}([g]_{r},\sigma)$, hence so is $\mathcal{E}_{\rho}$.
In general, however, $\mathcal{E}_{\rho}$ will not be trivial over
$\mathfrak{Z}([g]_{r})_{\xi}$.

Our $\xi$ is a point of the minimal compactification $S^{*}$ (over
$W_{s}$). The completed local ring $\widehat{\mathcal{O}}_{S^{*},\xi}$
is described in \cite{S-U} Theorem 5.3 and \cite{Lan} Proposition
7.2.3.16. In the following, let $\check{H}^{+}$ be the set of elements
of $\check{H}$ which are non-negative on $H_{\mathbb{R}}^{+}.$
\begin{prop}
\label{prop:FJ_expansions_functions}There is a canonical isomorphism
between $\widehat{\mathcal{O}}_{S^{*},\xi}$ and the ring $\mathcal{FJ}_{\xi}$
of all formal power series
\[
f=\sum_{\check{h}\in\check{H}^{+}}a(\check{h})q^{\check{h}}
\]
which are invariant under $\Gamma$. Here $a(\check{h})\in H^{0}(C_{\xi},\mathcal{L}(\check{h}))$
where $C_{\xi}\mathrm{=Hom}_{\mathcal{O}_{E}}(X,\mathcal{A}_{r,\xi}^{t})$
is the abelian variety which is the fiber of $C$ over $\xi$.
\end{prop}

Recall that $\pi:\overline{S}\to S^{*}$ was the map between the toroidal
compactification and the minimal one. There is a similar description
of the completion of the stalk of $\pi_{*}\mathcal{E}_{\rho}$ at
$\xi$ (\cite{S-U} Proposition 5.5).
\begin{prop}
\label{prop:FJ_expansions}The completion of $(\pi_{*}\mathcal{E}_{\rho})_{\xi}$
is canonically isomorphic to the $\widehat{\mathcal{O}}_{S^{*},\xi}$-module
of formal power series
\[
f=\sum_{\check{h}\in\check{H}^{+}}a(\check{h})q^{\check{h}}
\]
which are invariant under $\Gamma$. Here $a(\check{h})\in H^{0}(C_{\xi},\mathcal{L}(\check{h})\otimes{\mathcal{E}_{\rho}})$.
\end{prop}

The action of $\Gamma$ on $a(\check{h})$ demands an explanation,
and for that we must bring back the dependence on $[g]_{r}\in\mathscr{C}_{r}$
and even on $g$ itself. Still assuming that we are at the standard
cusp, i.e. $[g]_{r}=[1]_{r},$ we may replace the representative $g=1$
by $g=\gamma\in\Gamma=GL(V_{r})\cap K=GL_{r}(E)\cap K.$ The following
changes then take place. The lattice $\Lambda\cap V_{r}$ is replaced
by $\gamma(\Lambda\cap V_{r})=\Lambda\cap V_{r},$ so does not change.
The subgroups $X$ and $Y$ therefore remain the same, but $\gamma$
acts on them non-trivially. This induces an action of $\gamma$ on
the abelian variety $C_{\xi}=\mathrm{Hom}_{\mathcal{O}_{E}}(X,\mathcal{A}_{r,\xi}^{t})$
classifying the extensions $\mathcal{G}$ of $\mathcal{A}_{r,\xi}$
by $T_{X}$, as well as an action on the torus $T_{X}$. Thus $\gamma$
induces an isomorphism
\[
\gamma_{*}:\mathcal{G}_{c}\simeq\mathcal{G}_{\gamma(c)}
\]
($c\in C_{\xi}),$ which on the toric part is the given automorphism
of $T_{X}$, and on the abelian part induces the identity. This induces
isomorphisms $\gamma_{*}=(\gamma^{*})^{-1}$ from the fibers of $\mathcal{P}_{0},\mathcal{P}_{\mu}$
and $\mathcal{Q}$ at $c$ to the corresponding fibers at $\gamma(c)$.
As $\mathcal{P}_{0}$ depends only on the abelian part, the action
of $\gamma_{*}$ on it is trivial. Assume, for simplicity, that $\mathcal{E}_{\rho}=\mathcal{P}_{\mu}.$
Then $a(\check{h})$ is a section (over $C_{\xi})$ of $\mathcal{L}(\check{h})\otimes\mathcal{P}_{\mu}$
and $\gamma(a(\check{h}))$ will be the section of $\mathcal{L}(\gamma\check{h})\otimes\mathcal{P}_{\mu}$
satisfying
\[
\gamma(a(\check{h}))|_{\gamma(c)}=\gamma_{*}(a(\check{h})|_{c}).
\]

We also remark that in (\cite{S-U} Proposition 5.5) the automorphic vector bundle is incorrectly assumed to be constant
along $C_{\xi}$. We thank one of the referees for pointing this out to us. However, in one important case, that will be needed
below, this is true. If $\xi$ is a rank $m$ cusp, the three basic automorphic vector bundles $\mathcal{Q}, \mathcal{P}_{\mu}$
and $\mathcal{P}_0$  depend only on the toric part and the abelian part of the universal semi-abelian scheme separately, and (unlike
$\mathcal{P}$) do not depend on the extension class parametrized by $C_{\xi}$. This implies that they are constant along $C_{\xi}$ and so is
every $p$-adic automorphic vector bundle generated by them.

In the sequel we shall only need the case of the \emph{maximally degenerate}
cusps, i.e. $r=m.$ Now the Shimura variety $S_{\boldsymbol{G}_{m},K_{m,g}}$
is $0$-dimensional, and $\xi$ is one of its (schematic) points.
The abelian variety $C_{\xi}$ is $m(n-m)$-dimensional. In this case
$\mathcal{P}_{\mu}$ is the $\Sigma$-part of the cotangent space
at the origin to 
\[
T_{X}=\mathrm{Hom}_{\mathcal{O}_{E}}(X,\mathfrak{d}_{E}^{-1}\otimes\mathbb{G}_{m})
\]
(if $r<m$ it also captures part of the cotangent space at the origin
of $\mathcal{A}_{r,\xi}$). In other words, we may identify
\[
\mathcal{P}_{\mu}\simeq\mathcal{O}_{C_{\xi}}\otimes_{\Sigma,\mathcal{O}_{E}}X
\]
and $\gamma_{*}:\mathcal{P}_{\mu}|_{c}\to\mathcal{P}_{\mu}|_{\gamma(c)}$
with $\gamma_{*}\otimes\gamma_{*}.$ Similarly we may identify $\mathcal{Q}\simeq\mathcal{O}_{C_{\xi}}\otimes_{\overline{\Sigma},\mathcal{O}_{E}}X$.
As the action of $\gamma$ on $X$ is via the contragredient $\mathrm{st}^{\vee}$ of
the standard representation, it follows that to obtain the action
of $\gamma\in\Gamma$ on $\rho(W)$ in general, we have first to embed
$\gamma$ as $\iota^{\vee}(\gamma):=(\,^{t}\overline{\gamma}^{-1},\,^{t}\gamma^{-1},1)$
in $GL_{m}\times GL_{m}\times GL_{n-m}$. (Recall that these three
factors correspond to $\mathcal{Q},\mathcal{P}_{\mu}$ and $\mathcal{P}_{0}$
in this order, see \S\ref{subsec:auotomorpic _vector_bundles_mod_p}.)
The action of $\gamma$ on $a(\check{h})\in H^{0}(C_{\xi},\mathcal{L}(\check{h}))\otimes_{W}\rho(W)$
will then be via $\gamma_{*}\otimes\rho(\iota^{\vee}(\gamma)).$

We remark that when $n=m$ the Fourier-Jacobi expansions are in fact
Fourier expansions in the naive sense, and the $a(\check{h})$ are
scalars.

\subsubsection{The Fourier-Jacobi expansion of the Hasse invariant\label{subsec:FJ hasse}}

Assume now that $s=1,$ i.e. we are again over the special fiber in
characteristic $p,$ and the automorphic vector bundle is the line
bundle $\mathcal{L}^{p^{2}-1}$ where $\mathcal{L}=\det(\mathcal{Q}).$
Let $h\in H^{0}(S,\mathcal{L}^{p^{2}-1})$ be the Hasse invariant,
previously denoted $h$ (\ref{eq:hasse_inv}).
\begin{prop}
The Fourier-Jacobi expansion of $h$ at a rank-$m$ cusp is $1$.
\end{prop}

\begin{proof}
Let us check the claim at the standard cusp. Fix a local chart $\mathfrak{Z}([1]_{m},\sigma)$
as above. As we have seen, $\mathcal{Q}$, hence also $\mathcal{L}$,
are trivialized there. The trivialization is obtained from a similar
trivialization of the $p$-divisible group of the toric part $T_{X}$
of the semi-abelian variety $\mathcal{G}$. As the isogeny $\mathrm{Ver}$
acts like the identity on $T_{X}[p^{\infty}],$ the Hasse invariant
maps a trivializing section $\ell$ of $\mathcal{L}$ over $\mathfrak{Z}([1]_{m},\sigma)$
to $\ell^{(p^{2})}$. It follows that in terms of the basis $\ell^{p^{2}-1}$
of $\mathcal{L}^{p^{2}-1}$, its FJ expansion is simply $1$. Note
that a choice of another $\kappa$-rational section $\ell$ will result
in the same value for $h$. 
\end{proof}
\begin{cor}
\label{cor:vanishing locus of Hasse}The open set $\overline{S}^{\mathrm{ord}}\subset\overline{S}$
is the non-vanishing locus of $h$.
\end{cor}

\begin{proof}
By definition, $\overline{S}^{\mathrm{ord}}$ is the complement of
the Zariski closure of $S^{\mathrm{no}},$ which is the vanishing
locus of $h$ in $S^{\mathrm{ord}}.$ It is therefore clear that $h$
vanishes on its complement, and to prove the corollary it is enough
to check that $h$ does not vanish on any irreducible component of
$\partial S=\overline{S}\smallsetminus S$. But by \cite{Lan2} any such irreducible
component contains a rank $m$ cusp, so the claim follows from the
previous Proposition. 
\end{proof}

\subsection{Analytic continuation of $\Theta$ to the boundary and its effect
on Fourier-Jacobi expansions}

\subsubsection{The partial toroidal compactification of the Igusa scheme}

Fix $s\ge1$ and work over $W_{s}$ as a base ring. Since the semi-abelian
scheme $\mathcal{G}$ over $\overline{S}_{s}^{\mathrm{ord}}$ is $\mu$-ordinary,
the \emph{relative} moduli problem defining the big Igusa scheme of
level $p^{t}$ makes sense over $\overline{S}_{s}^{\mathrm{ord}}$.
More precisely, for an $R$-valued point of $\overline{S}_{s}^{\mathrm{ord}}$
denote by $\mathcal{G}_{R}$ the pull-back of $\mathcal{G}$ to $\mathrm{Spec}(R).$
Then $\mathcal{G}_{R}[p^{\infty}]^{0}$ is still isomorphic, locally
in the pro-$\S$tale topology on $\mathrm{Spec}(R),$ to an extension
of $\mathfrak{G}{}^{n-m}$ by $(\mathfrak{d}_{E}^{-1}\otimes\mu_{p^{\infty}})^{m}.$
The relative moduli problem $\overline{T}_{t,s}$ classifies Igusa
structures $(\epsilon^{1},\epsilon^{2})$ on $\mathcal{G}_{R}$ as
in (\ref{eq:Igusa structures}). The compatibility with Weil pairings
is imposed on $\epsilon^{1}$ only, as there is no $\epsilon^{0}$
to pair with $\epsilon^{2}.$ This makes sense even if $\mathcal{G}_{R}$
is not an abelian scheme, while when it is, $\epsilon^{0}$ is determined
by $\epsilon^{2}$. We call the resulting scheme $\overline{T}_{t,s}.$
The following proposition is then obvious.
\begin{prop}
(i) The partially compactified Igusa scheme $\overline{T}_{t,s}$
is a finite {\'e}tale Galois cover of $\overline{S}_{s}^{\mathrm{ord}}$
with Galois group $\Delta_{t}.$

(ii) If $t\ge s$ then the basic vector bundles $\mathcal{P}_{0},\mathcal{P}_{\mu}$
and $\mathcal{Q}$ are canonically trivialized over $\overline{T}_{t,s}$.
\end{prop}

We continue to denote by $\tau:\overline{T}_{t,s}\to\overline{S}_{s}^{\mathrm{ord}}$
the covering map and by $\varepsilon^{1},\varepsilon^{2}$ the resulting
trivializations over $\overline{T}_{t,s}$. The definition of $\widetilde{\Theta}$
over $\overline{S}_{s}^{\mathrm{ord}}$ is then precisely the same
as over the open ordinary stratum $S_{s}^{\mathrm{ord}}$, \emph{cf.
}(\ref{eq:pre_theta}).

\subsubsection{The extended $\Theta$ operator}

To extend the definition of $\Theta$ we need to recall how the Kodaira-Spencer
isomorphism extends to the toroidal compactification. The answer is
given by \cite{Lan}, Theorem 6.4.1.1, part 4. See also \cite{F-C},
Ch. III, Corollary~9.8. In our case (see \cite{Lan}, Definition 6.3.1)
it translates into the following.
\begin{prop}
The Kodaira Spencer isomorphism extends to an isomorphism
\[
\mathrm{KS}:\mathcal{P}\otimes\mathcal{Q}\simeq\Omega_{\overline{S}/W}^{1}[d\log\infty]
\]
over $\overline{S}$.
\end{prop}

The inverse isomorphism $\mathrm{KS}^{-1}$ therefore maps $\Omega_{\overline{S}/W}^{1}$
to sections of $\mathcal{P}\otimes \mathcal{Q}$ \emph{vanishing along the boundary
}$\partial S$. We deduce the following.
\begin{prop}
The formula
\[
\Theta=(1\otimes pr_{\mu}\otimes1)\circ(1\otimes\mathrm{KS^{-1})}\circ\widetilde{\Theta}:\mathcal{E}_{\rho}\to\mathcal{E}_{\rho}\otimes\mathcal{P_{\mu}}\otimes\mathcal{Q}
\]
defines an extension of $\Theta$ over $\overline{S}^{\mathrm{ord}}$.
For any section $f$ of $\mathcal{E}_{\rho}$, $\Theta(f)$ vanishes
along $\partial\overline{S}^{\mathrm{ord}}.$
\end{prop}

\subsubsection{The isomorphism between $\mathcal{P}_{\mu}\otimes\mathcal{Q}$ and
$\check{H}\otimes\mathcal{O}_{\overline{S}}$ when $r=m$}

We now turn to determining the effect of $\Theta$ on Fourier-Jacobi
expansions. This will be done at \emph{maximally degenerate} cusps
only. We therefore take $r=m$ and denote by $\xi\in S^{*}$ a cusp
of rank $m$. Note that there are only finitely many such cusps. Nevertheless,
there are sufficiently many of them to lie in every irreducible component
of $\overline{S}$ \cite{Lan2}. This will allow us to apply the $q$-expansion
principle with these cusps only, not having to worry about expansions
at lower rank cusps, where the formulae are not as nice.
\begin{lem}
\label{lem:Identification}Let $x\in\overline{S}$ be any point lying
above $\xi$. Let $g$ be a representative of the cusp label $[g]_{m}$
to which $\xi$ belongs, $H=H_{g}$ the rank-$m^{2}$ lattice of hermitian
bilinear forms on $Y=Y_{g}$ as in $\S$\ref{subsec:The-local-charts},
and $\check{H}$ its $\mathbb{Z}$-dual. Then there is a canonical
identification of the completed stalk $(\mathcal{P}_{\mu}\otimes\mathcal{Q})_{x}^{\wedge}$
with $\check{H}\otimes\mathcal{\widehat{O}}_{\overline{S},x}$,
\begin{equation}
(\mathcal{P}_{\mu}\otimes\mathcal{Q})_{x}^{\wedge}\simeq\check{H}\otimes\mathcal{\widehat{O}}_{\overline{S},x}.\label{eq:P_m_tensorQ}
\end{equation}
This identification is compatible with the natural action of $\Gamma$
on both sides.
\end{lem}

\begin{proof}
Let $R=\mathcal{\widehat{O}}_{\overline{S},x}$. It is enough to deal
with the standard cusp. When $r=m$ the stalks of the vector bundles
$\mathcal{P}_{\mu}$ and $\mathcal{Q}$ are the $\Sigma$ and $\overline{\Sigma}$-parts
of $\omega_{T_{X}}$, the cotangent bundle of the toric part of $\mathcal{G}$.
Since $T_{X}=\mathrm{Hom}_{\mathcal{O}_{E}}(X,\mathfrak{d}_{E}^{-1}\otimes\mathbb{G}_{m}),$
it follows that there are canonical identifications
\[
\mathcal{P}_{\mu,x}^{\wedge}=X\otimes_{\mathcal{O}_{E},\Sigma}R,\,\,\mathcal{Q}_{x}^{\wedge}=X\otimes_{\mathcal{O}_{E},\overline{\Sigma}}R.
\]
The map $Y\otimes Y\to\check{H}$ described in the course of the construction
of the torsor $\Xi$ in $\S$\ref{subsec:The-local-charts} yields an
isomorphism
\[
(Y\otimes_{\mathcal{O}_{E},\Sigma}R)\otimes_{R}(Y\otimes_{\mathcal{O}_{E},\overline{\Sigma}}R)\simeq\check{H}\otimes R=\mathrm{Hom}(H,R).
\]
Explicitly, $(y\otimes1)\otimes(y'\otimes1)$ goes to the map sending
$h\in H$ to $((h-1)y',y)$. Using the isomorphism $\phi_{X}:Y\simeq X$
we get the isomorphism (\ref{eq:P_m_tensorQ}).

Let us verify that the isomorphism given in the lemma is compatible
with the natural actions of our group $\Gamma$ on $\mathcal{P}_{\mu}\otimes\mathcal{Q}$
and $\check{H}$. At the end of $\S$\ref{subsec:Fourier-Jacobi-expansions}
we computed the action of $\gamma\in\Gamma$ on $(\mathcal{P}_{\mu}\otimes\mathcal{Q})_{x}^{\wedge}$
to be through $^{t}\gamma^{-1}\times\,{}^{t}\overline{\gamma}^{-1}\in GL_{m}\times GL_{m}.$
On the other hand, $\gamma$ acts on $h\in H$ via $h\mapsto\gamma h\,^{t}\overline{\gamma}$.
As $\check{H}$ is the $\mathbb{Z}$-dual of $H$, these actions match
each other.
\end{proof}

\subsubsection{The main theorem}
\begin{thm}
\label{thm:theta on q expansions}Let $\xi$ be a rank-$m$ cusp.
Let $f$ be a section of $\mathcal{E}_{\rho}$ and
\[
f=\sum_{\check{h}\in\check{H}^{+}}a(\check{h})\cdot q^{\check{h}}
\]
its Fourier-Jacobi expansion at $\xi$, as in Proposition \ref{prop:FJ_expansions}.
Then the section $\Theta(f)$ of $\mathcal{E}_{\rho}\otimes\mathcal{P}_{\mu}\otimes\mathcal{Q}$
has the Fourier-Jacobi expansion
\[
\Theta(f)=\sum_{\check{h}\in\check{H}^{+}}a(\check{h})\otimes\check{h}\cdot q^{\check{h}},
\]
using the identification from Lemma \ref{lem:Identification}.
\end{thm}

\begin{proof}
We may work over $W_{s}$ as a base ring. Fix any $t\ge s$ and let
$\overline{T}=\overline{T}_{t,s}.$ We assume that $\xi$ is the standard
cusp of rank $m$ (regarded as a $k$-valued point of $S^{*},$ where
$k$ is algebraically closed and contains $\kappa$), and fix a geometric
point $x$ in the toroidal compactification lying above it. Fix a
local chart $\mathfrak{Z}([g]_{m},\sigma)$ containing $x$ (where
$[g]_{m}=[1]_{m}$ by our assumptions) and let $\mathfrak{Z}([g]_{m},\sigma)_{\xi}$
be the pre-image of $\xi$ in it. As the abelian and toric parts of
$\mathcal{G}$ are constant over $\mathfrak{Z}([g]_{m},\sigma)_{\xi}$
we may fix admissible trivializations $\epsilon^{1}$ and $\epsilon^{2}$
of the graded pieces $gr^{1}$ and $gr^{2}$ of $\mathcal{G}[p^{t}]^{0}$,
over the complete local ring at $x$. Indeed, the point $\xi$ on
the $0$-dimensional Shimura variety $S_{\boldsymbol{G}_{m},K_{m,g}}$
corresponds to an $n-m$ dimensional abelian variety $\mathcal{A}_{m}$
over the algebraically closed field $k,$ with associated PEL structure
of signature $(n-m,0)$. Fix a symplectic trivialization
\[
\epsilon^{1}:\mathfrak{G}[p^{t}]^{n-m}\simeq\mathcal{A}_{m}[p^{t}]=gr^{1}.
\]
Similarly, using the standard basis of $\Lambda\cap V_{m}$ we get
a standard basis on $X$, which gives us a trivialization
\[
\epsilon^{2}:(\mathfrak{d}_{E}^{-1}\otimes\mu_{p^{t}})^{m}\simeq T_{X}[p^{t}]=gr^{2}.
\]
As usual, since $t\ge s$, these trivializations induce trivializations
of $\mathcal{P}_{0},\mathcal{P}_{\mu}$ and $\mathcal{Q}$, hence
of $\mathcal{E}_{\rho}$. They also determine a choice of a point
$\overline{x}$ in $\overline{T}$ above $x$. (If $\sigma$ is replaced
by a $\Gamma$-equivalent cone $\gamma(\sigma)$ the trivialization
$\epsilon^{2}$ is twisted by the action of $\gamma$ on $X$, and
this results in a different $\overline{x}$. The choice of $\epsilon^{1}$
was also arbitrary, and effects the point $\overline{x}$ in a similar
way.) 

We use $R=\mathcal{\widehat{O}}_{\overline{T},\overline{x}}\simeq\mathcal{\widehat{O}}_{\overline{S},x}$
as the ring in which we compute $\Theta$. Recall that the Fourier-Jacobi
coefficient $a(\check{h})$ is a section of the vector bundle $\mathcal{L}(\check{h})\otimes\mathcal{E}_{\rho}$
over the abelian scheme $C$ of relative dimension $m(n-m)$  and that $\mathcal{E}_{\rho}$
has already been trivialized by our choices. Trivializing also the
pull-back of the line bundles $\mathcal{L}(\check{h})$ to $\mathrm{Spec(}R)$,
we may write the ring $R$ as
\[
R=W_{s}(\overline{\kappa})[[u_{1},\dots,u_{m(n-m)},q_{1},\dots,q_{m^{2}}]],
\]
where the $u_{i}$ are pull-backs of local coordinates on $C$ at
the image of $x$, and we may assume that the $a(\check{h})$ are
(vector-valued) functions of the $u_{i}$. We now have
\[
d(\tau^{*}f)=\sum_{\check{h}\in\check{H}^{+}}da(\check{h})\cdot q^{\check{h}}+\sum_{h\in\check{H}^{+}}a(\check{h})\cdot\frac{dq^{\check{h}}}{q^{\check{h}}}\cdot q^{\check{h}}.
\]
Recall that the image of $dq^{\check{h}}/q^{\check{h}}$ modulo $\Omega_{C/W_{s}}$
is $\omega(\check{h})\in\Omega_{\mathfrak{Z}/C}[d\log\infty].$ To
complete the proof of the theorem we shall show the following two
claims.
\begin{enumerate}
\item For any $\eta\in\Omega_{C/W},$ we have $\eta\in\mathrm{KS}(\mathcal{P}_{0}\otimes\mathcal{Q}).$ 
\item The resulting isomorphism $\mathrm{KS}^{-1}:\Omega_{\mathfrak{Z}/C}[d\log\infty]\simeq\mathcal{P}_{\mu}\otimes\mathcal{Q}\simeq\check{H}\otimes R$
(see Lemma \ref{lem:Identification}) carries $\omega(\check{h})$
to $\check{h}\otimes1$.
\end{enumerate}
Indeed, by (1), when we follow the definition of $\Theta$ and mod
out by $\mathcal{P}_{0}\otimes\mathcal{Q}$, the first sum, containing
the $da(\check{h})$'s, disappears. The second sum provides the desired
formula, by (2).

\medskip{}
\textbf{Proof of (1):} This follows from the discussion of the Kodaira-Spencer
map for semi-abelian schemes in \cite{Lan} $\S$4.6.1. Let $C$ assume
the role of the base-scheme denoted there by $S$, and $\mathcal{G}$
the semi-abelian scheme denoted there by $G^{\natural}.$ Then Lan
constructs a Kodaira-Spencer map for semi-abelian schemes $\mathrm{KS}_{\mathcal{G}/C},$
which in our case is an isomorphism
\[
\mathrm{KS}_{\mathcal{G}/C}:\mathcal{P}_{0}\otimes\mathcal{Q}\simeq\Omega_{C/W_{s}}.
\]
Note that Lan allows the abelian part to deform as well, but in our
case $\mathcal{A}=\mathcal{A}_{m}$ is constant. This implies that
the Kodaira-Spencer map, which is a-priori defined on $\omega_{\mathcal{A}}\otimes\omega_{\mathcal{G}},$
factors through its quotient $\omega_{\mathcal{A}}\otimes\omega_{T}$.
In addition, because of the constraints imposed by the endomorphisms,
we may restrict it to $\omega_{\mathcal{A}}(\Sigma)\otimes\omega_{T}(\overline{\Sigma})=\mathcal{P}_{0}\otimes\mathcal{Q}$
without losing any information. Finally, \cite{Lan} Remark 4.6.2.7
and Theorem 4.6.3.16 imply that the diagram
\[
\begin{array}{ccc}
\mathcal{P}_{0}\otimes\mathcal{Q} & \overset{\mathrm{KS}_{\mathcal{G}/C}}{\longrightarrow} & \Omega_{C/W_{s}}\\
\cap &  & \cap\\
\mathcal{P}\otimes\mathcal{Q} & \overset{\mathrm{KS}}{\longrightarrow} & \Omega_{\overline{S}/W_{s}}[d\log\infty]
\end{array}
\]
 is commutative, and this proves (1).

\medskip{}
\textbf{Proof of (2):} The second claim goes to the root of how $\mathrm{KS}$
is defined on $\overline{S}$. See \cite{Lan} $\S$4.6.2, especially
the discussion on p. 269, preceding Definition 4.6.2.10. Fix a basis
$y_{1},\dots,y_{m}$ of $Y$ and let $\chi_{i}=\phi_{X}(y_{i})$ be
the corresponding basis of $X$. Then as a basis of $\check{H}$ we
may take the elements $\check{h}_{ij}=[y_{i}\otimes y_{j}]$ (\emph{cf.
 }the proof of Lemma \ref{lem:Identification}). The corresponding
element of the stalk of $\mathcal{P}_{\mu}\otimes\mathcal{Q}$ at
$x$ is $\chi_{i}\otimes\chi_{j}.$ The variable $q_{ij}=q^{\check{h}_{ij}}$
is then a generator of the invertible $R$-module denoted in \cite{Lan}
by $I(y_{i},\chi_{j}),$ and the extended Kodaira-Spencer homomorphism
is defined in \cite{Lan}, Definition 4.6.2.12 so that it takes $\chi_{i}\otimes\chi_{j}$
to $d\log(q_{ij})=\omega(\check{h}_{ij}).$ The base schemes $S$
and $S_{1}$ in \cite{Lan} are in our case $\mathrm{Spec}(R)$ and
its generic point.
\end{proof}
\begin{cor}
Let $f\in H^{0}(S_{s}^{\mathrm{ord}},\mathcal{E}_{\rho})$ and let
$h$ be the Hasse invariant (\ref{eq:hasse_inv}). Then $\Theta(hf)=h\Theta(f).$
\end{cor}

\begin{proof}
Obvious.
\end{proof}
\begin{cor}
(i) Let $f\in H^{0}(\mathcal{S}^{\mathrm{ord}},\mathcal{E}_{\rho}).$
Then $\Theta(f)=0$ if and only if the FJ expansion of $f$ at every
rank $m$ cusp is constant.

(ii) $f\in\ker(\Theta)$ if and only if its Fourier-Jacobi
expansion at every rank $m$ cusp is supported on $\check{h}\in p\check{H}^{+}.$ 
\end{cor}

\begin{proof}
(i) This follows from our theorem and the FJ-\emph{expansion principle}:
a $p$-adic modular form vanishes if and only if its FJ expansion
at every rank $m$ cusp vanishes. This principle was proved in \cite{Lan},
Proposition 7.1.2.14, under the assumption that every irreducible
component of $\overline{S}$ contains at least one rank $m$ cuspidal
stratum $Z([g]_{m},\sigma).$ This assumption was later verified,
for our Shimura variety among others, in Corollary A.2.3 of \cite{Lan2}.
(ii) Follows by the same argument, noticing that for $a(\check{h})\otimes\check{h}$
to vanish, it is necessary and sufficient that either $a(\check{h})=0$
or $\check{h}\in p\check{H}^{+}.$
\end{proof}

\section{\label{sec:Analytic-continuation-of}Analytic continuation of $\Theta$
to the non-ordinary locus}

\subsection{The almost ordinary locus}

\subsubsection{The stratum $S^{\mathrm{ao}}$}

In this section we assume that $n>m,$ as the question we are about
to discuss requires different considerations when $n=m,$ which will
be handled separately. Let $S$ denote, as in the beginning, the special
fiber of the Shimura variety $\mathcal{S}$. Thus $S$ is $nm$-dimensional,
smooth over $\kappa=\mathbb{F}_{p^{2}}$, and is stratified by the
Ekedahl-Oort strata \cite{Oo}, \cite{Mo2}, \cite{V-W}. The ($\mu$-)ordinary
stratum $S^{\mathrm{ord}}$ is open and dense, and the operator $\Theta$
acts on sections of the automorphic vector bundle $\mathcal{E}_{\rho}$
over it, sending them to sections of $\mathcal{E}_{\rho}\otimes\mathcal{P}_{\mu}\otimes\mathcal{Q}\simeq\mathcal{E}_{\rho}\otimes\mathcal{Q}^{(p)}\otimes\mathcal{Q},$
\[
\Theta:H^{0}(S^{\mathrm{ord}},\mathcal{E}_{\rho})\to H^{0}(S^{\mathrm{ord}},\mathcal{E}_{\rho}\otimes\mathcal{Q}^{(p)}\otimes\mathcal{Q}).
\]
Here we have used the fact that in characteristic $p$ the vector
bundle homomorphism $V_{\mathcal{P}}:\mathcal{P}\to\mathcal{Q}^{(p)}$
is surjective with kernel $\mathcal{P}_{0}$, so induces an isomorphism
$\mathcal{P}_{\mu}=\mathcal{P}/\mathcal{P}_{0}\simeq\mathcal{Q}^{(p)}.$
Our goal in this section is to study the analytic continuation of
$\Theta$ to all of $S$. This is reminiscent of the fact that the
theta operator on $GL_{2}$ (denoted by $A\theta$ in \cite{Ka2})
extends holomorphically across the supersingular points of the modular
curve.
\begin{prop}
There exists a unique EO stratum $S^{\mathrm{ao}}$ of dimension $nm-1.$
The homomorphism $V_{\mathcal{P}}$ is still surjective in every geometric
fiber over $S^{\mathrm{ao}}$, so $\mathcal{P}_{0}=\mathcal{P}[V_{\mathcal{P}}]$
extends to a rank $n-m$ vector bundle over $S^{\mathrm{ord}}\cup S^{\mathrm{ao}}.$
The same applies to $\mathcal{P}_{\mu}$ and of course to $\mathcal{Q}$,
hence every $p$-adic automorphic vector bundle $\mathcal{E}_{\rho}$
extends canonically to the open set $S^{\mathrm{ord}}\cup S^{\mathrm{ao}}$.
\end{prop}

We call $S^{\mathrm{ao}}$ the \emph{almost-ordinary }locus. It is
the divisor of the Hasse invariant $h$ on $S^{\mathrm{ord}}\cup S^{\mathrm{ao}}$,
and, like any other EO stratum, is non-singular.
\begin{proof}
The uniqueness of the EO stratum in codimension 1 is proved in \cite{Woo},
Corollary 3.4.5, where it is deduced from the classification of the
EO strata by Weyl group elements and the calculation of their dimensions
in \cite{Mo2}. The assertion on $V_{\mathcal{P}}$ being surjective
in every geometric fiber follows from the computation of the Dieudonn{\'e}
space at geometric points of $S^{\mathrm{no}}$ (\cite{Woo}, Proposition
3.5.6), reviewed below. Since the base scheme is a non-singular variety,
constancy of the fibral rank of $V_{\mathcal{P}}$ suffices to conclude
that $\mathcal{P}_{0}$ and $\mathcal{P}_{\mu}$ are locally free
sheaves. Finally, $\mathcal{E}_{\rho}$ is constructed by twisting
the representation $\rho$ of $GL_{m}\times GL_{m}\times GL_{n-m}$
(with values in $\kappa$) by the vector bundles $\mathcal{Q},\mathcal{P}_{\mu}$
and $\mathcal{P}_{0}$ as in $\S$\ref{subsec:Twisting}.
\end{proof}

\subsubsection{Dieudonn{\'e} spaces}

Let $k$ be a perfect field of characteristic $p.$ For the following
see \cite{Oda}, \cite{B-W} and \cite{We2} (5.3). A polarized Dieudonn{\'e}
space over $k$ is a finite dimensional $k$-vector space $D$ equipped
with a non-degenerate skew-symmetric pairing $\left\langle ,\right\rangle $
and two linear maps $F:D^{(p)}\to D$ and $V:D\to D^{(p)}$ such that
$\mathrm{Im}(F)=\ker(V)$ and $\mathrm{Im}(V)=\ker(F),$ and such
that $\left\langle Fx,y\right\rangle =\left\langle x,Vy\right\rangle ^{(p)}$
for every $x\in D^{(p)}$ and $y\in D.$ It follows immediately from
the definition that $\dim D=2g$ and $F$ and $V$ have rank $g.$
If $M$ is a principally polarized Dieudonn{\'e} module over $W(k)$ then
$D=M/pM$ is a polarized Dieudonn{\'e} space. If $A$ is a principally
polarized abelian variety over $k$ then its de Rham cohomology $D=H_{dR}^{1}(A/k)$
is equipped with a canonical structure of a Dieudonn{\'e} space, which
may also be identified with the (contravariant) Dieudonn{\'e} module of
$A[p].$ The Hodge filtration is then related to $F$ via
\[
\omega=H^{0}(A,\Omega^{1})=(D^{(p)}[F])^{(p^{-1})}.
\]
It is essential for this that we work over a perfect base.

A polarized $\mathcal{O}_{E}$-Dieudonn{\'e} space is a polarized Dieudonn{\'e}
space admitting, in addition, endomorphisms by $\mathcal{O}_{E}$,
for which $F$ and $V$ are $\mathcal{O}_{E}$-linear and $\left\langle ax,y\right\rangle =\left\langle x,\overline{a}y\right\rangle $
($a\in\mathcal{O}_{E}$). Assume that $k$ contains $\kappa.$ Then
$D(\Sigma)$ and $D(\overline{\Sigma})$ are set in duality by the
pairing, hence are each of dimension $g,$ $V$ maps $D(\Sigma)$
to $D(\overline{\Sigma})^{(p)}$ and $D(\overline{\Sigma})$ to $D(\Sigma)^{(p)}$
and a similar statement, going backwards, holds for $F$. The type
$(n,m)$ of $\omega$ ($n=\dim\omega(\Sigma),$ $m=\dim\omega(\overline{\Sigma})$)
is called the \emph{type}, or \emph{signature}, of $D.$

Over a non-perfect base $\mathrm{Spec}(R)$ (in characteristic $p$,
say, as this is all that we need) one can still associate to a principally
polarized abelian scheme $A/R,$ or to its $p$-divisible group, a
Dieudonn{\'e} \emph{crystal} as in \cite{Gro}, and when \emph{evaluated}
at $\mathrm{(Spec}(R)\subset\mathrm{Spec}(R))$ it yields a polarized
$R$-module $D(A/R)$ with an $F$ and a $V$ as before, which may
be identified with $H_{dR}^{1}(A/R).$ If $R$ is an equi-characteristic
PD-thickening of $k$ then in fact $D(A/R)=R\otimes_{k}D(A_{k}/k)$
with the polarization, $F$ and $V$ extended $R$-linearly. The Hodge
filtration can not be read from $D(A/R)$ any more. In fact, Grothendieck's
theorem asserts that giving $(D(A/R),\omega)$ is tantamount to giving
the deformation of $A$ from $\mathrm{Spec}(k)$ to $\mathrm{Spec}(R).$
We shall apply these remarks later on when $k$ is algebraically closed,
$x\in S(k)$ is a geometric point, and $R=\mathcal{O}_{S,x}/\mathfrak{m}_{S,x}^{2}$
is its first infinitesimal neighborhood.

\medskip{}

Let $k$ be an algebraically closed field containing $\kappa$. Consider
the following polarized $\mathcal{O}_{E}$-Dieudonn{\'e} spaces. We use
the convention that $\mathcal{O}_{E}$ acts on the $e_{i}$ via $\Sigma$
and on the $f_{i}$ via $\overline{\Sigma}$. Write $\mathfrak{G}_\Sigma$ for the
$p$-divisible group $\mathfrak{G}_k$ equipped with the $\mathcal{O}_E$-action
inducing $\Sigma$ on the tangent space, and likewise $\mathfrak{G}_{\overline{\Sigma}}$.

\medskip{}

$(i)$ $D(\mathfrak{G}_{\Sigma}[p])=\mathrm{Span}_{k}\{e_{1},f_{1}\},$
$\left\langle e_{1},f_{1}\right\rangle =1,$ $Ff_{1}^{(p)}=e_{1},$
$Fe_{1}^{(p)}=0,$ $Vf_{1}=e_{1}^{(p)},$ $Ve_{1}=0$. Here $\omega=ke_{1}$
and the signature is $(1,0).$

\medskip{}

$(i)'$ $D(\mathfrak{G}_{\overline{\Sigma}}[p])=\mathrm{Span}_{k}\{e_{2},f_{2}\},$
$\left\langle f_{2},e_{2}\right\rangle =1,$ $Fe_{2}^{(p)}=f_{2},$
$Ff_{2}^{(p)}=0,$ $Ve_{2}=f_{2}^{(p)},$ $Vf_{2}=0.$ Note that $D(\mathfrak{G}_{\overline{\Sigma}}[p])=D(\mathfrak{G}_{\Sigma}[p])^{(p)},$
$\omega=kf_{2}$ and the signature is $(0,1).$

\medskip{}

$(ii)$ $AO(2,1)=\mathrm{Span}_{k}\{e_{i},f_{i}|\,1\le i\le3\},$
$\left\langle e_{1},f_{3}\right\rangle =\left\langle f_{2},e_{2}\right\rangle =\left\langle f_{1},e_{3}\right\rangle =1$
(and $\left\langle e_{i},f_{j}\right\rangle =0$ if $i+j\ne4$); $F$
and $V$ are given by the following table, where to ease notation
the $^{(p)}$ is left out.

\begin{center}%
\begin{tabular}{|c|c|c|c|c|c|c|}
\hline 
 & $e_{1}$ & $e_{2}$ & $e_{3}$ & $f_{1}$ & $f_{2}$ & $f_{3}$\tabularnewline
\hline 
\hline 
$F$ & $0$ & $f_{1}$ & $0$ & $e_{1}$ & $0$ & $e_{2}$\tabularnewline
\hline 
$V$ & $0$ & $0$ & $f_{2}$ & $0$ & $e_{1}$ & $e_{3}$\tabularnewline
\hline 
\end{tabular}\end{center}

\medskip{}

This is the Dieudonn{\'e} space denoted by $\overline{B}(3)$ in \cite{B-W}.
Here $\omega=\mathrm{Span}_{k}\{e_{1},e_{3},f_{2}\}$ and $\mathcal{P}_{0}=\omega(\Sigma)[V]=ke_{1}.$

\medskip{}

$(iii)$ $AO(3,1)=\mathrm{Span}_{k}\{e_{i},f_{i}|\,1\le i\le4\},$
$\left\langle e_{1},f_{4}\right\rangle =\left\langle e_{2},f_{3}\right\rangle =\left\langle f_{2},e_{3}\right\rangle =\left\langle f_{1},e_{4}\right\rangle =1$
(and $\left\langle e_{i},f_{j}\right\rangle =0$ if $i+j\ne5$); $F$
and $V$ are given by the following table, where to ease notation
the $^{(p)}$ is left out.

\begin{center}%
\begin{tabular}{|c|c|c|c|c|c|c|c|c|}
\hline 
 & $e_{1}$ & $e_{2}$ & $e_{3}$ & $e_{4}$ & $f_{1}$ & $f_{2}$ & $f_{3}$ & $f_{4}$\tabularnewline
\hline 
\hline 
$F$ & $0$ & $f_{1}$ & $0$ & $0$ & $e_{1}$ & $e_{2}$ & $0$ & $e_{3}$\tabularnewline
\hline 
$V$ & $0$ & $0$ & $0$ & $f_{3}$ & $0$ & $e_{1}$ & $e_{3}$ & $e_{4}$\tabularnewline
\hline 
\end{tabular}\end{center}

\medskip{}

This is the Dieudonn{\'e} space denoted by $\overline{B}(4)$ in \cite{B-W}.
Here $\omega$ is equal to $\mathrm{Span}_{k}\{e_{1},e_{3},e_{4},f_{3}\}$ and $\mathcal{P}_{0}=\omega(\Sigma)[V]=\mathrm{Span}_{k}\{e_{1},e_{3}\}.$
\begin{prop}
\label{prop:DM_at_ao_points}(\cite{Woo}, Proposition 3.5.6) Let
$x\in S^{\mathrm{ao}}(k)$ be an almost-ordinary geometric point.
Then $D_{x}=D(\mathcal{A}_{x}/k)$ is isomorphic to the following.

(i) $n=m+1:$
\[
D=AO(2,1)\oplus\bigoplus_{i=1}^{m-1}\mathrm{Span}_{k}\{e_{i}^{\mu},e_{i}^{et},f_{i}^{\mu},f_{i}^{et}\}
\]
where $\left\langle e_{i}^{\mu},f_{i}^{et}\right\rangle =\left\langle f_{i}^{\mu},e_{i}^{et}\right\rangle =1$
(and $\left\langle e_{i}^{\mu},f_{i}^{\mu}\right\rangle =\left\langle e_{i}^{et},f_{i}^{et}\right\rangle =0$),
and $F$ and $V$ are given by the following table

\begin{center}%
\begin{tabular}{|c|c|c|c|c|}
\hline 
 & $e_{i}^{\mu}$ & $e_{i}^{et}$ & $f_{i}^{\mu}$ & $f_{i}^{et}$\tabularnewline
\hline 
\hline 
$F$ & $0$ & $f_{i}^{et}$ & $0$ & $e_{i}^{et}$\tabularnewline
\hline 
$V$ & $f_{i}^{\mu}$ & $0$ & $e_{i}^{\mu}$ & $0$\tabularnewline
\hline 
\end{tabular}\end{center}

(ii) $n\ge m+2:$
\[
D=AO(3,1)\oplus\bigoplus_{i=1}^{m-1}\mathrm{Span}_{k}\{e_{i}^{\mu},e_{i}^{et},f_{i}^{\mu},f_{i}^{et}\}\oplus D(\mathfrak{G}_{\Sigma}[p])^{n-m-2}.
\]
\end{prop}

\subsubsection{The Kodaira-Spencer isomorphism along the almost ordinary stratum}

The following result is the key to the analytic continuation of the
theta operator, which will be proved in the next section.
\begin{thm}
\label{thm: =00005Cpsi(du)_has_a_zero}Let
\[
\psi=(pr_{\mu}\otimes1)\circ\mathrm{KS}^{-1}:\Omega_{S/\kappa}\to\mathcal{P}_{\mu}\otimes\mathcal{Q}
\]
be the composition of the inverse of the Kodaira-Spencer isomorphism
and the projection from $\mathcal{P}$ to $\mathcal{P}_{\mu}=\mathcal{P}/\mathcal{P}_{0}$
(well-defined over $S^{\mathrm{ord}}\cup S^{\mathrm{ao}}$). Let $u=0$
be a local equation of the divisor $S^{\mathrm{ao}}$ in a Zariski
open set $U\subset S^{\mathrm{ord}}\cup S^{\mathrm{ao}}.$ Then $\psi(du)$
vanishes along $S^{\mathrm{ao}}\cap U.$
\end{thm}

\begin{rem*}
Compare with \cite{dS-G1} Proposition 3.11. In terms of the foliation
$\mathcal{T}S^{+}$ introduced in \cite{dS-G2} the theorem asserts
that at any point $x\in S^{\mathrm{ao}}$ this foliation is tangential
to $S\mathrm{^{ao}},$ i.e $\mathcal{T}S^{+}|_{x}\subset\mathcal{T}S^{\mathrm{ao}}|_{x}.$
In \cite{dS-G2} we studied a certain open subset $S_{\sharp}\subset S$
which was a union of Ekedahl-Oort strata, including $S^{\mathrm{ord}},$
$S^{\mathrm{ao}}$ and a unique minimal EO stratum denoted there $S\mathrm{^{fol}}$,
of dimension $m^{2}.$ The subset $S_{\sharp}$ and the foliation
$TS^{+}$ are related to the geometry of auxiliary Shimura varieties
of parahoric level structure at $p,$ and seem to play an important
role. In \emph{loc. cit.} Theorem 25, it was proved that $\mathcal{T}S^{+}$
is tangential to $S^{\mathrm{fol}}.$ In view of these two results,
claiming tangentiality to $S^{\mathrm{ao}}$ and $S\mathrm{^{fol}}$,
it is reasonable to expect that $\mathcal{T}S^{+}$ is tangential
to \emph{every} EO strata in $S_{\sharp}.$ The proofs of the known
cases, whether in \emph{loc. cit.} or here, invoke delicate computations
with Dieudonn{\'e} modules, and at present we see no conceptual reason
justifying our expectation, which could avoid such computations.
\end{rem*}
\begin{proof}
Let $k$ be an algebraically closed field containing $\kappa,$ $x\in S^{\mathrm{ao}}(k)$
a geometric point and $D_{x}=D(\mathcal{A}_{x}/k).$ Let $R=\mathcal{O}_{S,x}/\mathfrak{m}_{S,x}^{2}$
and $d:R\to\Omega_{S/k}|_{x}=\mathfrak{m}_{S,x}/\mathfrak{m}_{S,x}^{2}$
the canonical derivation $df=(f-f(x))\mod\mathfrak{m}_{S,x}^{2}.$
Let $D=H_{dR}^{1}(\mathcal{A}/R).$ The Gauss-Manin connection on
$H_{dR}^{1}(\mathcal{A}/S)$ induces a map
\[
\nabla:D\to D_{x}\otimes_{k}\Omega_{S/k}|_{x}
\]
satisfying $\nabla(r\alpha)=r(x)\nabla(\alpha)+\alpha\otimes dr,$
which by abuse of language we call the Gauss-Manin connection on $D$.
It is easy to see that every $\alpha\in D_{x}$ has a unique extension
to a horizontal section $\alpha\in D$, i.e. a section satisfying
$\nabla(\alpha)=0$. Thus, we may identify $D$ with $R\otimes_{k}D_{x},$
the horizontal sections being $D_{x}$. Since the Gauss-Manin connection
commutes with isogenies, $V:D\to D^{(p)}$ and $F:D^{(p)}\to D$ map
horizontal sections to horizontal sections. For the same reason, if
$x,y$ are horizontal sections of $D$, their pairing $\left\langle x,y\right\rangle $
is horizontal for $d,$ i.e. lies in $k$.

\medskip{}

We now distinguish between two cases. 

I. Assume that $n=m+1$. Then
\[
D_{x}=\mathrm{Span}_{k}\{\underline{e_{1}},e_{2},\underline{e_{3}},f_{1},\underline{f_{2}},f_{3},\underline{e_{i}^{\mu}},e_{i}^{et},\underline{f_{i}^{\mu}},f_{i}^{et}\}_{1\le i\le m-1}
\]
where the first six vectors span $AO(2,1),$ as in Proposition \ref{prop:DM_at_ao_points}(i).
For the convenience of the reader we have underlined the vectors spanning
$\omega_{x}$. The module~$D$ is spanned by the same vectors over
$R$, and the pairings and the tables giving $F$ and~$V$ remain
the same over $R$. 

We now write the most general deformation of $\omega_{x}$ to a projective
submodule of~$D$ which is invariant under the endomorphisms and isotropic.
An easy check yields that it is given by
\[
\omega=\mathrm{Span}_{R}\{\widetilde{e}_{1},\widetilde{e}_{3},\widetilde{f}_{2},\widetilde{e}_{i}^{\mu},\widetilde{f}_{i}^{\mu}\}_{1\le i\le m-1},
\]
where
\begin{itemize}
\item $\widetilde{e}_{1}=e_{1}+ue_{2}+\sum_{i=1}^{m-1}u_{i}e_{i}^{et}$\medskip{}
\item $\widetilde{e}_{3}=e_{3}+ve_{2}+\sum_{i=1}^{m-1}v_{i}e_{i}^{et}$\medskip{}
\item $\widetilde{f}_{2}=f_{2}-vf_{1}+uf_{3}+\sum_{i=1}^{m-1}w_{i}f_{i}^{et}$\medskip{}
\item $\widetilde{e}_{i}^{\mu}=e_{i}^{\mu}+w_{i}e_{2}+\sum_{j=1}^{m-1}w_{ij}e_{j}^{et}$\medskip{}
\item $\widetilde{f}_{i}^{\mu}=f_{i}^{\mu}-v_{i}f_{1}+u_{i}f_{3}+\sum_{j=1}^{m-1}w_{ji}f_{j}^{et}$.
\end{itemize}

\vspace{0.1cm}

\noindent
The $mn$ parameters $u,u_{i},v,v_{i},w_{i},w_{ij}$ are, according
to Grothendieck, the local parameters of $R,$ serving as a basis
of $\mathfrak{m}_{R}$ over $k.$ It follows that $\mathcal{P}_{0}$
is indeed of rank 1, as claimed before, spanned over $R$ by $\widetilde{e}_{1}$,
while $\mathcal{Q}$ is spanned over $R$ by the $m$ vectors $\widetilde{f}_{2},\widetilde{f}_{i}^{\mu}.$
Furthermore, computing the Hasse matrix $H=V_{\mathcal{P}}^{(p)}\circ V_{\mathcal{Q}}$
in the bases $(\widetilde{f}_{2},\widetilde{f}_{i}^{\mu})$ and $(\widetilde{f}_{2}^{(p^{2})},\widetilde{f}_{i}^{\mu(p^{2})})$
of $\mathcal{Q}$ and $\mathcal{Q}^{(p^{2})}$ we get
\[
H=\left(\begin{array}{ccccc}
u & u_{1} & u_{2} & \cdots & u_{m-1}\\
0 & 1 & 0 &  & 0\\
0 & 0 & 1 &  & 0\\
\vdots &  &  & \ddots & \vdots\\
0 & \cdots &  & \cdots & 1
\end{array}\right),
\]
so the (trivialized) Hasse invariant $h$ is simply $u$. Since we
know that $S^{\mathrm{ao}}$ is the zero divisor of $h,$ $S^{\mathrm{ao}}\cap\mathrm{Spec}(R)$
is given by the equation $u=0.$ Note, in passing, that this proves
that the zero divisor of the Hasse invariant is reduced and equal
to the non-ordinary locus.

To compute $\mathrm{KS}(\mathcal{P}_{0}\otimes\mathcal{Q})$ recall
how it is defined. From the Gauss-Manin connection we get a homomorphism
of sheaves
\[
\overline{\nabla}:\omega\to\Omega_{S/\kappa}\otimes(H_{dR}^{1}(\mathcal{A}/S)/\omega)
\]
which induces a homomorphism $\mathcal{P}\to\Omega_{S/\kappa}\otimes\mathcal{Q}^{\vee}.$
This induces the map
\[
\mathrm{KS}:\mathcal{P}\otimes\mathcal{Q}\to\Omega_{S/\kappa}
\]
which happens to be an isomorphism. We begin by computing the Gauss-Manin connection on $\mathcal{P}_{0}$:
\[
\nabla(\widetilde{e}_{1})=e_{2}\otimes du+\sum_{i=1}^{m-1}e_{i}^{et}\otimes du_{i}.
\]
Projecting $D_{x}$ to $D_{x}/\omega_{x}=H^{1}(\mathcal{A}_{x},\mathcal{O})$
and noting that $e_{2},e_{i}^{et}$ modulo $\omega_{x}$ are a basis
for the dual $\mathcal{Q^{\vee}}$ of $\mathcal{Q}$, equipped with
the conjugate action of $\mathcal{O}_{E}$ via $\Sigma,$ we find
out that
\[
\mathrm{KS}(\mathcal{P}_{0}\otimes\mathcal{Q})|_{x}=\mathrm{Span}_{k}\{du,du_{i}\}.
\]
From the definition of $\psi$ it follows that $\psi(du)|_{x}=0.$
Now assume that $u$ is a \emph{global} generator of the ideal of
$S^{\mathrm{ao}}$ in a Zariski open set $U$. Then we conclude that
$\psi(du)=0$ along $S^{\mathrm{ao}}\cap U$ as claimed.\medskip{}

II. The proof of the theorem in the case $n-m\ge2$ is similar, using
Proposition \ref{prop:DM_at_ao_points}(ii). Here
\[
D_{x}=\mathrm{Span}_{k}\{\underline{e_{1}},e_{2},\underline{e_{3}},\underline{e_{4}},f_{1},f_{2},\underline{f_{3}},f_{4},\underline{e_{i}^{\mu}},e_{i}^{et},\underline{f_{i}^{\mu}},f_{i}^{et},\underline{e_{j}^{\sharp}},f_{j}^{\sharp}\}_{1\le i\le m-1,\,1\le j\le n-m-2}
\]
where the first eight vectors span $AO(3,1)$ and for every $j$ the
vectors $e_{j}^{\sharp},f_{j}^{\sharp}$ span a copy of $D(\mathfrak{G}_{\Sigma}[p]).$
For convenience we have again underlined the vectors spanning $\omega_{x}$.
The most general deformation of $\omega_{x}$ in $D=R\otimes_{k}D_{x}$
is spanned by the following vectors:
\begin{itemize}
\item $\widetilde{e}_{1}=e_{1}-ue_{2}+\sum_{i=1}^{m-1}u_{i}e_{i}^{et}$\medskip{}
\item $\widetilde{e}_{3}=e_{3}+ve_{2}+\sum_{i=1}^{m-1}v_{i}e_{i}^{et}$\medskip{}
\item $\widetilde{e}_{4}=e_{4}+we_{2}+\sum_{i=1}^{m-1}w_{i}e_{i}^{et}$\medskip{}
\item $\widetilde{f}_{3}=f_{3}+wf_{1}+vf_{2}+uf_{4}-\sum_{l=1}^{m-1}x_{l}f_{l}^{et}-\sum_{k=1}^{n-m-2}y_{k}f_{k}^{\sharp}$\medskip{}
\item $\widetilde{e}_{i}^{\mu}=e_{i}^{\mu}+x_{i}e_{2}+\sum_{l=1}^{m-1}x_{il}e_{l}^{et}$\medskip{}
\item $\widetilde{f}_{i}^{\mu}=f_{i}^{\mu}-w_{i}f_{1}-v_{i}f_{2}+u_{i}f_{4}+\sum_{l=1}^{m-1}x_{li}f_{l}^{et}+\sum_{k=1}^{n-m-2}y_{ki}f_{k}^{\sharp}$\medskip{}
\item $\widetilde{e}_{j}^{\sharp}=e_{j}^{\sharp}+y_{j}e_{2}+\sum_{l=1}^{m-1}y_{jl}e_{l}^{et}$.
\end{itemize}

\vspace{0.1cm}

\noindent
The $nm$ parameters $u,v,w,u_{i},v_{i},w_{i},x_{i},x_{il},y_{k},y_{kj}$
form a basis of $\mathfrak{m}_{R}$ over $k$. Calculating the Hasse
matrix $H$ yields exactly the same $m\times m$ matrix as above,
hence $u=0$ is again the local infinitesimal equation of $S^{\mathrm{ao}}.$
The submodule $\mathcal{P}_{0}$ is $n-m$ dimensional, and is spanned
by $\widetilde{e}_{1},\widetilde{e}_{3}$ and the $\widetilde{e}_{j}^{\sharp}$.
Calculating $\mathrm{KS}$ we find that
\[
\mathrm{KS}(\mathcal{P}_{0}\otimes\mathcal{Q})|_{x}=\mathrm{Span}_{k}\{du,du_{i},dv,dv_{i},dy_{j},dy_{jl}\}
\]
($1\le i,l\le m-1,\,\,1\le j\le n-m-2)$ so as before $\psi(du)|_{x}=0.$
We conclude the proof as in the first case.
\end{proof}

\subsection{Analytic continuation of $\Theta$ ($m<n$)}

\subsubsection{Compactification of a certain intermediate Igusa cover}

Recall the Igusa tower $T_{t,s}$ over $S_{s}^{\mathrm{ord}}$ that
has been constructed in $\S$\ref{subsec:bigIg}. Let
\[
\Delta_{t}^{1}=SL_{m}(\mathcal{O}_{E}/p^{t}\mathcal{O}_{E})\times U_{n-m}(\mathcal{O}_{E}/p^{t}\mathcal{O}_{E})\vartriangleleft\Delta_{t}
\]
and denote by $T_{t,s}^{1}$ the intermediate covering of $S_{s}^{\mathrm{ord}}$
fixed by $\Delta_{t}^{1}.$ It is a Galois {\'e}tale cover of $S_{s}^{\mathrm{ord}}$
with Galois group $(\mathcal{O}_{E}/p^{t}\mathcal{O}_{E})^{\times}.$
In this section let $T=T_{1,1}^{1},$ and let $\tau:T\to S^{\mathrm{ord}}$
be the covering map, whose Galois group is identified with $\kappa^{\times}.$

Let $\mathcal{L}=\det(\mathcal{Q})$ and recall that the Hasse invariant
$h\in H^{0}(S,\mathcal{L}^{p^{2}-1})$ (\ref{eq:hasse_inv}).
\begin{lem}
(i) The line bundle $\mathcal{L}$ is canonically trivialized over
$T,$ i.e. there is a canonical isomorphism $\varepsilon:\mathcal{O}_{T}\simeq\tau^{*}\mathcal{L}.$

(ii) Denoting by $a$ the global section of $\tau^{*}\mathcal{L}$
corresponding to the section ``$1$'' under the trivialization,
we have $a^{p^{2}-1}=\tau^{*}h.$
\end{lem}

\begin{proof}
(i) The canonical trivialization $\varepsilon^{2}(\overline{\Sigma}):\mathcal{O}_{T_{1,1}}^{m}\simeq\tau_{1,1}^{*}\mathcal{Q}$
over the big Igusa variety $T_{1,1}$ induces a canonical trivialization
on the determinants $\varepsilon:\mathcal{O}_{T_{1,1}}\simeq\tau_{1,1}^{*}\mathcal{L}$.
The latter descends to $T$ because it is invariant under $\Delta_{1}^{1}.$

(ii) Since $\Ver$ is the identity on $\mu_{p},$ the trivialization
$\epsilon^{2}$ of $gr^{2}\mathcal{A}[p]$ satisfies
\[
\Ver\circ\Ver^{(p)}\circ(\epsilon^{2})^{(p^{2})}=\epsilon^{2}.
\]
Passing to cohomology (recall $\varepsilon^{2}=((\epsilon^{2})^{-1})^{*}$)
yields the relation $(\varepsilon^{2}(\overline{\Sigma}))^{(p^{2})}=H\circ\varepsilon^{2}(\overline{\Sigma})$
where $H,$ recall, is $V_{\mathcal{P}}^{(p)}\circ V_{\mathcal{Q}}.$
Taking determinants we get
\[
\varepsilon^{(p^{2})}=h\circ\varepsilon
\]
and evaluating at ``$1$'' gives the desired relation.
\end{proof}
The following Kummer-type result was proved in \cite{dS-G1} $\S$2.4.2
for signature (2,1) and the proof easily generalizes. See also \cite{Go}.
Let 
\[
S'=S^{\mathrm{ord}}\cup S^{\mathrm{ao}}.
\]
 Consider the fiber product
\begin{equation}
T'=\mathcal{L}\times_{\mathcal{L}^{p^{2}-1}}S'\label{eq:T'}
\end{equation}
where the two maps to $\mathcal{L}^{p^{2}-1}$ are $\lambda\mapsto\lambda^{p^{2}-1}$
and $h$. Note that the pull-back of $\mathcal{L}$ from $S'$ to
$T'$ admits a tautological $p^{2}-1$ root of $h$ extending $a$,
which we still call $a.$ Then $T'\to S'$ is finite flat of degree
$p^{2}-1,$ is Galois {\'e}tale with Galois group $\kappa^{\times}$ over
$S^{\mathrm{ord}},$ and totally (tamely) ramified along $S^{\mathrm{ao}}.$
It satisfies a universal property with respect to extracting a $p^{2}-1$
root from the section $h;$ \emph{cf. }loc. cit.\emph{ }From part
(ii) of the last proposition it follows that there is a canonical
map
\[
T\to T'.
\]
Since both source and target are $\kappa^{\times}$-torsors over $S^{\mathrm{ord}}$
and the map respects the $\kappa^{\times}$ action, this map is an
isomorphism of $T$ with the pre-image of $S^{\mathrm{ord}}$ in $T'.$
In this way we may identify $T'$ with a (partial) compactification
of $T.$ We then have the following.
\begin{prop}
\label{prop:Igusa formal parameters}(i) The morphism $\tau':T'\to S'$
is finite flat of degree $p^{2}-1$, Galois {\'e}tale with Galois group
$\kappa^{\times}$ over $S^{\mathrm{ord}},$ and totally (tamely)
ramified along $S^{\mathrm{ao}}.$

(ii) $T'$ is everywhere non-singular.

(iii) Let $x\in S^{\mathrm{ao}}(k)$ be a geometric point, and $y\in T'(k)$
the unique geometric point mapping to it. Then there are formal parameters
$u,v_{i}$ ($1\le i\le nm-1$) at $x$ such that $u=0$ is the infinitesimal
equation of $S^{\mathrm{ao}}$, and such that as formal parameters
on $T'$ at $y$ we can take $w,v_{i}$ where $w^{p^{2}-1}=u.$

(iv) $T'$ and $T=T_{1,1}^{1}$ are irreducible.
\end{prop}

\begin{proof}
The proof is the same as in the case of signature $(2,1),$ \emph{cf.
\cite{dS-G1}, }$\S$2.4.3, Proposition 2.16.
\end{proof}

\subsubsection{The main theorem for scalar-valued modular forms}

We can now prove the analytic continuation of $\Theta$ in characteristic
$p,$ when applied to \emph{scalar-valued} $p$-adic modular forms.
Recall that $\mathcal{L}=\det(\mathcal{Q})$.
\begin{thm}
\label{thm:holomorphic extension of theta m less n}Assume that $m<n.$
Consider the operator
\[
\Theta:H^{0}(S^{\mathrm{ord}},\mathcal{L}^{k})\to H^{0}(S^{\mathrm{ord}},\mathcal{L}^{k}\otimes\mathcal{Q}^{(p)}\otimes\mathcal{Q}).
\]
Then $\Theta$ extends holomorphically to an operator
\[
\Theta:H^{0}(S,\mathcal{L}^{k})\to H^{0}(S,\mathcal{L}^{k}\otimes\mathcal{Q}^{(p)}\otimes\mathcal{Q}).
\]
\end{thm}

\begin{rem*}
The analytic continuation of $\Theta$ to global modular forms is
a characteristic $p$ phenomenon and does not seem to extend to $S_{s}$
(i.e. modulo $p^{s}$) for $s>1$. Had it extended for all $s$, we
would have obtained, for any algebraic modular form $f$ of weight
$k$, a well-defined rigid analytic ``$\Theta(f)$'', of weight
$k+p+1$, on the whole rigid analytic space associated to $\mathcal{S}$.
By GAGA (and the K{\"o}cher principle) this $\Theta(f)$ would have been
algebraic. However, the Maass-Shimura operators in characteristic
0 do not preserve the space of classical modular forms.
\end{rem*}
\begin{proof}
Let $f\in H^{0}(S,\mathcal{L}^{k}).$ Then $\Theta(f)$ is a section
of $\mathcal{L}^{k}\otimes\mathcal{Q}^{(p)}\otimes\mathcal{Q}$ over
$S^{\mathrm{ord}}$ and we have to show that it extends holomorphically
to $S$. Since $S$ is non-singular, it is enough to show that it
extends holomorphically to $S'=S^{\mathrm{ord}}\cup S^{\mathrm{ao}}$,
an open set whose complement is of codimension $2$. Indeed, 
Zariski locally we may trivialize the vector bundles, and then any coordinate
of $\Theta(f)$ becomes a meromorphic function, whose polar set, if
non-empty, should have codimension 1.

Let $\tau':T'\to S'$ be the intermediate Igusa variety constructed
above. Over $T$ (the pre-image of $S^{\mathrm{ord}}$) we can write
the trivialization $\varepsilon$ of $\mathcal{L}$ as $f\mapsto f/a^{k}$.
This introduces a pole of order $k$, at most, along $T^{\mathrm{ao}}=\tau'^{-1}(S^{\mathrm{ao}}).$
Let $y\in T^{\mathrm{ao}}$ be a geometric point and $x=\tau'(y).$
Let $u,v_{i}$ be formal parameters at $x$ and $w,v_{i}$ formal
parameters at $y$ as in Proposition \ref{prop:Igusa formal parameters}.
Let
\[
f/a^{k}=\sum_{r=-k}^{\infty}g_{r}(v)w^{r}
\]
be the Taylor expansion of $f/a^{k}$ in $\mathcal{\widehat{O}}_{T',y}$,
where the $g_{r}(v)$ are power series in the $v_{i}$. Note that
$du=d(w^{p^{2}-1})=-w^{p^{2}-2}dw$. Thus,

\begin{equation*}
\begin{split}
d(f/a^{k})&=\sum_{r=-k}^{\infty}w^{r}dg_{r}(v)+\sum_{r=-k}^{\infty}rg_{r}(v)w^{r-1}dw
\\
&=\sum_{r=-k}^{\infty}w^{r}dg_{r}(v)-\sum_{r=-k}^{\infty}rg_{r}(v)w^{r-(p^{2}-1)}du
\\ &
=\sum_{r=-k}^{\infty}w^{r}dg_{r}(v)-\sum_{r=-k}^{\infty}rg_{r}(v)w^{r}u^{-1}du.
\end{split}
\end{equation*}
When we compute
\[
\widetilde{\Theta}(f)=a^{k}\left(\sum_{r=-k}^{\infty}w^{r}dg_{r}(v)-\sum_{r=-k}^{\infty}rg_{r}(v)w^{r}u^{-1}du\right),
\]
which we know descends to $S'$, we see that the first sum becomes
holomorphic ($a$ vanishes along $T^{\mathrm{ao}}$), while the second
sum retains a simple pole along $S^{\mathrm{ao}}.$ However, to get
$\Theta(f)$ we must still apply the vector-bundle homomorphism $\psi.$
Theorem \ref{thm: =00005Cpsi(du)_has_a_zero} says that $\psi(du)$
vanishes along $S^{\mathrm{ao}},$ hence the simple pole disappears
and $\Theta(f)$ is holomorphic at $x$. This being true at every
$x\in S^{\mathrm{ao}}$, we conclude that $\Theta(f)$ is everywhere
holomorphic.
\end{proof}

\subsection{Analytic continuation of $\Theta$ ($m=n$)}

We briefly indicate the modifications in the proof which are necessary
to deal with the case $m=n$. In this case $\mathrm{rk}(\ker(V_{\mathcal{P}}))$
changes when we move from $S^{\mathrm{ord}}$ to $S^{\mathrm{ao}},$
so $\mathcal{P}_{0}$ and $\mathcal{P}_{\mu}$ do not extend, with
the same definitions, to vector bundles over $S'=S^{\mathrm{ord}}\cup S^{\mathrm{ao}}.$
As such, we cannot extend $\Theta$ beyond $S^{\mathrm{ord}}$ using
$pr_{\mu}$. Instead, we apply $(V_{\mathcal{P}}\otimes1)\circ\mathrm{KS}^{-1}$
to $\Omega_{S/\kappa}$, a map that gives the same result as $(pr_{\mu}\otimes1)\circ\mathrm{KS}^{-1}$
over $S^{\mathrm{ord}}$ in characteristic $p,$ but does not make
sense over $S_{s}^{\mathrm{ord}}$ for $s>1$. Let $\mathcal{L}=\det(\mathcal{Q})$
as before.

\subsubsection{Preliminary results on the Igusa variety when $m=n$}

Let $T=T_{1,1}^{1}$ as before. Let $T'$ be defined by (\ref{eq:T'}).
As before, it is a partial compactification of $T$. Since the divisor
of the Hasse invariant is not reduced when $n=m$ (see $\S$\ref{subsec:basic}
and Lemma \ref{lem:h=00003Dh_Q^(p+1)}), the proof of the irreducibility
of $T$ as in \cite{dS-G1}, Proposition 2.16, breaks down.
\begin{prop}
(i) The morphism $T'\to S'$ is finite flat of degree $p^{2}-1,$
with Galois group $\kappa^{\times}$.

(ii) $T'$ is non-singular.

(iii) $T'$ and the Igusa variety $T$ decompose into $p+1$ irreducible
components $T'_{\zeta}$ (resp. $T_{\zeta}$) labeled by $\zeta$
such that $\zeta^{p+1}=1.$ More canonically,
\[
\pi_{0}(T)=\pi_{0}(T')\simeq\kappa^{\times}/\mathbb{F}_{p}^{\times}.
\]

(iv) The map $T'_{\zeta}\to S'$ is totally (tamely) ramified over
$S^{\mathrm{ao}}$ of degree $p-1$.
\end{prop}

\begin{proof}
The proof of (i) is the same as when $n>m.$ Our $T'$ is still obtained
from $S'$ by extracting a $p^{2}-1$ root of $h$. However, this
time $h=h_{\mathcal{Q}}^{p+1}$ where $h_{\mathcal{Q}}$ is in $H^{0}(S',\mathcal{L}^{p-1})$
so $T'=\bigsqcup T'_{\zeta}$ where $T'_{\zeta}=S'[\sqrt[p-1]{\zeta h_{\mathcal{Q}}}].$
As the divisor of $h_{\mathcal{Q}}$ is reduced and equal to $S^{\mathrm{ao}},$
the rest of the proof is similar to the case $n>m.$ 
\end{proof}

\subsubsection{The main theorem when $m=n$}
\begin{thm}
\label{thm:m=00003Dn_holmorphicity}Assume that $m=n$. The operator
\[
\Theta=(1\otimes V_{\mathcal{P}}\otimes1)\circ\mathrm{(1\otimes KS}^{-1})\circ\widetilde{\Theta}:H^{0}(S^{\mathrm{ord}},\mathcal{L}^{k})\to H^{0}(S^{\mathrm{ord}},\mathcal{L}^{k}\otimes\mathcal{Q}^{(p)}\otimes\mathcal{Q})
\]
extends holomorphically to an operator
\[
\Theta:H^{0}(S,\mathcal{L}^{k})\to H^{0}(S,\mathcal{L}^{k}\otimes\mathcal{Q}^{(p)}\otimes\mathcal{Q}).
\]
\end{thm}

\begin{proof}
As before, let
\[
\psi=(V_{\mathcal{P}}\otimes1)\circ\mathrm{KS}^{-1}:\Omega_{S/\kappa}\to\mathcal{Q}^{(p)}\otimes\mathcal{Q}.
\]
Let $x\in S^{\mathrm{ao}}(k)$ be a geometric point. The Dieudonn{\'e}
module $D_{x}$ is now given by
\[
D_{x}=\mathrm{Span}_{k}\{\underline{e}_{i}^{\mu},\underline{f}_{i}^{\mu},e_{i}^{et},f_{i}^{et},\underline{e}^{\sharp},f^{\sharp},e^{\flat},\underline{f}^{\flat}\}_{1\le i\le m-1}
\]
where $\mathrm{Span}_{k}\{e^{\sharp},f^{\sharp}\}$ is isomorphic
to $D(\mathfrak{G}_{\Sigma}[p])$ and $\mathrm{Span}_{k}\{e^{\flat},f^{\flat}\}$
to $D(\mathfrak{G}_{\overline{\Sigma}}[p])$ (see \cite{Woo}, Proposition
3.5.6). The underlined vectors span $\omega_{x}$. As before, we let
$R=\mathcal{O}_{S,x}/\mathfrak{m}_{S,x}^{2}$ and $D=R\otimes_{k}D_{x}$.
The most general deformation of $\omega_{x}$ to $\omega\subset D$
compatible with the endomorphisms and the polarization is spanned
by
\begin{itemize}
\item $\widetilde{e}_{i}^{\mu}=e_{i}^{\mu}+w_{i}e^{\flat}+\sum_{j=1}^{m-1}w_{ij}e_{j}^{et}$\medskip{}
\item $\widetilde{e}^{\sharp}=e^{\sharp}+ue^{\flat}+\sum_{j=1}^{m-1}u_{j}e_{j}^{et}$\medskip{}
\item $\widetilde{f}_{i}^{\mu}=f_{i}^{\mu}+u_{i}f^{\sharp}+\sum_{j=1}^{m-1}w_{ji}f_{j}^{et}$\medskip{}
\item $\widetilde{f}^{\flat}=f^{\flat}+uf^{\sharp}+\sum_{j=1}^{m-1}w_{j}f_{j}^{et}.$
\end{itemize}

\vspace{0.1cm}

\noindent
The $m^{2}$ quantities $w_{i},w_{ij},u,u_{i}$ then form a system
of local (infinitesimal) parameters at $x$. The matrix of $V_{\mathcal{Q}}$
in the bases $\{\widetilde{f}^{\flat},\widetilde{f}_{i}^{\mu}\}$
of $\mathcal{Q}$ and $\{(\widetilde{e}^{\sharp})^{(p)},(\widetilde{e}_{i}^{\mu})^{(p)}\}$
of $\mathcal{P}^{(p)}$ is
\[
\left(\begin{array}{cccc}
u & u_{1} & \dots & u_{m-1}\\
 & 1\\
 &  & \ddots\\
 &  &  & 1
\end{array}\right).
\]
The infinitesimal equation of $S^{\mathrm{ao}}\cap\mathrm{Spec}(R)$
is $u=0$. As before we compute the Kodaira-Spencer homomorphism and
find out that
\[
\mathrm{KS}(e^{\sharp})=du\wedge e^{\flat}+\sum_{j=1}^{m-1}du_{j}\wedge e_{j}^{et},
\]
which means that
\[
\mathrm{KS}(e^{\sharp}\otimes\mathcal{Q}|_{x})=\mathrm{Span}_{k}\{du,du_{j}\}\subset\Omega_{S}|_{x}.
\]
This implies that $\mathrm{KS}^{-1}(du)\in e^{\sharp}\otimes\mathcal{Q}|_{x}.$
However, $V_{\mathcal{P}}$ is expressible in the same bases as above
by the matrix
\[
\left(\begin{array}{cccc}
u & w_{1} & \dots & w_{m-1}\\
 & 1\\
 &  & \ddots\\
 &  &  & 1
\end{array}\right),
\]
which means that $\ker(V_{\mathcal{P}})$ is 1-dimensional at $x$,
and spanned by $e^{\sharp}.$ Thus if $u=0$ is a local equation of
$S^{\mathrm{ao}}$, $\psi(du)$ vanishes along $S^{\mathrm{ao}}.$
This yields Theorem \ref{thm: =00005Cpsi(du)_has_a_zero} when $m=n.$

\medskip{}

Let $f\in H^{0}(S,\mathcal{L}^{k}).$ Then $\Theta(f)$ is a section
of $\mathcal{L}^{k}\otimes\mathcal{Q}^{(p)}\otimes\mathcal{Q}$ over
$S^{\mathrm{ord}}$ and we have to show that it extends holomorphically
to $S$. Since $S$ is non-singular, as in the case $n>m,$ it is
enough to show that it extends holomorphically to $S'=S^{\mathrm{ord}}\cup S^{\mathrm{ao}}$.

Let $\tau':T'\to S'$ be the intermediate Igusa variety constructed
above. Let $a$ be, as before, the tautological $p^{2}-1$ root of
$h$ over $T'$; it vanishes to order 1 along $T^{\mathrm{ao}}=\tau'^{-1}(S^{\mathrm{ao}}).$
Over $T$ (the pre-image of $S^{\mathrm{ord}}$), where $a$ does
not vanish, we can write the trivialization $\varepsilon$ of $\mathcal{L}$
as $f\mapsto f/a^{k}$. This introduces a pole of order $k$, at most,
along $T^{\mathrm{ao}}.$ Let $y\in T^{\mathrm{ao}}$ be a geometric
point and $x=\tau'(y).$ Let $\zeta$ be the $p+1$ root of $1$ such
that $y\in T'_{\zeta}.$ Let $u,v_{i}$ be formal parameters at $x$
and $w,v_{i}$ formal parameters at $y$ so that $u=0$ is a local
equation of $S^{\mathrm{ao}}$ and
\[
u=w^{p-1}.
\]
Let
\[
f/a^{k}=\sum_{r=-k}^{\infty}g_{r}(v)w^{r}
\]
be the Taylor expansion of $f/a^{k}$ in $\mathcal{\widehat{O}}_{T',y}$,
where the $g_{r}(v)$ are power series in the $v_{i}$. Note that
$du=d(w^{p-1})=-w^{p-2}dw$. Thus, similarly to the case $n>m$
\begin{equation*}
\begin{split}
d(f/a^{k})& =\sum_{r=-k}^{\infty}w^{r}dg_{r}(v)+\sum_{r=-k}^{\infty}rg_{r}(v)w^{r-1}dw
\\
& =\sum_{r=-k}^{\infty}w^{r}dg_{r}(v)-\sum_{r=-k}^{\infty}rg_{r}(v)w^{r-(p-1)}du
\\&=\sum_{r=-k}^{\infty}w^{r}dg_{r}(v)-\sum_{r=-k}^{\infty}rg_{r}(v)w^{r}u^{-1}du.
\end{split}
\end{equation*}
We conclude the proof as in the case $n>m.$
\end{proof}

\section{\label{sec:Theta-cycles}Theta cycles}

For the group $GL_{2}$, the application of the theta operator to
mod $p$ modular forms was linked to twisting Galois representations
by the cyclotomic character (see \cite{Se1} over $\mathbb{Q}$ and
\cite{A-G} over a totally real base field). The variation of the
weight filtration upon iteration of $\Theta$ was of much interest
in this context. While the connection to Galois representations in
the unitary case requires further study, our goal here is to present
a similar behavior on the level of $q$-expansions. We consider only
signature $(n,1)$, $n>1$, as signature $(1,1)$ is essentially the
case of modular curves.

In this section, let $S$ be a \emph{connected component} of the special
fiber of a unitary Shimura variety of signature $(n,1)$, so that
$\mathcal{P}_{\mu}\otimes\mathcal{Q}\simeq\mathcal{Q}^{(p)}\otimes\mathcal{Q}\simeq\mathcal{L}^{p+1}.$
The theta operator maps $\mathcal{L}^{k}$ to $\mathcal{L}^{k+p+1}$
and may be iterated. The index set $\check{H}$ for the FJ expansions
at a given level and a given rank-$1$ cusp may be identified with
$\mathbb{Z}$ so that $\check{H}^{+}$ is identified with the non-negative
integers. The effect of $\Theta$ on FJ expansions is
\begin{equation}
\Theta\left(\sum_{n\ge0}a(n)\cdot q^{n}\right)=\sum_{n\ge0}a(n)n\cdot q^{n}.\label{eq:Theta_FJ_(n,1)}
\end{equation}

Given the $q$-expansion principle and the irreducibility of the Igusa
variety $T_{1,1}^{1}$ (see Proposition \ref{prop:Igusa formal parameters}),
the proofs of the following results are verbatim as for signature
$(2,1),$ see \cite{dS-G1} $\S$$\S$3.1-3.3.
\begin{lem}
Let $\xi$ be a rank $1$ cusp on $S^{*}.$ Let $\ell$ be the non-zero
section of $\mathcal{L}$ used to trivialize $\mathcal{L}$ at a formal
neighborhood of $\xi$ as before. Consider the homomorphism
\[
FJ_{\xi}:\bigoplus_{k=0}^{\infty}H^{0}(S,\mathcal{L}^{k})\to\mathcal{FJ}_{\xi},
\]
where $\mathcal{FJ}_{\xi}$ is as in Proposition \ref{prop:FJ_expansions_functions}
and where we have identified formal sections of $\mathcal{L}^{k}$
near $\xi$ with elements of $\widehat{\mathcal{O}}_{S^{*},\xi}$
by dividing the sections by $\ell^{k}$. Then the kernel of $FJ_{\xi}$
is given by the ideal
\[
\ker(FJ_{\xi})=(h-1),
\]
where $h$ is the Hasse invariant.
\end{lem}

Given an element $f=f(q)\in\mathcal{FJ}_{\xi}$ of the form $FJ_{\xi}(g),$
$g\in H^{0}(S,\mathcal{L}^{k}),$ we denote by $\omega(f)$ the minimal
$k\ge0$ for which there exists such a $g.$ We call $\omega(f)$
the \emph{filtration }of $f$. By the previous Lemma, if $f$ arises
from $g$ of weight $k$ then $\omega(f)\equiv k\,\mod(p^{2}-1).$
\begin{prop}
Let $f\in H^{0}(S,\mathcal{L}^{k})$ be in the image of $\Theta,$
i.e. $f=\Theta(g).$

(i) We have $\Theta^{p-1}(f)=fh$ where $h$ is the Hasse invariant.

(ii) The sequence $\omega(\Theta^{i}(f))$ $i=0,1,2,\dots,p-1$ increases
by $p+1$ at each step, except for a single $i=i_{0}(f)<p-1$ for
which $\omega(\Theta^{i+1}(f))=\omega(\Theta^{i}(f))-p^{2}+p+2.$ 
\end{prop}

The combinatorics of weights has some peculiarities not present in
the case of elliptic modular forms, see \cite{dS-G1}.

\end{document}